\documentclass[12pt,oneside]{amsart}

\usepackage[utf8]{inputenc}
\usepackage[american]{babel}
\usepackage{amsfonts,amsmath,amssymb,amsthm,amsxtra,mathtools}
\usepackage{geometry}
\usepackage{graphicx}
\usepackage{enumitem}
\usepackage{tikz}

\usepackage{csquotes} % Removes one warning of biblatex
\usepackage[style=alphabetic, backend=biber, sorting=nyt, doi=false,
isbn=false, url=false, hyperref, maxnames=20,
maxalphanames=3]{biblatex}

\newbibmacro{string+doiurl}[1]{%
  \iffieldundef{doi}{%
    \iffieldundef{url}{%
          #1%
    }{%
      \href{\thefield{url}}{#1}%
    }%
  }{%
    \href{https://dx.doi.org/\thefield{doi}}{#1}%
  }%
}
\DeclareFieldFormat{title}{\usebibmacro{string+doiurl}{\mkbibemph{#1}}}
\DeclareFieldFormat[article,incollection,inproceedings,unpublished,misc,book]{title}%
    {\usebibmacro{string+doiurl}{\mkbibquote{#1}}}
\DeclareFieldFormat[online]{title}{\mkbibquote{#1}}

\DeclareFieldFormat{titlecase}{\MakeTitleCase{#1}}
\newrobustcmd{\MakeTitleCase}[1]{%
  \ifthenelse{\ifcurrentfield{booktitle}%
    \OR\ifcurrentfield{journaltitle}%
    \OR\ifentrytype{book}%
    \OR\ifentrytype{booklet}%%
    \OR\ifentrytype{collection}%
    \OR\ifentrytype{proceedings}}
    {#1}
    {\MakeSentenceCase{#1}}}

\renewbibmacro{in:}{%
  \ifentrytype{article}{}{\printtext{\bibstring{in}\intitlepunct}}}

\AtBeginBibliography{\small}

\usepackage{xurl}
\usepackage{xcolor}
\usepackage[pdfencoding=auto, psdextra]{hyperref}
\definecolor{mylinkcolor}{rgb}{0.5,0.0,0.0}
\definecolor{myurlcolor}{rgb}{0.0,0.0,0.7}
\hypersetup{
  colorlinks,
  urlcolor=myurlcolor,
  citecolor=myurlcolor,
  linkcolor=mylinkcolor,
  breaklinks=true,
  pdfauthor=%{Noam D.~Elkies and Jean Kieffer}
}
\usepackage[capitalize]{cleveref}

\numberwithin{equation}{section}
\newtheorem{prop}[equation]{Proposition}
\Crefname{prop}{Proposition}{Propositions}
\newtheorem{thm}[equation]{Theorem}
\Crefname{thm}{Theorem}{Theorems}
\newtheorem{lem}[equation]{Lemma}
\Crefname{lem}{Lemma}{Lemmas}
\newtheorem{cor}[equation]{Corollary}
\newtheorem{conj}[equation]{Conjecture}
\Crefname{conj}{Conjecture}{Conjectures}
\theoremstyle{definition}
\newtheorem{defn}[equation]{Definition}

\newtheorem{rem}[equation]{Remark}
\newtheorem{algorithm}[equation]{Algorithm}

\newcommand{\algoinput}[1]{~\\ \noindent\textbf{Input:} #1}
\newcommand{\algooutput}[1]{~\\ \noindent\textbf{Output:} #1}

\newcommand{\Z}{\mathbb{Z}}
\newcommand{\R}{\mathbb{R}}
\newcommand{\red}{\mathcal{R}}
\newcommand{\Crv}{\mathcal{C}}
\renewcommand{\O}{\mathcal{O}}
\newcommand{\Q}{\mathbb{Q}}
\newcommand{\Qbar}{\overline{\Q}}
\newcommand{\C}{\mathbb{C}}
\newcommand{\F}{\mathbb{F}}

\newcommand{\Half}{\mathcal{H}}
\newcommand{\eps}{\varepsilon}
\renewcommand{\phi}{\varphi}
\newcommand{\placeholder}{\,\cdot\,}
\newcommand{\eld}[3]{\ensuremath{\mathcal{E}(#1, #2\,; #3)}}
\newcommand{\tailsum}[4]{\ensuremath{\Omega(#1,#2,#3\,;#4)}}
\newcommand{\embedsp}[2]{\iota_{#1;\,#2}}

\DeclareMathOperator{\im}{Im}
\DeclareMathOperator{\re}{Re}
\DeclareMathOperator{\Sp}{Sp}
\DeclareMathOperator{\GSp}{GSp}
\DeclareMathOperator{\SL}{SL}
\DeclareMathOperator{\GL}{GL}
\DeclareMathOperator{\Mp}{Mp}
\DeclareMathOperator{\Mat}{Mat}
\DeclareMathOperator{\TrigSp}{TrigSp}
\DeclareMathOperator{\DiagSp}{DiagSp}

\DeclareMathOperator{\Dist}{Dist}
\DeclareMathOperator{\Otilde}{\widetilde{\mathit{O}}}

\DeclareMathOperator{\E}{\tt E}
\DeclareMathOperator{\eldF}{\tt F}

\DeclareMathOperator{\thetatilde}{\rlap{\raisebox{-0.1ex}{$\widetilde{\phantom{\,t}}$}}\theta}
\DeclareMathOperator{\rdeg}{rdeg}

\DeclareMathOperator{\e}{\mathbf{e}}
\DeclareMathOperator{\Mul}{\mathbf{M}}
\DeclareMathOperator{\rank}{rank}
\DeclareMathOperator{\Perm}{Perm}

\DeclareMathOperator{\Gal}{Gal}
\DeclareMathOperator{\End}{End}

\DeclarePairedDelimiter{\abs}{|}{|}
\DeclarePairedDelimiter{\paren}{(}{)}
\DeclarePairedDelimiter{\braces}{\{}{\}}
\DeclarePairedDelimiter{\norm}{\|}{\|}
\DeclarePairedDelimiter{\set}{\{}{\}}

\DeclarePairedDelimiter{\ceil}{\lceil}{\rceil}
\DeclarePairedDelimiter{\floor}{\lfloor}{\rfloor}

\title{Fast evaluation of Riemann theta functions in any dimension}

\thanks{This research was partially supported by the Simons Foundation grant
  no.~550031 (to Noam~D.~Elkies).}


\author{Noam D.~Elkies}
\address{Department of Mathematics, Harvard
  University, 1 Oxford Street, Cambridge, MA 02138, USA}
\email{elkies@math.harvard.edu}

\author{Jean Kieffer}
\address{Université de Lorraine, CNRS, Inria, LORIA, 54000 Nancy, France}
\email{jean.kieffer@loria.fr}

\begin{document}

\begin{abstract}
  We describe an algorithm to numerically evaluate Riemann theta functions in
  any dimension in quasi-linear time in terms of the required precision,
  uniformly on reduced input. This algorithm is implemented in the FLINT number
  theory library and vastly outperforms existing software. As an application,
  we evaluate the theta constants attached to certain special abelian varieties
  of dimension~$6$ to construct explicit polynomials of degree~$65$ over~$\Q$
  with conjectural Galois group~$\SL_2(\F_{64})$.
\end{abstract}

\maketitle

\section{Introduction}
\label{sec:intro}

\subsection{Riemann theta functions}
\label{subsec:overview}

For every $g\in \Z_{\geq 1}$, denote by $\Half_g$ the Siegel upper half-space
in dimension (or genus)~$g$, which is the space of symmetric $g\times g$
complex matrices~$\tau$ such that $\im(\tau)$ is positive definite. The
\emph{Riemann theta function} is defined on~$\C^g\times \Half_g$ by the
convergent series
\begin{equation}
  \label{eq:theta-00}
  \theta_{0,0}(z,\tau) = \sum_{n\in \Z^g} \e(n^T\tau n + 2n^T z).
\end{equation}
Here and throughout, we write $\e(x) := \exp(\pi i x)$ for every $x\in \C$,
consider vectors as column vectors, and use $^T$ to denote transposition. More
generally, we are interested in \emph{theta functions with characteristics} (of
level 4). For $a,b \in\{0,1\}^g$, we define
\begin{equation}
  \label{eq:theta-ab}
  \begin{aligned}
    \theta_{a,b}(z,\tau)
    &= \sum_{n\in \Z^g + \tfrac a2} \e\paren[\big]{n^T\tau n + 2 n^T (z + \tfrac b2)}\\
    &= \e\paren[\big]{\tfrac14 a^T\tau a + a^T (z + \tfrac b2)}
      \,\theta_{0,0}\paren[\big]{z + \tau\tfrac a2 + \tfrac b2, \tau}
  \end{aligned}
\end{equation}
and refer to $(a,b)$ as the \emph{characteristic} of $\theta_{a,b}$.

Theta functions are ubiquitous in mathematics and physics: we refer to the
introduction of~\cite{agostiniComputingThetaFunctions2021} for a varied list of
examples. The following uses of theta functions in number theory and arithmetic
geometry are of particular interest to us. Fixing a principally polarized
abelian variety $A$ of dimension~$g$ over~$\C$, one can always find a period
matrix $\tau\in \Half_g$ such that $A$ is isomorphic to the complex torus
$\C^g/ (\Z^g + \tau\Z^g)$. The theta functions
$\theta_{a,b}(\placeholder,\tau)$ then provide a projective embedding of~$A$ by
Lefschetz's
theorem~\cite[Thm.~4.5.1]{birkenhakeComplexAbelianVarieties2004}. On the other
hand, fixing $z=0$ and letting~$\tau$ vary, the \emph{theta constants}
$\theta_{a,b}(0,\placeholder)$ are universal Siegel modular forms, in the sense
that every Siegel modular form of level~$1$ lies in the integral closure of the
ring generated by theta constants~\cite{igusaGradedRingThetaconstants1964};
moreover, this ring is already integrally closed when $g\leq 3$
\cite{freitagVarietyAssociatedRing2019}. Theta constants can thus be used to
construct coordinates on many moduli spaces of abelian varieties
over~$\C$. These coordinates then remain meaningful over other bases thanks to
Mumford's theory of algebraic theta
functions~\cite{mumfordEquationsDefiningAbelian1966}.

In this paper, we are interested in the numerical evaluation of Riemann theta
functions at a given point~$(z,\tau)$ to some chosen precision~$N$, i.e.~up to
an error of at most~$2^{-N}$. Evaluating theta functions is a fundamental
building block in algorithmic applications, which include class
polynomials~\cite{engeComplexityClassPolynomial2009,
  engeComputingClassPolynomials2014}, the CM
method~\cite{kilicerPlaneQuartics$mathbbQ$2018}, isogenies between abelian
varieties~\cite{engeComputingModularPolynomials2009,vanbommelComputingIsogenyClasses2024},
the Schottky and Torelli
problems~\cite{agostiniNumericalReconstructionCurves2022}, and the explicit
inverse Galois problem, for which we refer
to~\cite{vanbommel17T7GaloisGroup2024} and the final section of this paper. For
such number-theoretic applications, $g$ is typically small (say $g\leq 8$), but
the desired precision~$N$ may range from a few hundred to millions of bits. For
larger values of~$g$, obtaining a decent approximation of theta values (say
$N=32$) is already a challenge. Evaluating partial derivatives of theta
functions is also of interest.

\subsection{State of the art on evaluation algorithms}
\label{subsec:intro-comparison}

A direct strategy to evaluate (partial derivatives of) theta functions is to
compute a partial sum of the series~\eqref{eq:theta-ab}. We refer to this
family of algorithms as the \emph{summation algorithms}. Before applying
summation algorithms, it is essential to \emph{reduce the arguments}~$(z,\tau)$
under the action of~$\Sp_{2g}(\Z)$ and translations of~$z$, taking advantage of
the periodicity properties of theta functions. This reduction ensures that the
series~\eqref{eq:theta-ab} is quickly approximated by partial sums over small
vectors~$n\in \Z^g$. Summation algorithms and argument reduction have attracted
a lot of interest and have been implemented multiple times, including in
Maple~\cite{deconinckComputingRiemannTheta2004},
Matlab~\cite{frauendienerEfficientComputationMultidimensional2019},
Julia~\cite{agostiniComputingThetaFunctions2021},
Magma~\cite{bosmaMagmaAlgebraSystem1997},
SageMath~\cite{swierczewskiComputingRiemannTheta2016,
  bruinRiemannThetaSageMathPackage2021}, and for $g=1$, in the FLINT number
theory library~\cite{engeShortAdditionSequences2018}.

Among these implementations, only Magma,
SageMath~\cite{bruinRiemannThetaSageMathPackage2021} and FLINT support
arbitrary precisions~$N$, a necessary feature for applications in number
theory. The complexity of summation algorithms for a fixed~$g$ in terms of~$N$,
on reduced arguments, is $\Otilde(N^{1 + g/2})$~\cite[Thms.~3
and~8]{deconinckComputingRiemannTheta2004}. Already for $g=2$, this algorithm
is quadratic, and unsuitable for the large values of~$N$ mentioned above.

An asymptotically fast, but heuristic, approach to evaluating theta constants
was pioneered in~\cite{dupontMoyenneArithmeticogeometriqueSuites2006,
  dupontFastEvaluationModular2011} for $g\leq 2$, then extended to general
theta functions in~\cite{labrandeComputingThetaFunctions2016}, mainly for
$g\leq 3$. Their crucial observation is that the inverse operation, namely
going from theta values to the corresponding matrix in~$\Half_g$, corresponds
to a higher-dimensional version of Gauss's arithmetic-geometric mean (AGM) and
can be performed in quasi-linear time in~$N$ for fixed~$(z,\tau)$. One can
then set up a Newton scheme around the AGM to evaluate theta functions in
quasi-linear time as well.

The Newton approach is particularly successful in the case of genus~$2$ theta
constants~\cite{engeComputingClassPolynomials2014,
  vanbommelComputingIsogenyClasses2024}, where one can reformulate it as a
provably correct and uniform
algorithm~\cite{kiefferCertifiedEwtonSchemes2022}. This algorithm has been
implemented in~C~\cite{engeCMHComputationGenus2014} (see also version 0.1
of~\cite{kiefferHDMELibraryEvaluation2022}). In higher dimensions, the Newton
approach is trickier: for all we know, there might exist
matrices~$\tau\in \Half_3$ where it simply fails.

\subsection{Our contributions}
\label{subsec:intro-main}

The first main contribution of this paper is to describe a uniform,
quasi-linear time algorithm to evaluate Riemann theta functions at reduced
arguments, in any dimension~$g$, with provably correct error bounds.

To state this as a theorem, we denote by~$\Mul(N)$ the time needed to multiply
$N$-bit integers, which is $O(N\log N)$
asymptotically~\cite{harveyIntegerMultiplicationTime2021}. For
$(z,\tau)\in \C^g\times \Half_g$ and each characteristic $(a,b)$, we introduce
the modified theta value
\begin{equation}
  \label{eq:thetatilde}
  \thetatilde_{a,b}(z,\tau) = \exp \paren[\big]{-\pi y^T Y^{-1} y} \theta_{a,b}(z,\tau).
\end{equation}
where~$y\in \R^g$ and~$Y\in \Mat_{g\times g}(\R)$ denote the imaginary parts
of~$z$ and~$\tau$ respectively.  The modified theta
functions~$\thetatilde_{a,b}$ are not holomorphic, but for any fixed~$\tau$,
the function $\abs[\big]{\thetatilde_{a,b}(\cdot, \tau)}$ is doubly periodic
with respect to the lattice~$\Z^g + \tau\Z^g$, and hence uniformly bounded
on~$\C^g$.

\begin{thm}
  \label{thm:main-intro}
  Given~$g\in \Z_{\geq 1}$,~$N\in \Z_{\geq 2}$, and a reduced point
  $(z,\tau)\in \C^g\times \Half_g$, one can compute $\thetatilde_{a,b}(z,\tau)$
  for all characteristics $a,b\in \{0,1\}^g$ up to an error of at most~$2^{-N}$
  in time~$2^{O(g\log^2 g)} \Mul(N)\log N$.
\end{thm}

The constant in Landau's~$O$ notation is absolute; we could in fact make it
explicit, but its value would be too pessimistic compared to reality. By a
\emph{reduced} $(z,\tau)$, we mean that~$(z,\tau)\in \red_g$, where~$\red_g$ is
the reduced domain to be defined later in \cref{def:reduced}. To make
\cref{thm:main-intro} fully precise, we must also explain how we encode complex
numbers in bits. We do this using ball arithmetic: dyadic complex numbers
(elements of $2^{-p}\Z[i]$ for some~$p\in \Z$) are represented exactly, while a
general complex number $x$ is represented as a ball, i.e.~as a pair
$(x_0,\eps)$ (the \emph{center} and \emph{radius}) where both $x_0$ and~$\eps$
are dyadic, $\eps$ is either real nonnegative or~$+\infty$, and the inequality
$\abs{x-x_0}\leq \eps$ is guaranteed to hold. This is virtually equivalent to
the implementation of complex numbers in the FLINT
library~\cite{theflintteamFLINTFastLibrary2024}, which incorporates
Arb~\cite{johanssonArbEfficientArbitraryprecision2017} since version 3.0, where
non-exact complex numbers $x$ are represented as rectangles. In
\cref{thm:main-intro}, we mean that the input consists of exact values
($\eps = 0$), and that the output consists of complex balls of radius
$\eps\leq 2^{-N}$. In other words, the absolute binary precision of the output
is at least~$N$.

We sketch the main ideas behind the quasi-linear evaluation algorithm
in~§\ref{subsec:intro-overview} below, and describe it in full details in
\cref{sec:ql} of the paper.

Our second main contribution is a carefully tested and optimized
\emph{implementation} of Riemann theta functions in FLINT. The associated
documentation is available~at
\url{https://flintlib.org/doc/acb_theta.html}. Our aim was to implement the
algorithm of \cref{thm:main-intro} together with all the related user
functionality, while maintaining provably correct error bounds and the best
practical performances we could achieve. This often builds on new mathematical
work beyond \cref{thm:main-intro} (and beyond FLINT's ball arithmetic and
general high performance) that we also present in this paper:
\begin{enumerate}
\item \label{item:intro-reduction} In \cref{sec:reduction} (up
  to~§\ref{subsec:inexact-reduction}), we describe how any
  $(z,\tau)\in \C\times \Half_g$ can be efficiently reduced to some
  $(z',\tau')\in \red_g$ (or at least close to~$\red_g$, taking non-exactness
  and rounding errors into account). Our attempts at obtaining the ``best
  algorithmically possible'' reduced domain~$\red_g$ led us to revisit and
  strengthen existing reduction algorithms. Our choices there are sometimes
  arbitrary, and further analysis would be welcome to assess how beneficial
  they are.
\item \label{item:intro-transf} Another challenge is to correctly pin down the
  correct 8th root of unity that appears in the transformation formula relating
  theta values at~$(z,\tau)$ and~$(z',\tau')$. We also address this issue in
  \cref{sec:reduction} by decomposing any given matrix in~$\Sp_{2g}(\Z)$ as a
  product of ``elementary'' matrices for which this root of unity can be
  efficiently computed.
\item As we will see below, summation algorithms appear as key subroutines in
  the quasi-linear evaluation algorithm. In our implementation, and for the
  proof of \cref{thm:main-intro}, we thus needed to study summation algorithms
  in a new setting, where we care about provably correct error bounds, all
  rounding errors, the dependency on~$(z,\tau)$, and performance
  simultaneously. In particular, we provide a new upper bound on the tails of
  theta series, as~\cite[Lemma 2]{deconinckComputingRiemannTheta2004} was not
  precise enough for us. This is the subject of \cref{sec:summation}.
\item \label{item:intro-inexact} In practice, the user will want to evaluate
  the actual theta functions~$\theta_{a,b}$, not the modified
  ones~$\thetatilde_{a,b}$, and their input~$(z,\tau)$ might be inexact
  (e.g.~if reduction was performed). We handle this in FLINT by using
  \cref{thm:main-intro} on the center of~$(z,\tau)$, multiplying
  by~$\exp(\pi y^T Y^{-1} y)$, then using upper bounds on derivatives of theta
  functions obtained in~§\ref{sec:summation} to propagate errors.
\item Another common task is to evaluate partial derivatives of theta functions
  with characteristics to some absolute precision~$N$. In~§\ref{subsec:deriv},
  we show that this can be done in quasi-linear time in~$N$ at an exact,
  reduced point~$(z,\tau)$ using \cref{thm:main-intro} and finite differences.
  The complexity of this algorithm is made explicit in \cref{thm:deriv} in
  terms of~$(z,\tau)$ (as it is not uniform). Inexact input is then handled as
  in~\eqref{item:intro-inexact} above.
\end{enumerate}

Experimentally, our implementation outperforms existing software by a large
margin on a wide range of precisions: see \cref{sec:cmp} for (hopefully
reproducible) practical runtimes. We hope that our code can become a reference
for the evaluation of Riemann theta functions, and that it will be directly
accessible from SageMath~\cite{thesagedevelopersSageMathSageMathematics} and
Python-FLINT~\cite{johanssonPythonFLINTPythonBindings2025} in the future. There
is still room for improvement, of course, as we will observe several times in
the main text. In particular, see~§\ref{subsec:ql-compare} for how our
implementation could be improved when~$g$ is larger and~$N$ is small.

We note in passing that analyzing the complexity and precision losses to reduce
the input in~\eqref{item:intro-reduction} and compute the correct 8th root of
unity in~\eqref{item:intro-transf} remain unknown at the moment. Such an
analysis would provide an upper bound on the complexity of evaluating theta
functions at any exact point in~$\C^g\times \Half_g$ in terms of the desired
precision. This upper bound would not be uniform in~$(z,\tau)$.

Finally, our third main contribution lies in number theory. In
\cref{sec:galois}, we apply our algorithm to evaluate theta constants attached
to certain special abelian varieties of dimension~6, and from those, construct
explicit degree~65 polynomials over~$\Q$ with (conjectural) Galois group
$\SL_2(\F_{64})$. This is in the spirit of similar explicit inverse Galois
computations in~\cite{vanbommel17T7GaloisGroup2024}, albeit in a conceptually
simpler setting. Such a computation would have been hopeless without
\cref{thm:main-intro}.

\subsection{Overview of the fast algorithm}
\label{subsec:intro-overview}

Before describing our quasi-linear algorithm, it is convenient to review the
summation strategy and explain where its complexity comes from,
following~\cite[§3]{deconinckComputingRiemannTheta2004}.

Fix a precision $N\in \Z_{\geq 2}$, a point $(z,\tau)\in \C^g\times \Half_g$,
and a theta characteristic~$(a,b)$. We wish to
evaluate~$\thetatilde_{a,b}(z,\tau)$ up to an error of at most~$2^{-N}$ using
the series~\eqref{eq:theta-ab}. Write $y = \im(z)$ and $Y = \im(\tau)$.
Then for each~$n\in \Z^g + \tfrac a2$, we have
\begin{equation}
  \label{eq:term-modulus}
  \begin{aligned}
    \abs[\big]{\e\paren[\big]{n^T\tau n + 2n^T (z + \tfrac b2)}} &= \exp(-\pi n^T Y n - 2\pi n^T y)\\
    &= \exp(\pi y^T Y^{-1} y) \cdot \exp \paren[\big]{- \norm{n + Y^{-1}y}_\tau^2}
  \end{aligned}
\end{equation}
where $\norm{\placeholder}_\tau$ denotes the Euclidean norm attached
to~$\pi Y$, defined by
\begin{equation}
  \label{eq:def-norm}
  \norm{x}_\tau := \sqrt{\pi x^T \im(\tau) x}
\end{equation}
for every~$x\in\R^g$. The leading factor $\exp(\pi y^T Y^{-1} y)$ is
independent of~$n$, and cancels out given the definition
of~$\thetatilde_{a,b}(z,\tau)$.

We introduce the following notation: for every $v\in \R^g$ and
$R\in\R_{\geq 0}$, we denote by
\begin{equation}
  \label{eq:def-ellipsoid}
  \eld{v}{R}{\tau} := \braces[\big]{x\in \R^g: \norm{x - v}_\tau < R}
\end{equation}
the open ball for $\norm{\placeholder}_\tau$ centered at~$v$ of radius~$R$,
which is an ellipsoid in~$\R^g$.  By~\eqref{eq:term-modulus}, the lattice
points $n\in \Z^g + \smash{\tfrac a2}$ indexing terms whose absolute value is
greater than~$2^{-N}$ are exactly those lying inside
$\eld{-Y^{-1}y}{\sqrt{N\log 2}}{\tau}$. Because the theta
series~\eqref{eq:theta-ab} converges rapidly, the absolute value of its tail
(the sum of all terms outside the previous ellipsoid) should not be much larger
than $2^{-N}$. \Cref{thm:summation-bound-new} provides an explicit upper bound
on the tail that often improves on~\cite[Lemma
2]{deconinckComputingRiemannTheta2004}. Consequently, the complexity of the
summation algorithm directly depends on the number of lattice points in an
ellipsoid $\eld{v}{R}{\tau}$ where $R = \Theta(\sqrt{N})$.

If~$(z,\tau)\in \red_g$, then the eigenvalues of~$\im(\tau)$ are bounded below
by an explicit constant (depending on $g$), yielding the claimed uniform
complexity $O_g\paren[\big]{N^{g/2}\Mul(N)}$. However, in the favorable case
where the eigenvalues of~$\im(\tau)$ are all in~$\Omega_g(N)$, the above
ellipsoid contains $O_g(1)$ lattice points. In that case, the cost of the
summation algorithm decreases to $O_g(\Mul(N)\log N)$ binary operations, the
same asymptotic cost as evaluating an exponential to
precision~$N$~\cite{borweinArithmeticgeometricMeanFast1984},
\cite[§3.5]{dupontMoyenneArithmeticogeometriqueSuites2006}.

The algorithm of \cref{thm:main-intro} combines this crucial remark with
duplication formulas relating theta values at~$\tau$
and~$2\tau$. If~$h\simeq\log_2(N)$, then applying such formulas~$h$ times
relates theta values at~$\tau$ vith theta values at~$2^h\tau$, where summation
to precision~$N$ is efficient.

Fix~$N\in \Z_{\geq 2}$ and a reduced matrix~$\tau\in \Half_g$
(i.e.~$(0,\tau)\in \red_g$), and assume in this presentation that we wish to
evaluate the theta constants $\theta_{a,0}(0,\tau) = \thetatilde_{a,0}(0,\tau)$
up to an error of $2^{-N}$. The duplication formula we use is the following:
for every~$\tau'\in \Half_g$ and every $a\in\{0,1\}^g$, we have
\begin{equation}
  \label{eq:dupl-const}
  \theta_{a,0}(0,\tau')^2 = \sum_{a'\in \{0,1\}^g} %(-1)^{a'^T b}
  \theta_{a',0}(0, 2\tau')\, \theta_{a+a',0}(0,2\tau').
\end{equation}
Here the addition $a + a'$ can be considered modulo 2, as in the
expression~\eqref{eq:theta-ab}, translating $a$ by an element of $2\Z^g$
changes nothing.

As an aside, we note that equation~\eqref{eq:dupl-const} can be obtained as
in~\cite[Prop.~5.6]{dupontMoyenneArithmeticogeometriqueSuites2006} by applying
the transformation $\tau' \mapsto -\tau'^{-1}$ \cite[Case III
p.\,195]{mumfordTataLecturesTheta1983} to the perhaps more usual formula
\begin{equation}
  \label{eq:dupl-0b}
  \theta_{a,b}(0,2\tau')^2 = \frac{1}{2^g}\sum_{b'\in \{0,1\}^g} (-1)^{a^T b'}
  \, \theta_{0,b'}(0,\tau') \, \theta_{0,b + b'}(0, \tau').
\end{equation}
The duplication formula~\eqref{eq:dupl-0b} is the one that appears in the
Newton approach to evaluate theta
constants~\cite{dupontMoyenneArithmeticogeometriqueSuites2006,
  labrandeComputingThetaFunctions2016}, and amounts to going from one term to
the next in an AGM sequence. For us, expressing theta values at $2\tau'$ in
terms of theta values at $\tau'$ as in~\eqref{eq:dupl-0b} goes in the wrong
direction.

In the main algorithm, we apply~\eqref{eq:dupl-const} $h$ times, first for
$\tau' = 2^{h-1}\tau$, etc., down to $\tau' = \tau$.  After evaluating the
right hand side, we have to extract a square root of~$\theta_{a,0}(0,\tau')^2$
for each~$a$. This raises two issues:
\begin{enumerate}
\item \label{issue:sign} We must determine the correct sign of the square root,
  and
\item \label{issue:prec} Taking square roots implies a loss of absolute
  precision.
\end{enumerate}

One can hope to solve issue~\eqref{issue:sign} by computing low-precision
approximations of the quantities~$\theta_{a,0}(0,\tau')$ with the summation
algorithm, but this might be too costly if~$\abs{\theta_{a,0}(0,\tau')}$ is
very small. In such a case, issue~\eqref{issue:prec} is also problematic, since
we could lose as much as half of the current absolute precision when extracting
the square root. To make matters worse, for each $\tau\in \Half_g$ and each
nonzero $a\in \{0,1\}^g$, the theta constants $\theta_{a,0}(0,2^h\tau)$ do
converge quickly to zero as $h\to +\infty$.

Nevertheless, we will see that this simple algorithm works well when the
eigenvalues of~$\im(\tau)$ are small (say~$O(1)$), and when the values
$\theta_{a,0}(0,2^j\tau)$ for $a\in \{0,1\}^g$ and $0\leq j\leq h-1$ are of the
expected order of magnitude, i.e.~not much smaller than the biggest term in its
defining series~\eqref{eq:theta-ab}. Quantitatively,
\begin{equation}
  \label{eq:big-thetaa0}
  \abs[\big]{\theta_{a,0}(0,2^j\tau)} \geq \eta_{g,h} \exp\paren[\big]{-2^j \Dist_{\tau}(0, \Z^g + \tfrac a2)^2},
\end{equation}
where $\Dist_\tau$ denotes the distance associated to $\norm{\cdot}_\tau$, and
$\eta_{g,h} \in \R_{>0}$ is a reasonably large fixed value which is independent
of~$\tau$. When~\eqref{eq:big-thetaa0} holds, the correct sign
of~$\theta_{a,0}(0,2^j\tau)$ can be determined by computing a short partial sum
of the associated series. Proving that summation algorithms can compute these
partial sums within the claimed complexity bound relies on technical results
from \cref{sec:summation}.

Moreover, if~\eqref{eq:big-thetaa0} holds, the precision losses in the
duplication formula and when extracting square roots can be tightly
controlled. The key here is to incorporate the exponential factor
from~\eqref{eq:big-thetaa0} in the definition of the precision. The resulting
notion of \emph{shifted absolute precision} bears similarities with both the
relative and absolute precisions, and is the right framework to analyze our
quasi-linear algorithm.

Unfortunately, the inequalities~\eqref{eq:big-thetaa0} do not hold in general,
as some of the theta values $\theta_{a,0}(0, 2^j\tau)$ may vanish
altogether. (This cannot actually happen for theta constants if $g\leq 2$, by
\cite[Lem.~2.9]{coxArithmeticgeometricMeanGauss1984}
and~\cite[Cor.~7.7]{strengComputingIgusaClass2014}.) Further, this strategy
does not yet allow us to evaluate theta functions at a nonzero $z\in \C^g$. Our
solution here is to introduce an auxiliary vector~$t\in \R^g$ (in
fact~$t\in [0,1]^g$), and to use duplication formulas involving
$\thetatilde_{a,b}(2^jx, 2^j\tau)$ for all~$x\in \{0,t,2t, z,z+t,z+2t\}$. We
still have to extract square roots at each step, but only of theta values
corresponding to~$x\in \{t, 2t, z+t, z+2t\}$. For a good choice of~$t$, one can
then hope that analogous inequalities to~\eqref{eq:big-thetaa0} will be
satisfied at those points. We prove that this is indeed the case when~$t$ is
chosen either at random (which leads to a probabilistic algorithm for
\cref{thm:main-intro}, the one we implemented) or through a dedicated
deterministic process (which is mainly of theoretical interest).

The resulting algorithm is still not satisfactory when the eigenvalues
of~$\im(\tau)$ are large or unbalanced: in such a case, we should perform a
smaller number of duplication steps, then write dimension~$g$ theta values as
short sums of theta values in dimension strictly less than~$g$ -- a procedure
we call ``lowering the dimension''. We refer to~§\ref{subsec:dupl}, in
particular \Cref{algo:ql}, for a summary of the whole algorithm.
\Cref{thm:ql-complexity} and its proof then provide all the missing details.

\subsection{Structure of the paper}
\label{subsec:structure}

\Cref{sec:reduction}, which focuses on reduction and the transformation
formula, is largely independent from the rest of the paper. Only the definition
of the reduced domain~$\red_g$ in \cref{def:reduced} and the properties of
reduced arguments listed in~§\ref{subsec:properties} are relevant for later
sections. \Cref{sec:summation}, on summation algorithms, and \cref{sec:ql},
where we describe our fast algorithm and prove \cref{thm:main-intro}, are the
main technical body of the paper. Some readers may prefer to browse quickly
through \cref{sec:summation}, then read \cref{sec:ql}, referring back to
technical results as needed. \Cref{sec:cmp} compares our implementation with
other software, and \cref{sec:galois} presents the application to the inverse
Galois problem; these can again be read independently.

\subsection{Notation}
\label{subsec:notation}

The following notation will be used throughout.
\begin{itemize}
\item The letter $i$ (without indices) refers to~$\sqrt{-1}$, and $\e(x) = \exp(\pi i x)$.
\item If $(z,\tau)$ denotes a point in~$\C^g\times \Half_g$, then we write
  $y = \im(z)$,~$Y = \im(\tau)$, and $v = -Y^{-1}y$. Moreover,~$C$ always
  denotes the Cholesky matrix attached to~$\pi Y$, in other words~$C$ is the
  unique upper-triangular $g\times g$ matrix with positive diagonal
  coefficients such that $\pi Y = C^T C$. We denote the diagonal coefficients
  of~$C$ by $c_1,\ldots,c_g$.
\item We continue using $\norm{\placeholder}_\tau$ and~$\eld{v}{R}{\tau}$ as in
  this introduction (§\ref{subsec:intro-overview}). We denote the distance
  (between points and sets) associated to~$\norm{\placeholder}_\tau$
  by~$\Dist_\tau$.
\item If $x$ is a vector of length~$g$, we denote its entries by
  $x_1,\ldots, x_g$.
\item $\norm{\placeholder}_2$ and $\norm{\placeholder}_\infty$ denote the usual
  norms on~$\R^g$ for any $g\geq 1$. If~$P$ is a polynomial or a power series,
  then $\norm{P}_\infty$ refers to the maximum absolute value among all
  coefficients of~$P$.
\item If~$M$ is a matrix, then we denote the maximal absolute value of an entry
  of~$M$ by~$\abs{M}$, while $\norm{M}_2$ and~$\norm{M}_\infty$ refer to the
  induced norms with respect to $\norm{\placeholder}_2$ and
  $\norm{\placeholder}_\infty$ on~$\C^g$.
\item If~$\nu = (\nu_1,\ldots,\nu_g)$ is a tuple of nonnegative integers, we
  write $\abs{\nu} := \nu_1+\ldots+\nu_g$, and use the abbreviation
  \begin{displaymath}
    \partial^\nu\theta_{a,b} := \frac{\partial^{\abs{\nu}}\theta_{a,b}}{\partial^{\nu_1} z_1\cdots \partial^{\nu_g} z_g}.
  \end{displaymath}
\end{itemize}

\section{Argument reduction and transformation formulas}
\label{sec:reduction}

\subsection{The reduction algorithm}
\label{subsec:reduction-overview}

Let~$g\in \Z_{\geq 1}$, and write
\begin{equation}
 \label{eq:Jg}
J_g = \begin{pmatrix} 0 & I_g \\ -I_g & 0 \end{pmatrix},
\end{equation}
where~$I_g$ denotes the $g\times g$ identity matrix. The symplectic group
\begin{equation}
  \label{eq:Sp}
  \Sp_{2g}(\Z) = \braces[\big]{\sigma\in \GL_{2g}(\Z): \sigma^T J_g\sigma = J_g}
\end{equation}
acts on~$\C^g\times \Half_g$ as follows: for each
$(z,\tau)\in \C^g\times \Half_g$ and each
$\sigma = \smash{\begin{pmatrix} \alpha & \beta \\ \gamma & \delta \end{pmatrix}}
\in \Sp_{2g}(\Z)$ where $\alpha,\beta,\gamma,\delta$ are $g\times g$ blocks, we have
\begin{equation}
  \label{eq:action}
  \sigma \cdot (z,\tau) = \paren[\big]{(\gamma\tau + \delta)^{-t}z,
    (\alpha\tau + \beta)(\gamma\tau + \delta)^{-1}}.
\end{equation}
Considering $\tau$ only, we get an action of $\Sp_{2g}(\Z)$ on~$\Half_g$ that
we denote by $\tau\mapsto \sigma\tau$. For any~$\sigma\in \Sp_{2g}(\Z)$
and~$(z,\tau)\in \C^g\times \Half_g$, theta values at~$(z,\tau)$
and~$\sigma\cdot(z,\tau)$ are related by a transformation formula that we study
in~§\ref{subsec:transf} below.

As the analysis in~§\ref{subsec:intro-overview} suggests, before evaluating
theta functions at a given point $(z,\tau) \in \C^g\times \Half_g$, it is
essential to reduce $\tau$, i.e.~to compute a matrix $\sigma\in \Sp_{2g}(\Z)$
such that the smallest eigenvalue of~$\im(\sigma\tau)$ is as large as
possible. This will ensure that~$\norm{\placeholder}_{\sigma\tau}$ is as large
as possible compared to the usual Euclidean norm.

This reduction process has been studied, both experimentally and theoretically,
in several works including~\cite{deconinckComputingRiemannTheta2004,
  kilicerPlaneQuartics$mathbbQ$2018,
  frauendienerEfficientComputationMultidimensional2019,
  agostiniComputingThetaFunctions2021}. One difficulty is that an optimal
reduction algorithm -- in the sense that~$\sigma\tau$ will belong to the
fundamental domain for the action of~$\Sp_{2g}(\Z)$ on~$\Half_g$ given
in~\cite[Def.~1 p.\,29]{klingenIntroductoryLecturesIegel1990} -- has only been
described for~$g\leq 2$. As soon as $g\geq 3$, the optimal reduction algorithm
is unknown, so one can conceivably try to improve on the (partial) reduction
algorithms that have been previously described, with the aim of obtaining
stronger reduction properties while keeping the costs reasonable. This is
precisely what we attempt in this section. In particular, our modified
reduction algorithm will be optimal for~$g\leq 2$, a property that is not
satisfied by the previous reduction algorithms for general~$g$.

We use the following notation. If~$S\in \Mat_{g\times g}(\Z)$ is symmetric,
resp.~$U\in \GL_g(\Z)$, we define the symplectic matrices
\begin{equation}
  \label{eq:trig-diag}
  \TrigSp(S) := \begin{pmatrix} I_g & S \\ 0 & I_g \end{pmatrix},
  \quad \text{resp.} \quad
  \DiagSp(U) := \begin{pmatrix} U & 0 \\ 0 & (U^{-1})^T \end{pmatrix}.
\end{equation}
We also define a finite list of matrices~$\Sigma_g\subset\Sp_{2g}(\Z)$ as
follows.

\begin{defn}
  \label{def:sp-embed} Let $1\leq i_1 < \cdots < i_r \leq g$ be any increasing
  collection of indices, let~$\sigma \in \Sp_{2r}(\Z)$, and
  let~$\alpha,\beta,\gamma,\delta$ be the $r\times r$ blocks of~$\sigma$. Let
  $\alpha'$ be the $g\times g$ matrix obtained by embedding $\alpha$ on the
  rows and columns numbered $i_1,\ldots,i_r$ and the identity on other rows and
  columns, in other words
  \begin{displaymath}
    \alpha'_{j,k} = \begin{cases} \alpha_{l,m} & \text{if there exist $1\leq l,m\leq r$ such that
                                                 $j = i_l$ and $k = i_m$},\\
      1 &\text{if } j = k \notin \{i_1,\ldots,i_r\},\\
      0 & \text{otherwise.}
    \end{cases}
  \end{displaymath}
  Define~$\beta',\gamma',\delta'$ similarly. Then we define
  \begin{displaymath}
    \embedsp{g}{i_1,\ldots,i_r}(\sigma)
    = \begin{pmatrix} \alpha' & \beta' \\ \gamma' & \delta'
    \end{pmatrix} \in \Sp_{2g}(\Z).
  \end{displaymath}
  The finite collection of matrices $\Sigma_g\subset \Sp_{2g}(\Z)$ is the following:
  \begin{itemize}
  \item If~$g=1$, then $\Sigma_1= \{J_1\}$.
  \item If~$g=2$, then~$\Sigma_2$ is the list of~19 matrices defining the
    boundary of the Siegel fundamental domain
    \cite{gottschlingExpliziteBestimmungRandflaechen1959}; see
    also~\cite[§2.1]{frauendienerEfficientComputationMultidimensional2019}.
  \item If~$g\geq 3$, then~$\Sigma_g$ is the list of matrices of the form
    $\embedsp{g}{j,k}(\sigma)$ where $\sigma\in \Sigma_2$ and
    $1\leq j < k\leq g$, and $\embedsp{g}{i_1,\ldots,i_r}(J_r)$ where
    $3\leq r\leq g$ and $1\leq i_1 < \cdots < i_r\leq g$.
  \end{itemize}
\end{defn}

We can now describe the reduction algorithm we use, assuming that all
operations on real and complex numbers can be performed exactly.
See~§\ref{subsec:inexact-reduction} below for a discussion of how we can adapt
the reduction algorithm to an inexact setting.

\begin{algorithm}
  \label{algo:siegel-red}
  \algoinput{$\tau\in \Half_g$.}  \algooutput{A matrix
    $\sigma\in \Sp_{2g}(\Z)$.}
  \begin{enumerate}
  \item Set $\sigma = I_{2g}$ and $\tau' = \tau$. Let~$k=0$.
  \item \label{step:lattice-red} Let~$U\in \GL_g(\Z)$ be a matrix such
    that~$U\im(\tau') U^T$ is reduced, in the sense of HKZ lattice reduction
    for Gram matrices~\cite{lagariasKorkinZolotarevBasesSuccessive1990}
    if~$k=0$, and reciprocal HKZ reduction if~$k=1$.
  \item Set
    $\sigma\gets \DiagSp(U)\sigma$ and recompute~$\tau' = \sigma\tau$.
  \item \label{step:real-red} Let~$S\in \Mat_{g\times g}(\Z)$ be a symmetric
    matrix such that~$\abs{\re(\tau') + S}\leq \tfrac 12$.
  \item Set
    $\sigma\gets \TrigSp(S)\sigma$ and recompute~$\tau' = \sigma\tau$.
  \item \label{step:det} Find a matrix~$\sigma'\in \Sigma_g$ such
    that~$\det \im(\sigma'\tau')$ is maximal, or equivalently
    $\det(\gamma\tau + \delta)$ is minimal, where~$\gamma, \delta$ denote the
    lower blocks of~$\sigma'$.
  \item \label{step:red-back} If $\det\im(\sigma'\tau') > \det\im(\tau')$, or
    equivalently $\det(\gamma\tau + \delta) < 1$, then
    set~$\sigma \gets \sigma'\sigma$, recompute $\tau' = \sigma\tau$, and go
    back to step~\eqref{step:lattice-red} with~$k=0$. Otherwise, if~$k=0$, then
    go back to step~\eqref{step:lattice-red} with~$k=1$; if~$k=1$, then stop
    and output~$\sigma$.
  \end{enumerate}
\end{algorithm}

Note that there may be multiple valid choices for~$U$ (resp.~$S$, $\sigma'$) in
steps~\eqref{step:lattice-red}, \eqref{step:real-red} and~\eqref{step:det}.

\begin{prop}
  \label{prop:reduction-terminates}
  Assuming exact operations, \cref{algo:siegel-red} terminates.
\end{prop}

\begin{proof}
  Given the structure of the algorithm, $\det\im(\tau)$ increases strictly each
  time through step~\eqref{step:lattice-red} with~$k=0$. On the other hand, the
  set $\det \im(\Sp_{2g}(\Z)\cdot \tau)\subset [0,+\infty)$ admits only zero as
  an accumulation point by \cite[Lemma 1
  p.\,29]{klingenIntroductoryLecturesIegel1990}. Therefore
  step~\eqref{step:lattice-red} is only performed finitely many times.
\end{proof}

\begin{defn}
  \label{def:strong-red}
  We say that a matrix~$\tau\in \Half_g$ is \emph{strongly reduced}
  if~$\sigma = I_{2g}$ is a valid output of \cref{algo:siegel-red} given~$\tau$
  as input. In other words: $\im(\tau)$ is reciprocal HKZ-reduced;
  $\abs{\re(\tau)}\leq 1/2$; for any matrix~$\sigma\in \Sigma_g$ with lower
  $g\times g$ blocks~$\gamma,\delta$, we have
  $\abs{\det(\gamma\tau+\delta)}\geq 1$; and finally, one can fix integral
  matrices~$U,S$ as above such that $\tau' := U\tau U^T + S$
  satisfies~$\abs{\re(\tau')}\leq 1/2$, $\im(\tau')$ is HKZ-reduced,
  and~$\abs{\det(\gamma\tau'+\delta)}\geq 1$ for each~$\sigma\in \Sigma_g$ with
  lower blocks~$\gamma,\delta$.
\end{defn}

The next proposition is then a tautology:

\begin{prop}
  \label{prop:strong-red}
  Given $\tau\in \Half_g$, and assuming exact operations,
  \cref{algo:siegel-red} outputs~$\sigma\in \Sp_{2g}(\Z)$ such
  that~$\sigma\tau$ is strongly reduced.
\end{prop}

The key differences between \cref{algo:siegel-red} and previous reduction
algorithms are the following:
\begin{enumerate}
\item Step~\eqref{step:lattice-red} alternates between HKZ and reciprocal HKZ
  reductions instead of using only one reduction type. The reason for this
  choice will become apparent in the next subsection: we want to ensure that
  the length of the shortest nonzero vector in $\mathbb{Z}^g$
  for~$\norm{\cdot}_\tau$ is bounded below, and that the diagonal Cholesky
  coefficients $c_1,\ldots,c_g$ as defined in~§\ref{subsec:notation} are as
  large as possible. (Reciprocal) HKZ reduction is exponential time in~$g$, but
  we accept such a cost in the context of evaluating theta functions. Previous
  approaches included using LLL
  reduction~\cite{lenstraFactoringPolynomialsRational1982}, which is polynomial
  time in~$g$ but significantly weaker; only identifying the shortest nonzero
  vector in~$\Z^g$ for~$\norm{\cdot}_\tau$; or using Minkowski reduction
  \cite{minkowskiDiskontinuitatsbereichFurArithmetische1905}, arguably the
  strongest, but which we can only perform for small~$g$. We note that the
  definition of the fundamental domain for the action of $\Sp_{2g}(\Z)$
  on~$\Half_g$ uses Minkowski reduction.

\item Step~\eqref{step:det} looks for~$\sigma$ in the whole set~$\Sigma_g$
  instead of only considering the matrix
  \begin{equation}
    \label{eq:one-gammap}
    \embedsp{g}{1}(J_1) =
    \begin{pmatrix}
      0 & 0 & 1 & 0 \\
      0 & I_{g-1} & 0 & 0\\
      -1 & 0 & 0 & 0\\
      0 & 0 & 0 & I_{g-1}
    \end{pmatrix}.
  \end{equation}
  In theory, we know that there exists an optimal finite list of matrices to
  consider in that step: see \cite[Thm.~2
  p.\,34]{klingenIntroductoryLecturesIegel1990}, which can be generalized to
  other kinds of lattice reductions. As indicated in \cref{def:sp-embed}, this
  list is explicitly known
  when~$g\leq 2$~\cite{gottschlingExpliziteBestimmungRandflaechen1959}, the
  only cases for which Minkowski, LLL, and (reciprocal) HKZ reductions are
  equivalent up to signs. When $g\geq 3$, this optimal list is not explicitly
  known, and we propose~$\Sigma_g$ as a list of interesting matrices to
  consider. Even if~$\Sigma_g$ has exponential size in~$g$, the cost of
  \cref{algo:siegel-red} remains negligible compared to that of evaluating
  theta functions, so the stronger reduction properties are a net benefit.
\end{enumerate}

\begin{rem}
  Our implementation of \cref{algo:siegel-red} currently uses the LLL algorithm
  in step~\eqref{step:lattice-red} which was readily available in FLINT
  3.0. Using \cref{algo:siegel-red} as written instead would be desirable in a
  future version. However, we note that the larger number of symplectic
  matrices used in step~\eqref{step:det} mitigates some drawbacks of not using
  HKZ. As an example, let~$g = 7$, and consider the matrix~$\tau$ obtained
  after the first LLL reduction
  in~\cite[§4.2]{frauendienerEfficientComputationMultidimensional2019} (we
  truncate its entries to $10^{-2}$ instead of~$10^{-4}$ in this display):
  \begin{displaymath}
    \tau =
    % \left(\begin{smallmatrix}
    %   0.0409 + 1.3005i & 0.0530 + 0.3624i & -0.4849 - 0.6245i &-0.1064 - 0.4257i &0.3590 + 0.5023i&  -0.4511 + 0.1383i&  0.2684 - 0.2975i\\
    %   0.0530 + 0.3624i & 0.4364 + 1.0753i & -0.3594 - 0.6598i &0.1205 - 0.1783i & 0.1990 - 0.1118i&  -0.0171 + 0.2485i& -0.4161 + 0.2521i\\
    %   -0.4849 - 0.6245i & -0.3594 - 0.6598i& -0.4706 + 1.3844i &-0.1946 - 0.1178i &-0.0510 - 0.0073i&-0.0543 - 0.3239i& 0.0481 + 0.3949i\\
    %   -0.1064 - 0.4257i & 0.1205 - 0.1783i& -0.1946 - 0.1178i &0.2815 + 0.7791i &0.3684 - 0.0369i& 0.3907 - 0.1531i& -0.2437 - 0.3094i\\
    %   0.3590 + 0.5023i& 0.1990 - 0.1118i& -0.0510 - 0.0073i &0.3684 - 0.0369i &0.2315 + 0.6895i&
    %                                                                                              0.3656 - 0.1563i&  -0.2134 - 0.1308i\\
    %   -0.4511 + 0.1383i& -0.0171 + 0.2485i & -0.0543 - 0.3239i& 0.3907 - 0.1531i &0.3656 - 0.1563i&  -0.4318 + 0.6585i& -0.1541 + 0.0260i\\
    %   0.2684 - 0.2975i& -0.4161 + 0.2521i & 0.0481 + 0.3949i& -0.2437 - 0.3094i &-0.2134 - 0.1308i&  -0.1541 + 0.0260i& -0.4997 + 1.0021i
    % \end{smallmatrix}\right)
    \left(\begin{smallmatrix}
      0.04 + 1.30i & 0.05 + 0.36i & -0.48 - 0.62i &-0.11 - 0.43i &0.36 + 0.50i&  -0.45 + 0.14i&  0.27 - 0.30i\\
      0.05 + 0.36i & 0.44 + 1.08i & -0.36 - 0.66i &0.12 - 0.18i & 0.20 - 0.11i&  -0.02 + 0.25i& -0.42 + 0.25i\\
      -0.48 - 0.62i & -0.36 - 0.66i& -0.47 + 1.38i &-0.19 - 0.12i &-0.05 - 0.01i&-0.05 - 0.32i& 0.05 + 0.39i\\
      -0.11 - 0.43i & 0.12 - 0.18i& -0.19 - 0.12i &0.28 + 0.78i &0.37 - 0.04i& 0.39 - 0.15i& -0.24 - 0.31i\\
      0.36 + 0.50i& 0.20 - 0.11i& -0.05 - 0.01i &0.37 - 0.04i &0.23 + 0.69i& 0.37 - 0.16i&  -0.21 - 0.14i\\
      -0.45 + 0.14i& -0.02 + 0.25i & -0.05 - 0.32i& 0.39 - 0.15i &0.37 - 0.16i&  -0.43 + 0.66i& -0.15 + 0.03i\\
      0.27 - 0.30i& -0.42 + 0.25i & 0.05 + 0.39i& -0.24 - 0.31i &-0.21 - 0.13i&  -0.15 + 0.03i& -0.50 + 1.00i
    \end{smallmatrix}\right).
  \end{displaymath}

  While an application of $\embedsp{g}{1}(J_1)$ would not
  increase~$\det\im(\tau)$, one loop through \Cref{algo:siegel-red} reveals
  that we should use~$\embedsp{g}{1,2}(J_2)$ instead. After further reduction
  of the imaginary and real parts, we obtain a reduced
  matrix~$\tau'\in \Sp_{2g}(\Z)\tau$ such that $\det\im(\tau')\simeq 0.108$,
  which seems to be the maximal value. This was also achieved
  in~\cite[§4.2]{frauendienerEfficientComputationMultidimensional2019} in two
  steps, using only~$\embedsp{g}{1}(J_1)$ but identifying shortest vectors
  instead of relying on LLL at each step.
\end{rem}

\subsection{Reduced matrices}
\label{subsec:properties}

We are now able to indicate which matrices $\tau\in \Half_g$ are considered in
\cref{thm:main-intro}, and list some of their properties for use in later
sections.

\begin{defn}
  \label{def:reduced-tau}
  We say that~$\tau\in \Half_g$ is \emph{reduced} if the following conditions
  hold:
  \begin{itemize}
  \item $\im(\tau)$ is reduced in the sense of reciprocal HKZ reduction for
    Gram matrices.
  \item The shortest nonzero vector in~$\Z^g$ for $\norm{\cdot}_\tau$ has
    squared norm at least~$\sqrt{3}/2$.
  \end{itemize}
\end{defn}

\begin{lem}
  \label{lem:strong-implies-red}
  Any strongly reduced $\tau\in \Half_g$ is reduced.
\end{lem}

\begin{proof}
  Assume that $I_{2g}$ is a valid output of \cref{algo:siegel-red}. Then,
  during the $k=0$ loop, step~\eqref{step:lattice-red} ensures that
  $e_1=(1,0,\ldots,0)$ is the shortest nonzero vector
  for~$\norm{\cdot}_\tau$. Since $\embedsp{g}{1}(J_1)\in \Sigma_g$,
  step~\eqref{step:det} ensures that $\abs{\tau_{1,1}}^2\geq 1$, while
  $\abs{\re(\tau_{1,1})}\leq 1/2$ by step~\eqref{step:real-red}. Hence
  $\im(\tau_{1,1}) = \norm{e_1}_\tau^2 \geq \sqrt{3}/2$, as in the usual
  picture of the fundamental domain for~$g=1$. The $k=1$ loop then ensures
  that~$\im(\tau)$ is reciprocal HKZ-reduced without changing the length of the
  shortest nonzero vector.
\end{proof}

The next immediate lemma is of interest in the context of the duplication
formula, and does not hold in general if we replace ``reduced'' by
``strongly reduced'', for instance due to real part issues.

\begin{lem}
  \label{lem:2-power-red}
  If $\tau\in \Half_g$ is reduced, then so is $2^h\tau$ for every~$h\geq 0$.
\end{lem}

\begin{lem}
  \label{lem:submatrix-red}
  Let~$\tau\in \Half_g$ be reduced, and let~$\tau'$ denote the upper left
  $(g-1)\times (g-1)$ submatrix of~$\tau$. Then~$\tau'$ is reduced as an
  element of~$\Half_{g-1}$.
\end{lem}

\begin{proof}
  By definition of reciprocal HKZ reduction, $\im(\tau')$ is again reciprocal
  HKZ-reduced. In addition, if~$v\in \Z^{g-1}$ is any nonzero vector, then
  \begin{displaymath}
    \norm{v}_{\tau'}^2 = \norm{(v,0)}_{\tau}^2\geq \sqrt{3}/2. \qedhere
  \end{displaymath}
\end{proof}

Next, we focus on Cholesky coefficients and distances.

\begin{prop}
  \label{prop:hkz-bound} Let $\tau\in \Half_g$ be a reduced matrix, and
  let~$c_1,\ldots,c_g$ be the diagonal Cholesky coefficients for $\pi\im(\tau)$
  as in~§\ref{subsec:notation}. Then for every $1\leq j\leq g$, we have
  \begin{displaymath}
    c_j \geq \frac{1.649}{g} \quad \text{and} \quad c_j \leq g^{(1+\log g)/2}c_g.
  \end{displaymath}
\end{prop}

\begin{proof}
  By \cite[Prop.~4.1]{lagariasKorkinZolotarevBasesSuccessive1990}, and given
  that the Hermite constant $\gamma_j$ satisfies $\gamma_j\leq j$, we have
  \begin{displaymath}
    c_j \geq \frac{\sqrt{\pi\tfrac{\sqrt{3}}{2}}}{j} \geq \frac{1.649}{g}.
  \end{displaymath}
  For the second inequality, we apply
  \cite[Prop.~4.2]{lagariasKorkinZolotarevBasesSuccessive1990} to the
  reciprocal lattice.
\end{proof}

\begin{cor}
  \label{cor:hkz-dist}
  For every reduced~$\tau\in \Half_g$ and every $v\in \R^g$, we have
  \begin{displaymath}
    \Dist_\tau(v,\Z^g)^2 \leq \frac14\,g^{2+\log g} c_g^2.
  \end{displaymath}
\end{cor}

\begin{proof}
  By \cref{prop:hkz-bound}, we have
  \begin{displaymath}
    \Dist_\tau(v,\Z^g)^2 \leq \frac 14 \sum_{j=1}^g c_j^2 \leq \frac14 \sum_{j=1}^g j^{1+\log j} c_g^2 \leq \frac 14g^{2 + \log g} c_g^2. \qedhere
  \end{displaymath}
\end{proof}

The next key point is that the ellipsoids $\eld{v}{R}{\tau}$
for~$\norm{\placeholder}_\tau$, as defined in the introduction, are not too
skewed for reduced~$\tau$. We rely on the following inductive procedure to list
all points $n\in \Z^g$ in such an ellipsoid (see~\cite[Remark
p.\,9]{deconinckComputingRiemannTheta2004}). Using notation
from~§\ref{subsec:notation}, we split the matrix~$C$ as
\begin{equation}
  \label{eq:C-split}
  C =
  \begin{pmatrix}
    C' & \Xi \\ 0 & c_{g}
  \end{pmatrix}
\end{equation}
where~$C'$ and~$\Xi$ have size $(g-1)\times(g-1)$ and $(g-1)\times 1$
respectively. Note that $C'$ is the Cholesky matrix attached to
$\pi \im(\tau')$, where~$\tau'$ is defined as in \cref{lem:submatrix-red}. For
every $n = (n_1,\ldots,n_{g})\in \Z^g$, we can write
\begin{equation}
  \label{eq:pyth}
  \begin{aligned}
    \norm{n-v}_\tau^2 &= \norm{C(n-v)}_2^2\\
                      &= \norm{C' (n' - v') + \Xi (n_{g} - v_{g})}_2^2
                        + c_{g}^2\abs{n_{g}- v_{g}}^2\\
                      &=  \norm{n' - v''}_{\tau'}^2 + c_{g}^2\abs{n_{g}- v_{g}}^2,
  \end{aligned}
\end{equation}
where~$n'$ and~$v'$ consist of the first $(g-1)$ coordinates of~$n$ and~$v$
respectively, and~$v'' := v' + C'^{-1} \Xi (v_{g} - n_{g}) \in \R^{g-1}$ is
independent of~$n'$.

Note that equality~\eqref{eq:pyth} can be interpreted as the Pythagorean
theorem for~$\norm{\placeholder}_\tau$: if~$W\subset \R^g$ denotes the
$(g-1)$-dimensional subspace of vectors whose last coordinate is equal to
zero, then $c_{g}\abs{n_{g} - v_{g}}$ is the distance between the affine
subspaces $n + W$ and $v + W$ for~$\norm{\placeholder}_\tau$, while
$n' + C'^{-1} \Xi n_{g}$ and $v' + C'^{-1} \Xi v_{g}$ are the orthogonal
projections of~$n$ and~$v$ onto~$W$.

Equation~\eqref{eq:pyth} leads to the following algorithm to list points
$n\in \eld{v}{R}{\tau}\cap \Z^g$, assuming operations on real numbers can be
performed exactly. First list all possible values of~$n_g \in \Z$, namely those
for which $c_g^2\abs{n_g - v_g}^2 \leq R^2$. If~$g=1$, we are done.
If~$g\geq 2$, then for each choice of~$n_g$, compute the possible
$(g-1)$-tuples $n'$ such that~$(n',n_g)\in \eld{v}{R}{\tau} \cap\Z^g$ by
recursively listing all points~$n'\in \eld{v''}{R'}{\tau'}\cap \Z^{g-1}$
where~$R'^2 = R^2 - c_g^2 \abs{n_g - v_g}^2$.

\begin{prop}
  \label{prop:eld-width}
  Let~$\tau\in \Half_g$ be a reduced matrix, let~$v\in \R^g$, and
  let~$R\in \R_{\geq 0}$ be such that $R^2 = \Dist_\tau(v,\Z^g)^2 + \delta^2$
  for some~$\delta\geq 0$. Choose~$1\leq d\leq g$, let~$n$ be a fixed point
  of~$\eld{v}{R}{\tau}\cap\Z^g$, and define
  \begin{displaymath}
    S_d(n) = \{n'_d: n'\in \eld{v}{R}{\tau}\cap\Z^g \text{ and } n'_j = n_j \text{ for each } d+1\leq j\leq g\}.
  \end{displaymath}
  Then~$S_d(n)$ is included in $I\cap\Z$ for a certain closed real interval~$I$ of
  length
  \begin{displaymath}
    \min\{2R/c_d,\, d^{1 + (\log d)/2} + 2\delta/c_d\}.
  \end{displaymath}
\end{prop}

\begin{proof}
  In the above inductive procedure, we always have
  \begin{displaymath}
    \Dist_\tau(v, \Z^g)^2 \leq c_g\abs{n_g - v_g}^2 + \Dist_{\tau'}(v'', \Z^g)^2
  \end{displaymath}
  by~\eqref{eq:pyth}. Therefore the inequality
  $R^2 - \Dist_\tau(v, \Z^g)^2\leq \delta^2$ remains true throughout the
  algorithm. To prove the above statement, we can therefore assume $d =
  g$. Then the inequality $c_g^2 \abs{n_g - v_g}^2 \leq R^2$ implies
  $\abs{n_g - v_g}\leq R/c_g$, and by \cref{cor:hkz-dist},
  \begin{displaymath}
    \abs{n_g - v_g}^2 \leq \frac{\delta^2}{c_g^2} + \frac 14 g^{2+\log g},
    \quad\text{so}\quad \abs{n_g - v_g} \leq \frac{\delta}{c_g} + \frac{1}{2} g^{1 + (\log g)/2}. \qedhere
  \end{displaymath}
\end{proof}

\begin{cor}
  \label{cor:eld-nb-pts}
  With the notation of \cref{prop:eld-width}, we have
  \begin{displaymath}
    \# \paren[\big]{\eld{v}{R}{\tau}\cap\Z^g} \leq \prod_{d=1}^g \paren[\bigg]{1 + \floor[\bigg]{d^{1 + (\log d)/2} + \frac{2\delta}{c_d}}}
    \leq \paren[\big]{ 2^{2 + \log(g)^2} + \delta g}^g
  \end{displaymath}
  and
  \begin{displaymath}
    \# \paren[\big]{\eld{v}{R}{\tau}\cap\Z^g} \leq \prod_{d=1}^g \paren[\bigg]{1 + \floor[\bigg]{\frac{2R}{c_d}}} \leq (1 + Rg)^g.
  \end{displaymath}
\end{cor}

\begin{prop}
  \label{prop:eld-width-2}
  Let~$\tau\in \Half_g$ be a reduced matrix. Then in the notation
  of~§\ref{subsec:notation}, we have $\norm{C^{-1}}_\infty \leq 2g^2$. In
  particular, any $n\in \eld{v}{R}{\tau}\cap\Z^g$ satisfies
  $\norm{n - v}_\infty\leq 2 R g^2$.
\end{prop}

\begin{proof}
  Up to transposing it and reversing its lines and columns, the matrix~$C^{-1}$
  is the Cholesky matrix of an HKZ-reduced Gram matrix, and its diagonal
  coefficients are~$1/c_1,\ldots, 1/c_g$. In particular, this Gram matrix is
  size-reduced. Hence
  \begin{displaymath}
    \norm{C^{-1}}_\infty \leq \sum_{j=1}^g \frac{1}{c_g} \leq 2g^2
  \end{displaymath}
  by \cref{prop:hkz-bound}. For the second part of the statement, we write
  \begin{displaymath}
    \norm{n - v}_\infty \leq \norm{C^{-1}}_\infty \norm{C (n-v)}_\infty \leq \norm{C^{-1}}_\infty R. \qedhere
  \end{displaymath}
\end{proof}

\begin{rem}
  If we only used HKZ reduction in \cref{algo:siegel-red} and
  \cref{def:reduced-tau}, then we could ensure that
  $c_j\geq 1.649 g^{-1-\log g}$ in \cref{prop:hkz-bound}.

  It would be interesting to find out whether the above results can be
  strengthened for strongly reduced matrices. A positive answer would bring
  theoretical support to the somewhat arbitrary definition of~$\Sigma_g$ in
  \cref{def:sp-embed}.
\end{rem}

\subsection{Reducing the first argument}
\label{subsec:z-reduction}

After~$\tau$ has been reduced, the next step is to reduce the vector~$z$ to
ensure that $v=-Y^{-1}y$ is reasonably close to zero, a condition that will
prove convenient in summation algorithms.

This is done as follows. Let~$w\in 2\Z^g$ be an even vector such that
$\norm{v-w}_\infty\leq 1$, and let~$z' = z + \tau w$. Let~$y' = \im(z')$ and
$v'=-Y^{-1}y'$. Then
\begin{equation}
  \label{eq:small-v}
  v' = -Y^{-1}(y + Y w) = v - w, \quad \text{so } \norm{v'}_\infty \leq 1.
\end{equation}
Moreover, for every characteristic $(a,b)$, we have as $w$ is even
\begin{equation}
  \label{eq:period-z}
  \theta_{a,b}(z,\tau) = \e(w^T\tau w + 2w^T z) \theta_{a,b}(z',\tau).
\end{equation}

To evaluate theta functions at $(z,\tau)$, we can therefore simply compute
$z'$, evaluate theta functions at $(z',\tau)$, then multiply by the correct
cofactor. Note that if we compute $\thetatilde_{a,b}(z',\tau)$ to absolute
precision~$N$, we also obtain $\thetatilde_{a,b}(z,\tau)$ to absolute
precision~$N$ (ignoring precision losses in the
multiplication~\eqref{eq:period-z}), because
\begin{equation}
  \label{eq:translate-thetatilde}
  \begin{aligned}
    \abs{\e(w^T\tau w + 2 w^T z)} &= \exp\paren[\big]{-\pi(y' - y)^T Y^{-1} (y' - y) - 2\pi (y'-y)^T Y^{-1} y}\\
    &= \exp\paren[\big]{\pi y^T Y^{-1} y - \pi y'^T Y^{-1} y'}.
  \end{aligned}
\end{equation}

\begin{defn}
  \label{def:reduced}
  We say that $(z,\tau) \in \C^g\times\Half_g$ is \emph{reduced} if~$\tau$ is
  reduced as in \cref{def:reduced-tau} and the condition
  $\norm{v}_\infty\leq 1$ holds. We denote the set of reduced points by
  $\red_g\subset \C^g\times\Half_g$.
\end{defn}

\begin{rem}
  \label{rem:stronger-z-red}
  A stronger notion of reduction would be to enforce
  $\norm{v}_\infty \leq \frac12$, allowing~$w$ to have odd coordinates. In that
  case however,~\eqref{eq:period-z} would involve signs and would modify the
  theta characteristic. We did not implement this variant. One could also
  enforce the inequality $\abs{\re(z)}\leq \frac12$ or $\abs{\re(z)}\leq 1$,
  but we do not require this in the rest of the paper.

  Relation~\eqref{eq:period-z} can be differentiated to express derivatives of
  theta functions at~$z$ in terms of derivatives at~$z'$: see the
  \href{https://flintlib.org/doc/acb_theta.html#c.acb_theta_jet}{FLINT documentation}.

\end{rem}

\subsection{Argument reduction and interval arithmetic}
\label{subsec:inexact-reduction}

In practice, in the context of interval arithmetic, we cannot run
\cref{algo:siegel-red} with tight inequalities as written. Instead, we choose a
tolerance parameter~$\eps > 0$ and interpret the different steps in this
algorithm as follows.
\begin{itemize}
\item In step~\eqref{step:lattice-red}, we choose a working
  precision~$N\in \Z_{\geq 2}$, look for a positive definite symmetric
  matrix~$M$ whose entries are integers divided by~$2^N$ and such
  that $\abs{M - \im(\tau')} \leq 2^{-N}$, and reduce the integral
  matrix~$2^N M$ to find~$U$.  This step fails if the error bound
  on~$\tau'$ is larger than~$2^{-N}$ or if $\abs{U M U^T - U \im(\tau') U^T}$ is
  not smaller than $\eps$.
\item In step~\eqref{step:real-red}, we impose that
  $\abs{\re(\tau)+S}< \tfrac 12 + \eps$. This step fails if the current error
  bound on~$\tau$ is larger than~$\eps$.
\item In step~\eqref{step:det}, we compute certified lower bounds
  on~$\det\im(\sigma'\tau)$, and pick a matrix~$\sigma'$ that realizes the
  greatest lower bound.
\item In step~\eqref{step:red-back}, we compare the lower bound from
  step~\eqref{step:det} with a certified upper bound on~$\det\im\tau$. We also
  check that $\det \im(\sigma'\tau) < (1+\eps)\det\im\tau$ for
  every~$\sigma'\in \Sigma_g$ before stopping; otherwise the algorithm fails.
\end{itemize}

With these modifications, \Cref{algo:siegel-red} still terminates,
because~$\det(\im\tau)$ still increases each time through the loop. If the
reduction algorithm succeeds, then it outputs a matrix~$\sigma\in \Sp_{2g}(\Z)$
so that~$\sigma\tau$ is close to being strongly reduced; the smaller~$\eps$ the
better.

A complete study of the complexity of evaluating theta functions on~$\Half_g$
would include answering the following questions: given~$\tau\in \Half_g$
and~$\eps \in\R_{>0}$,
\begin{enumerate}
\item How to choose the working
  precision~$N$ as above?
\item What is an upper bound (perhaps in terms of~$g$, $(\det \im\tau)^{-1}$
  and~$\abs{\tau}$) on the number of iterations in \cref{algo:siegel-red}?
\item What is an upper bound on the precisions that we need to consider
  throughout \cref{algo:siegel-red} to ensure its success?
\item How can we find a reduced matrix $\tau''\in \Half_g$ close
  to~$\sigma\tau$?  What is an upper bound on~$\abs{\tau'' - \sigma\tau}$ in
  terms of~$\eps$?
\end{enumerate}

Answering these questions is not completely out of reach: for example, it is
done for~$g=2$ in~\cite[§6.4]{strengComputingIgusaClass2014} in the context of
exact operations and~\cite{kiefferEvaluatingModularEquations2022} in the
context of interval arithmetic. (In particular, note how
recomputing~$\tau' = \sigma\tau$ regularly in \cref{algo:siegel-red} mitigates
precision losses; the working precision that is necessary at each step can be
estimated in terms of~$\eps$ and the current matrix~$\tau'$.) For a
general~$g$, a starting point would be to provide an upper bound on the norm of
the base-change matrix~$U$ when applying (reciprocal) HKZ reduction in order to
replace \cite[Lemma 6.6]{strengComputingIgusaClass2014}. Answering the fourth
question above would allow us to approximate theta values at~$\sigma\tau$ by
theta values at~$\tau''$, justifying the fact that we only consider reduced
input in \cref{thm:main-intro}.

Similar (but easier) questions arise when reducing~$z$ as
in~§\ref{subsec:z-reduction}.

In the present paper, we ignore these tolerance issues and only consider exact
input $(z,\tau)\in \mathcal{R}_g$ in \cref{thm:main-intro}. In practice,
\cref{algo:siegel-red} is very quick compared to the evaluation of theta
functions, and the fast evaluation algorithm works equally well if a small,
positive tolerance parameter was used in the reduction process.

\subsection{The transformation formula}
\label{subsec:transf}

In order to state the theta transformation formula succinctly, we introduce the
following notation.  Consider the metaplectic group~$\Mp_{2g}(\Z)$ whose
elements are pairs~$(\sigma,f)$, where~$\sigma\in \Sp_{2g}(\Z)$ and $f$ is a
holomorphic square root of $\tau\mapsto\det(\gamma\tau+\delta)$,
where~$\gamma,\delta$ denote the lower $g\times g$ blocks of~$\sigma$. The
group operation in~$\Mp_{2g}(\Z)$ is the obvious one:
\begin{equation}
  \label{eq:metaplectic}
  (\sigma,f)\cdot(\sigma',f') = \paren[\big]{\sigma\sigma', \tau\mapsto f'(\sigma\tau) f(\tau)}.
\end{equation}
Denote by $\mu_8\subset\C^\times$ the group of 8th roots of unity, and
let~$\Perm_g(\mu_8)$ be the group of ``$\mu_8$-signed'' permutations of
$2^{2g}$ elements, i.e.~$\Perm_g(\mu_8)$ is the group of $2^{2g}\times 2^{2g}$
matrices obtained as the product of a diagonal matrix with coefficients taken
in~$\mu_8$ and a permutation matrix. Finally,
let~$\Theta:\C^g\times\Half_g\to \C^{2^{2g}}$ denote the collection of all
theta functions $\theta_{a,b}$ for $a,b\in \{0,1\}^g$ in some fixed ordering.

\begin{thm}[{\cite[(5.1) p.\,189]{mumfordTataLecturesTheta1983},
    \cite[Thm.~2 p.\,175]{igusaThetaFunctions1972}}]
  \label{thm:theta-transf}
  There exists a group homomorphism $\phi: \Mp_{2g}(\Z)\to \Perm_g(\mu_8)$ with
  the following property. Let $(\sigma,f)\in \Mp_{2g}(\Z)$, and denote by
  $\gamma,\delta$ the lower $g\times g$ blocks of~$\sigma$. Then for all
  $(z,\tau)\in \C^g\times\Half_g$, we have
  \begin{displaymath}
    \Theta\paren[\big]{\sigma\cdot(z,\tau)} = \e\paren[\big]{z^T(\gamma\tau + \delta)^{-1} \gamma z}
    f(\tau)\, \phi(\sigma,f) \cdot \Theta(z,\tau).
  \end{displaymath}
\end{thm}

Further,~$\phi$ induces a group homomorphism
$\overline{\phi}:\Sp_{2g}(\Z) \to \Perm_g(\mu_8)/\mu_8$, where~$\mu_8$ is
embedded in $\Perm_g(\mu_8)$ as scalar matrices, i.e.~as its
center. Igusa~\cite[Thm.~2 p.\,175]{igusaThetaFunctions1972} gives an explicit
formula for the induced morphism $\overline{\phi}$, which we implemented but do
not need to reproduce here.  \Cref{thm:theta-transf} is the second main reason
why we consider theta functions with characteristics instead of
just~$\theta_{0,0}$, besides the duplication formula~\eqref{eq:dupl-const}.

When applying the transformation formula, we are given a
matrix~$\sigma\in \Sp_{2g}(\Z)$ and a point~$(z,\tau)\in \C^g\times\Half_g$,
say in the reduced domain~$\red_g$. We then have to choose a square root of
$\det(\gamma\tau + \delta)$, i.e.~evaluate $f(\tau)$ for a certain
lift~$(\sigma,f)$ of~$\sigma$ to $\Mp_{2g}(\Z)$, then determine the
$\mu_8$-signed permutation $\phi(\sigma,f)$.  By Igusa's result, we know
$\phi(\sigma,f)$ up to multiplication by an $8$th root of unity, so determining
one coefficient of $\phi(\sigma,f)$ will be enough for our purposes. (This is
equivalent to determining Igusa's $\kappa(\sigma)$ for the choice of square
root specified by~$f$; however, we find~$\kappa$ less convenient to manipulate
as it is not a group homomorphism.)

A first idea is to choose the square root~$f(\tau)$ of $\det(c\tau + d)$
arbitrarily and compute theta values af both~$\tau$ and $\sigma\tau$ at low
precision to determine the correct $\phi(\sigma,f)$. This is however extremely
slow when~$\sigma$ has large coefficients.

A better strategy is to decompose~$\sigma$ into a product of elementary
matrices, for which a lift to $\Mp_{2g}(\Z)$ and the value of~$\phi$ can be
readily determined from the proof of the transformation formula. This idea is
already used in~\cite[§7]{deconinckComputingRiemannTheta2004}, where~$\sigma$
is obtained as the result of the reduction algorithm with only
$\embedsp{g}{1}(J_1)$ being used in step~\eqref{step:det}, and thus comes with
a free decomposition. An issue we face is that not all the elements
of~$\Sigma_g$ are elementary. Moreover, one might want to consider
matrices~$\sigma$ that do not arise from the reduction
algorithm. In~§\ref{subsec:decomp} below, we instead describe an algorithm to
directly write down a given~$\sigma\in \Sp_{2g}(\Z)$ as a product of elementary
matrices which is often shorter than what \cref{algo:siegel-red} produces. To
this end, we extend our definition of elementary matrices to include embedded
matrices from~$\SL_2(\Z)$, as in the following definition.

\begin{defn}
  \label{defn:elementary}
  We say that $\sigma\in \Sp_{2g}(\Z)$ is \emph{elementary} if one of the following conditions holds:
  \begin{itemize}
  \item There exists~$U\in \GL_g(\Z)$ such that $\sigma = \DiagSp(U)$;
  \item There exists a symmetric~$S\in \Mat_{g\times g}(\Z)$ such that $\sigma = \TrigSp(S)$;
  \item There exists~$\sigma'\in \SL_2(\Z)$ such that~$\sigma = \embedsp{g}{1}(\sigma')$;
  \item There exist $1\leq r\leq g$ and $1\leq i_1 < \ldots < i_r \leq g$ such
    that $\sigma = \embedsp{g}{i_1,\ldots,i_r}(J_r)$.
  \end{itemize}
\end{defn}

Given an elementary $\sigma\in \Sp_{2g}(\Z)$, one can deterministically pick a
lift $(\sigma,f)$ of~$\sigma$ to~$\Mp_{2g}(\Z)$ and
determine~$\phi(\sigma,f)$. Moreover, there is an efficient algorithm which,
given~$\tau\in \Half_g$, evaluates~$f(\tau)$. This works as follows.
\begin{enumerate}
\item If~$\sigma = \DiagSp(U)$ for some~$U\in \GL_g(\Z)$, then we choose an
  arbitrary square root $\zeta\in \{\pm 1,\pm i\}$ of~$\det(U)$, and let~$f$ be
  the constant function equal to~$\zeta$. We determine~$\phi(\sigma,f)$ from
  the relation $\theta_{0,0}(0,U\tau U^T) = \theta_{0,0}(0,\tau)$ for
  all~$\tau\in \Half_g$.
\item If~$\sigma = \TrigSp(U)$ for some symmetric~$S\in \Mat_{g\times g}(\Z)$,
  then we choose~$f$ to be the constant~$1$. Let~$b\in\{0,1\}^g$ denote the diagonal
  of~$S$ modulo~$2$. We determine~$\phi(\sigma,f)$ from the relation
  $\theta_{0,0}(0,\tau + S) = \theta_{0,b}(0,\tau)$ for all~$\tau\in \Half_g$.
\item If~$\sigma = J_g$, then we let~$f$ be the unique holomorphic square root
  of $\tau\mapsto\det(-\tau)$ on~$\Half_g$ such that $\zeta_8^g f(iY)$ is a
  positive real number whenever~$Y\in \Mat_{g\times g}(\R)$ is positive
  definite, where~$\zeta_8 = \e(1/4)$. We determine~$\phi(\sigma,f)$ from
  \cite[Case III p.\,195]{mumfordTataLecturesTheta1983}: we have
  $\theta_{0,0}(0,J_g\tau) = \zeta_8^g f(\tau) \theta_{0,0}(0,\tau)$ for
  every~$\tau\in \Half_g$.
\item If~$g = 1$, then let~$\gamma,\delta$ be the lower entries
  of~$\sigma\in \SL_2(\Z)$. Up to considering~$-\sigma$ which has the same
  action on~$\Half_g$, we can assume that~$\gamma\geq 0$, and $\delta = 1$ if
  $\gamma = 0$. Then we let~$f$ be the unique square root of
  $\tau\mapsto \det(\gamma\tau + \delta)$ such that $\re(f)>0$. We
  determine~$\phi(\sigma,f)$ from Rademacher's
  formulas~\cite{rademacherTopicsAnalyticNumber1973}, which were already
  implemented in FLINT~\cite{engeShortAdditionSequences2018}.
\end{enumerate}

We still have to explain how to deal with embedded matrices and to provide an
algorithm to evaluate $f(\tau)$ in the case~$\sigma = J_g$.

First, consider indices~$1\leq i_1<\ldots < i_r\leq g$. Then the embedding
$\embedsp{g}{i_1,\ldots,i_r}$ is induced by an embedding
$\Mp_{2r}(\Z)\to \Mp_{2g}(\Z)$, also denoted
by~$\embedsp{g}{i_1,\ldots,i_r}$. Indeed, let~$(\sigma,f)\in \Mp_{2r}(\Z)$
and~$\sigma' = \embedsp{g}{i_1,\ldots,i_r}(\sigma)$. We can
define~$f':\Half_g\to \C^\times$ as follows: for every~$\tau\in \Half_g$,
$f'(\tau) = f(\tau_0)$ where~$\tau_0$ denotes the $r\times r$ submatrix
of~$\tau$ obtained by keeping only the rows and columns indexed
$i_1,\ldots,i_r$. The next proposition allows us to determine the value
of~$\phi$ on~$\embedsp{g}{i_1,\ldots,i_r}(\sigma,f)$ when~$\phi(\sigma,f)$ is
known.

\begin{prop}
  \label{prop:phi-embed}
  Let~$(\sigma,f)\in \Mp_{2r}(\Z)$, and let $\zeta\in \mu_8$ and
  $a,b\in\{0,1\}^r$ be such that the relation
  $\theta_{0,0}(0,\sigma\tau) = \zeta f(\tau) \theta_{a,b}(0,\tau)$ holds for
  all~$\tau\in \Half_r$. Define a characteristic $a',b'\in \{0,1\}^g$ as
  follows: for each $1\leq k\leq r$, $a'_{i_k} = a_k$; $a'_j = 0$ for other
  indices~$j$; and similarly
  for~$b'$. Let~$(\sigma',f')=\embedsp{g}{i_1,\ldots,i_r}(\sigma,f)$. Then for
  all~$\tau'\in \Half_g$, we have
  $\theta_{0,0}(0,\sigma'\tau') = \zeta f'(\tau') \theta_{a',b'}(0,\tau')$.
\end{prop}

In fact, we can completely describe the matrix~$\phi(\sigma',f')$, but
specifying one entry as in \cref{prop:phi-embed} is sufficient.

\begin{proof}
  By \cref{thm:theta-transf}, it is enough to prove the claimed relation
  when~$\tau'\in \Half_g$ is also an embedded matrix, i.e.~there
  exists~$\tau\in \Half_r$ such that for all $1\leq j,k\leq g$,
  \begin{displaymath}
    \tau'_{j,k} =
    \begin{cases}
      \tau_{l,m} &\text{if there exists $1\leq l,m\leq r$ such that $j = i_l$ and $k = i_m$},\\
      i &\text{if } j = k\notin\{i_1,\ldots,i_r\},\\
      0 &\text{otherwise}.
    \end{cases}
  \end{displaymath}
  Then~$\sigma'\tau'\in \Half_g$ is also embedded from~$\sigma\tau\in
  \Half_r$. Let~$\lambda = \theta_{0,0}(0,i)\neq 0$ (a~theta constant for
  $g=1$).
  Then~$\theta_{0,0}(0,\sigma'\tau') = \lambda^{g-r}\theta_{0,0}(0,\sigma\tau)$
  and $\theta_{a',b'}(0,\tau') = \lambda^{g-r}\theta_{a,b}(0,\tau)$, from which
  the result follows.
\end{proof}

Second, we describe how to evaluate $f(\tau)$ when $\sigma = J_g$. We propose
the following method inspired
from~\cite[§3.2.3]{molinComputingPeriodMatrices2019}, which can fail at any
step if the chosen working precision is too small.

\begin{algorithm}
  \label{algo:metaplectic}
  \algoinput{$\tau\in \Half_g$.}  \algooutput{The square root $f(\tau)$ of
    $\det(-\tau)$, where~$(J_g,f)$ is the lift of~$J_g$ to~$\Mp_{2g}(\Z)$ as
    specified above.}
  \begin{enumerate}
  \item Pick a positive definite $Y\in \Mat_{g\times g}(\R)$ such that
    $\det(\tau - iY)$ is not too small (in particular nonzero.) We may take~$Y$
    to be a diagonal matrix close to~$I_g$.
  \item The polynomial
    $P(t)=\det\paren[\big]{-iY + \tfrac{t+1}{2}(iY - \tau)} \in \C[t]$ has
    degree~$g$, and satisfies $P(-1) = \det(-iY)$ and $P(1) = \det(-\tau)$. Isolate
    the~$g$ roots $\lambda_1,\ldots,\lambda_g$ of~$P$ as complex balls. If any
    root of~$P$ intersects the segment $[-1,1]\subset\C$, then go back to
    step~(1) and choose another~$Y$. After reordering the roots, find an
    index~$0\leq r\leq g$ such that $\re(\lambda_j)\leq 0.01$ for each
    $1\leq j\leq r$ and $\re(\lambda_j)\geq -0.01$ for each $r+1\leq j\leq g$.
  \item For each $t\in [-1,1]$, the quantities $t-\lambda_j$ for
    $1\leq j\leq r$ and $\lambda_j - t$ for $r+1\leq j\leq g$ lie in the domain
    of definition $\C\setminus (-\infty,0]$ of the principal branch of the
    square root, so we can define
    \begin{displaymath}
      Q(t) = \prod_{j=1}^r \sqrt{t-\lambda_j} \prod_{j=r+1}^g \sqrt{\lambda_j - t}.
    \end{displaymath}
    The function~$Q$ is holomorphic on an open neighborhood of~$[-1,1]$ in~$\C$
    and satisfies~$Q^2 = \pm P$. Find~$\lambda\in \{\pm 1\}$ such
    that~$\lambda \zeta_8^{g} Q(-1) = \sqrt{\det Y}\in \R_{>0}$, and
    output~$\lambda Q(1)$.
  \end{enumerate}
\end{algorithm}

A thorough complexity analysis would involve analyzing the necessary working
precision, in terms of~$\tau$, to be able to isolate the roots of~$P$ for a
suitable choice of~$Y$, then the complexity of extracting the roots. We leave
these questions to future work as \cref{algo:metaplectic} is very efficient in
practice.

Summarizing, decomposing a given~$\sigma\in\Sp_{2g}(\Z)$ as a product of
elementary matrices specifies a lift of~$\sigma$ to~$\Mp_{2g}(\Z)$, the value
of~$\phi(\sigma,f)$, and an algorithm to evaluate~$f(\tau)$
given~$\tau\in \Half_g$ using the group law~\eqref{eq:metaplectic}. In
practice, we never write down the huge matrices~$\phi(\sigma,f)$.  Instead, we
use Igusa's formula for~$\overline{\phi}$ directly and use one entry
of~$\phi(\sigma,f)$ to track down the correct root of unity.

\begin{rem}
  Once~$\phi(\sigma,f)$ has been determined, it is straightforward to
  differentiate the transformation formula~\ref{thm:theta-transf} to also
  handle partial derivatives of theta functions. We refer to the
  \href{https://flintlib.org/doc/acb_theta.html#c.acb_theta_jet}{FLINT
    documentation} for more details.
\end{rem}

\subsection{Decomposition in elementary matrices}
\label{subsec:decomp}

The algorithm to decompose a given~$\sigma\in \Sp_{2g}(\Z)$ as a product of
elementary matrices is inspired from the proof of~\cite[Lemma 15
p.\,45]{igusaThetaFunctions1972}, which asserts that $\Sp_{2g}(\Z)$ is
generated by~$J_g$ and matrices of the form~$\DiagSp(U)$
and~$\TrigSp(S)$. Throughout the algorithm, the
letters~$\alpha,\beta,\gamma,\delta$ refer to the $g\times g$ blocks of the
current~$\sigma\in \Sp_{2g}(\Z)$.

\begin{algorithm}
  \label{algo:decomposition}
  \algoinput{$\sigma\in \Sp_{2g}(\Z)$.}  \algooutput{A list of elementary
    matrices in~$\Sp_{2g}(\Z)$ whose product is~$\sigma$.}
  \begin{enumerate}
  \item If~$g\leq 1$, then output~$\sigma$.
  \item \label{step:smith} Compute the Smith normal form of~$\gamma$, yielding $U,U'\in \GL_g(\Z)$
    such that after replacing~$\sigma$ by $\DiagSp(U)\sigma \DiagSp(U')$, the
    matrix $\gamma$ is diagonal, and its coefficients
    $\gamma_1,\ldots,\gamma_g$ are nonnegative and satisfy
    $\gamma_1 | \gamma_{2},\ \gamma_{2}|\gamma_{3},\ \ldots,\ \gamma_{g-1}
    |\gamma_g$. The matrix~$\delta$ is then symmetric.
  \item \label{step:smith-loop} While $\gamma$ has full rank, i.e.~$\gamma_g > 0$, do:
    \begin{enumerate}
    \item Compute a symmetric matrix $S\in \Mat_{g\times g}(\Z)$ such that
      after replacing~$\sigma$ by~$\sigma\TrigSp(S)$, the matrix~$\delta$
      satisfies the following property: for every index $1\leq j\leq g$, the
      $j$th line of~$\delta$ consists of coefficients of absolute value at
      most~$\gamma_j/2$.
    \item Swap~$\gamma$ and~$\delta$ by multiplying~$\sigma$ on the right
      by~$J_g$.
    \item Recompute the Smith normal form of~$\gamma$ as in step~(\ref{step:smith}).
    \end{enumerate}
  \item \label{step:hnf} At this step,~$\gamma$ is in Smith normal form
    and~$r = \rank(\gamma) < g$. Consider the matrix~$\delta'$ consisting of
    the last~$g-r$ lines of~$\delta$, of size $(g-r)\times g$. Put~$\delta'$ in
    Hermite normal form using column operations: since~$\det\sigma = 1$, we
    find a matrix $U'\in \GL_g(\Z)$ such that
    % \begin{displaymath}
    $ \delta'U' = \begin{pmatrix} 0 & I_{g-r} \end{pmatrix}.$
    % \end{displaymath}
    Multiply~$\sigma$ on the right by~$\DiagSp((U^{-1})^T)$, so that the last
    $g-r$ lines of~$\sigma$ coincide with those of~$I_{2g}$.
  \item \label{step:decompose-rec} Let~$\alpha',\beta',\gamma',\delta'$ be the
    upper left $r\times r$ blocks of~$\alpha,\beta,\gamma,\delta$
    respectively. The $2r\times 2r$ matrix~$\sigma'$ with $r\times r$
    blocks~$\alpha',\beta',\gamma',\delta'$ is symplectic, and moreover
    $\embedsp{g}{1,\ldots,r}(\sigma')^{-1} \sigma$ is block upper-triangular,
    i.e.~of the form $\DiagSp(U)\TrigSp(S)$ for some matrices $U,S$.
    Reorganizing the matrices of steps \eqref{step:smith}--\eqref{step:hnf}, we
    have decomposed the original~$\sigma$ as a product of the form
    \begin{align*}
      &\DiagSp(U_1) \TrigSp(S_1)\, \embedsp{g}{1,\ldots,r}(\sigma') \DiagSp(U_2) \cdot\\
      &\qquad\qquad\TrigSp(S_2) J_g \TrigSp(S_3) J_g \cdots J_g\TrigSp(S_r) \cdot (I_{2g} \text{ or } J_g)
    \end{align*}
    for certain $g\times g$ matrices~$U_1,U_2,S_1,S_2,S_3,\ldots,S_r$.
  \item Call \cref{algo:decomposition} recursively to decompose~$\sigma'$ into
    elementary matrices. Their images under~$\embedsp{g}{1,\ldots,r}$ are again
    elementary, so the algorithm ends.
  \end{enumerate}
\end{algorithm}

\begin{prop}
  \label{prop:decomposition} \Cref{algo:decomposition} terminates and is
  correct.
\end{prop}

\begin{proof}
  As in the proof of~\cite[Lemma 15 p.\,45]{igusaThetaFunctions1972}, the loop
  of step~(\ref{step:smith-loop}) terminates because $\abs{\gamma}$ is divided
  by~$2$ or more each time through. At the beginning of
  step~\eqref{step:decompose-rec}, since~$\sigma$ is symplectic, we have
  \begin{displaymath}
    \sigma =
    \begin{pmatrix}
      \alpha' & 0 & \beta' & \star \\
          \star   & I_{g-r}  & \star & \star & \\
      \gamma' & 0 & \delta' & \star \\
      0 & 0 & 0 & I_{g-r}
    \end{pmatrix},
  \end{displaymath}
  (see e.g.~\cite[Lemma 14 p.\,45]{igusaThetaFunctions1972}), so~$\sigma'$
  indeed satisfies the required properties. Finally, the matrices can be
  reorganized as claimed in step~\eqref{step:decompose-rec}, because matrices
  of the form $\DiagSp(U)$ ``commute'' with both~$J_g$ (at the cost of
  changing~$U$) and~$\TrigSp(S)$ (at the cost of changing~$S$).
\end{proof}

It is quite common that~$\gamma_1 = \ldots = \gamma_{g-1} = 1$ after
step~\eqref{step:smith}: for instance, this happens whenever~$\det(\gamma)$ is
nonzero and squarefree. Then after one pass through
step~\eqref{step:smith-loop}, $\rank(\gamma) = 1$, so~$\sigma'$ is
elementary. In this case, we obtain a decomposition of~$\sigma$ as a product of
at most~6 elementary matrices, which is less than what \cref{algo:siegel-red}
typically provides. An interesting question would be to provide an upper bound
on the length of the decomposition solely in terms of~$\abs{\sigma}$.

% \begin{rem} FLINT has a built-in function to compute Smith normal forms, but it
%   does not provide the base change matrices that we need in step~(1). Instead,
%   we obtain them by computing a sequence of Hermite normal forms with base
%   change, with transpositions in between.
% \end{rem}

\section{Summation algorithms}
\label{sec:summation}

In this section, we focus on summation algorithms to
evaluate~$\thetatilde_{0,b}(z,\tau)$, and possibly the partial derivatives
of~$\theta_{0,b}$. Fixing $a=0$ is convenient as it ensures that we sum
over~$n\in \Z^g$; for other values of~$a$, we shift~$z$ by~$\tau\tfrac a2$ and
use equation~\eqref{eq:theta-ab}. We separate summation algorithms into three
subtasks:
\begin{enumerate}
\item Compute a suitable ellipsoid radius~$R$;
\item Compute the ellipsoid $\eld{v}{R}{\tau} \cap \Z^g$ where $v$ is defined
  as in~§\ref{subsec:notation};
\item Compute the associated partial sum of the theta
  series~\eqref{eq:theta-ab}.
\end{enumerate}

First, we introduce new upper bounds on the tail of the theta
series~\eqref{eq:theta-00} that we will use to compute~$R$
(§\ref{subsec:bound-tail} and~§\ref{subsec:tailsum-proof}). Next, we describe
provably correct and efficient algorithms using interval arithmetic for the
three subtasks (§\ref{subsec:compute-R} to~§\ref{subsec:sum}). We also give
full complexity proofs when~$(z,\tau)\in \red_g$ is exact, where~$\red_g$ is
the reduced domain as in \cref{def:reduced}. We have two main use cases in
mind:
\begin{itemize}
\item When using the summation algorithm on its own at moderate to large
  precisions, we manipulate ellipsoid containing many points $n\in \Z^g$. We
  wish to compute ellipsoids that are as tight as possible, and compute partial
  sums of the theta series as efficiently as possible. Our main focus is
  $g\geq 2$: when $g=1$, the state of the art on summation
  is~\cite{engeShortAdditionSequences2018}, and our implementation relies on
  existing FLINT code associated to that paper. We also consider partial
  derivatives of theta functions.
\item When using summation in the context of the quasi-linear algorithm, we
  assume that~$(z,\tau)\in \red_g$ is exact, and we are interested in low
  precisions relatively to the eigenvalues of~$\im(\tau)$. In other
  words,~$\Dist_\tau(v,\Z^g)^2$ might be large compared to the chosen
  precision, and we want to take this into account.
\end{itemize}

\subsection{Upper bounds on the tail of the series.}
\label{subsec:bound-tail}

For any~$R\in \R_{\geq 0}$, by~\eqref{eq:term-modulus} and the triangle
inequality, we have
\begin{equation}
  \label{eq:triangular}
  \begin{aligned}
    &\abs[\bigg]{\theta_{0,b}(z,\tau) - \sum_{n\in \Z^g\cap \,\eld{v}{R}{\tau}} \e\paren[\big]{n^T\tau n + 2n^T (z + \tfrac b2)}} \\
    &\hspace{5cm}\leq \exp(\pi y^T Y^{-1} y) \sum_{n\in \Z^g \setminus \eld{v}{R}{\tau}} \exp\paren[\big]{-\norm{n - v}_\tau^2}.
  \end{aligned}
\end{equation}
In order to evaluate~$\thetatilde_{0,b}$ to a given absolute precision, we wish
to control the sum on the right hand side.

One can give a similar-looking upper bound for the tail of the series defining
partial derivatives of~$\theta_{0,b}$ with respect to~$z$. We will then also be
able to also evaluate partial derivatives with respect to~$\tau$ thanks to the
heat equation: if $1\leq j,k\leq g$, then
\begin{equation}
  \label{eq:heat}
  \frac{\partial \theta_{0,b}}{\partial \tau_{j,k}} = \frac{1}{2\pi i (1 + \delta_{j,k})} \frac{\partial^2 \theta_{0,b}}{\partial z_j \partial z_k}.
\end{equation}
Let $\nu = (\nu_1,\ldots,\nu_g)$ be a tuple of nonnegative integers, and write
$\abs{\nu} := \nu_1+\cdots + \nu_g$.  Differentiating the series~\eqref{eq:theta-ab},
we find that for all~$(z,\tau)\in \C^g\times\Half_g$,
\begin{equation}
  \label{eq:partial-series}
  % \frac{\partial^{\abs{k}}\theta_{0,b}}{\partial z^k}(z,\tau)
  \frac{\partial^{\abs{\nu}} \theta_{0,b}}{\partial z_1^{\nu_1}\cdots\partial z_{g}^{\nu_g}}(z,\tau)
  = (2 \pi i)^{\abs{\nu}}
  \sum_{n \in \Z^g} n_1^{\nu_1}\cdots n_{g}^{\nu_{g}}
  \e\paren[\big]{n^T\tau n + 2n^T(z + \tfrac b2)}.
\end{equation}
For brevity, we also denote this partial derivative by
$\partial^\nu \theta_{0,b}(z,\tau)$. We recall the notations
$C, \norm{\placeholder}_\tau$ and $\norm{\placeholder}_\infty$
from~§\ref{subsec:notation}.

\begin{lem}
  \label{lem:bound-derivatives}
  With the above notation, for every $n\in \Z^g$, we have
  \begin{displaymath}
    \abs[\big]{n_1^{\nu_1}\cdots n_g^{\nu_g}} \leq
    \paren[\big]{\norm{C^{-1}}_\infty \,\norm{n - v}_\tau
      + \norm{v}_\infty}^{\abs{\nu}}.
  \end{displaymath}
\end{lem}

\begin{proof}
  We have
  $\abs{n_1^{\nu_1}\cdots n_{g}^{\nu_{g}}} \leq \norm{n}_\infty^{\abs{\nu}}
  \leq \paren[\big]{\norm{n - v}_\infty + \norm{v}_\infty}^{\abs{\nu}}$, and
  \begin{displaymath}
    \norm{n-v}_\infty \leq \norm{C^{-1}}_\infty \norm{C (n-v)}_\infty \leq \norm{C^{-1}}_\infty \norm{C(n-v)}_2
    = \norm{C^{-1}}_\infty \norm{n-v}_\tau. \qedhere
  \end{displaymath}
\end{proof}

The upper bound in \cref{lem:bound-derivatives} can be expanded using the
binomial formula. This leads us to define the tail sums
\begin{equation}
  \label{eq:tailsum}
  \tailsum{v}{R}{p}{\tau} := \sum_{n\in \Z^g \setminus \eld{v}{R}{\tau}} \norm{n - v}_\tau^p \exp\paren[\big]{-\norm{n-v}_\tau^2}
\end{equation}
where~$v\in \R^g$, $R\in \R_{\geq 0}$, $p\in \Z_{\geq 0}$ (or even
$p\in \R_{\geq 0}$), and $\tau\in \Half_g$. We directly obtain:

\begin{prop}
  \label{prop:theta-tail} With the above notation, we have
  \begin{align*}
    &\abs[\bigg]{\frac{\partial^{\abs{\nu}} \theta_{0,b}}{\partial z_1^{\nu_1}\cdots\partial z_{g}^{\nu_g}}(z,\tau)
      - (2 \pi i)^{\abs{\nu}} \sum_{n\in \Z^g \cap \eld{v}{R}{\tau}} n_1^{\nu_1}\cdots n_{g}^{\nu_{g}}
      \e\paren[\big]{n^T\tau n + 2n^T(z + \tfrac b2)} } \\
    &\qquad\qquad \leq (2\pi)^{\abs{\nu}} \exp(\pi y^TY^{-1}y) \sum_{p=0}^{\abs{\nu}} \binom{\abs{\nu}}{p} \norm{C^{-1}}_\infty^p \norm{v}_\infty^{\abs{\nu}-p} \tailsum{v}{R}{p}{\tau}.
  \end{align*}
\end{prop}

Thus, our original problem of evaluating~$\theta_{0,b}$ and its derivatives
with a certified error bound reduces to providing upper bounds on the
sums~$\tailsum{v}{R}{p}{\tau}$.

Such an upper bound is obtained in~\cite{deconinckComputingRiemannTheta2004}
using the subharmonicity of the function
$x\mapsto \norm{x}_2^p \exp \paren[\big]{-\norm{x}_2^2}$ when~$\norm{x}_2$ is
sufficiently large. It involves the upper incomplete Gamma function, denoted
by
\begin{equation}
  \label{eq:incomplete-gamma}
  \Gamma(s, x) = \int_x^\infty t^{s-1}e^{-t} \,dt.
\end{equation}

\begin{thm}[{\cite[Lemma 2]{deconinckComputingRiemannTheta2004}}]
  \label{thm:summation-bound-old}
  With the above notation, let
  \begin{displaymath}
    \rho = \min_{n\in \Z^g\setminus\{0\}} \norm{n}_\tau
  \end{displaymath}
  be the minimal distance between distinct points of~$\Z^g$
  for~$\norm{\cdot}_\tau$. Assume that
  \begin{displaymath}
    \paren[\big]{R - \frac{\rho}{2}}^2 > \frac 14\paren[\big]{g + 2p + \sqrt{g^2 + 8p}}.
  \end{displaymath}
  Then we have
  \begin{displaymath}
    \tailsum{v}{R}{p}{\tau}
    %\sum_{n\in \Z^g\setminus \eld{v}{R}{\tau}} \norm{n - v}_\tau^p \exp\paren[\big]{-\norm{n-v}_\tau^2}
    \leq \frac{g\, 2^{g-1}}{\rho^{\,g}} \Gamma\paren[\Big]{\frac{g + p}{2},
      \paren[\big]{R - \frac{\rho}{2}}^2}.
  \end{displaymath}
\end{thm}

Using \cref{thm:summation-bound-old} in practice requires determining (upper
and lower bounds on)~$\rho$ first, which is not so straightforward: in our
implementation, we compute it with a variant of the algorithm to compute
distances that we describe later in~§\ref{subsec:distances}. Furthermore, the
result says nothing when~$R$ is small, and the offset of~$\rho/2$ in the second
argument of~$\Gamma$ is a bit unsatisfactory: we ought to obtain an upper bound
whose main factor is~$\exp(-R^2)$. We propose the following alternative upper
bound.

\begin{thm}
  \label{thm:summation-bound-new}
  With the above notation, for every $A\in \R_{\geq 0}$ such that $A+R^2\geq p$, we have
  \begin{align*}
    &\sum_{n\in  \Z^g\setminus \eld{v}{R}{\tau}} \paren[\big]{A + \norm{n - v}_\tau^2}^{p/2}
    \exp\paren[\big]{-\norm{n - v}_\tau^2}\\[-1em]
    &\hspace{2.5cm}
      \leq \paren[\Big]{1 + \sqrt{\frac 8\pi}\,} \max\{2,R\}^{g-1} (A + R^2)^{p/2} \exp(- R^2) \prod_{j=1}^{g}
      \paren[\Big]{1 + \frac{\sqrt{2\pi}}{c_j}},
  \end{align*}
  where~$c_1,\ldots,c_g$ are defined as in~§\ref{subsec:notation}. In
  particular, if~$R^2\geq p$, then
  \begin{displaymath}
    \tailsum{v}{R}{p}{\tau}
    \leq \paren[\Big]{1 + \sqrt{\frac 8\pi}\,} \max\{2,R\}^{g-1} R^p \exp(- R^2) \prod_{j=1}^{g} \paren[\Big]{1 + \frac{\sqrt{2\pi}}{c_j}}.
  \end{displaymath}
\end{thm}

We postpone the proof of \cref{thm:summation-bound-new} to the next subsection.

In \cref{fig:cmp-bounds}, we compare the upper bounds of
\cref{thm:summation-bound-old,thm:summation-bound-new} when $\im(\tau)$ is the
identity matrix~$I_g$ as in \cite[Table 1]{deconinckComputingRiemannTheta2004}
(and thus $\rho = 1$), for certain values of $g,p$ and~$R$. We include in the
table the tighter but heuristic
estimate~\cite[eq.\,(15)]{deconinckComputingRiemannTheta2004}. In this range,
the new bound is superior even to the heuristic one, except perhaps for small
values of~$R$. We have not yet attempted to make a systematic comparison
between \cref{thm:summation-bound-old,thm:summation-bound-new} for a
``typical'' reduced matrix $\tau\in\Half_g$, especially for larger~$g$. In any
case, our implementation computes the lower bounds provided by both
\cref{thm:summation-bound-old,thm:summation-bound-new} and chooses the smallest
one.

\begin{figure}[ht]
  \centering
  \begin{tabular}{c|c}
    $g = 2$, $p = 0$ & $g = 6$, $p = 4$ \\[6pt]

    \begin{tabular}{cccc}
      $R$ & Thm.~\ref{thm:summation-bound-old}
      & Heuristic & Thm.~\ref{thm:summation-bound-new} \\\hline
      2 & -- & -- & $ 1.2 \cdot 10^0 $ \\
      3 & -- & -- & %$ 0.012 $
      $1.2\cdot 10^{-2}$ \\
      4 & %$ 0.000019 $
          $1.9\cdot 10^{-5}$ & %$ 0.000015 $
                               $1.5\cdot 10^{-5}$ & %$ 0.000014 $
      $1.4\cdot 10^{-5}$ \\
      5 & $ 6.4 \cdot 10^{-9} $ & $ 5.0 \cdot 10^{-9} $ & $ 2.2 \cdot 10^{-9} $ \\
      6 & $ 2.9 \cdot 10^{-13} $ & $ 2.3 \cdot 10^{-13} $ & $ 4.4 \cdot 10^{-14} $ \\
      7 & $ 1.8 \cdot 10^{-18} $ & $ 1.4 \cdot 10^{-18} $ & $ 1.2 \cdot 10^{-19} $ \\
      8 & $ 1.5 \cdot 10^{-24} $ & $ 1.2 \cdot 10^{-24} $ & $ 4.1 \cdot 10^{-26} $ \\
      9 & $ 1.7 \cdot 10^{-31} $ & $ 1.3 \cdot 10^{-31} $ & $ 1.9 \cdot 10^{-33} $ \\
      10 & $ 2.6 \cdot 10^{-39} $ & $ 2.0 \cdot 10^{-39} $ & $ 1.2 \cdot 10^{-41} $ \\
      11 & $ 5.2 \cdot 10^{-48} $ & $ 4.1 \cdot 10^{-48} $ & $ 9.9 \cdot 10^{-51} $ \\
      12 & $ 1.5 \cdot 10^{-57} $ & $ 1.2 \cdot 10^{-57} $ & $ 1.1 \cdot 10^{-60} $ \\
      13 & $ 5.6 \cdot 10^{-68} $ & $ 4.4 \cdot 10^{-68} $ & $ 1.7 \cdot 10^{-71} $ \\
      14 & $ 2.8 \cdot 10^{-79} $ & $ 2.2 \cdot 10^{-79} $ & $ 3.4 \cdot 10^{-83} $ \\
      15 & $ 2.0 \cdot 10^{-91} $ & $ 1.5 \cdot 10^{-91} $ & $ 9.2 \cdot 10^{-96} $ \\
      16 & $ 1.8 \cdot 10^{-104} $ & $ 1.4 \cdot 10^{-104} $ & $ 3.4 \cdot 10^{-109} $ \\
    \end{tabular}
    &
    \begin{tabular}{cccc}
      Thm.~\ref{thm:summation-bound-old} & Heuristic & Thm.~\ref{thm:summation-bound-new} \\\hline
      -- & -- & %$ 45000. $
      $4.5\cdot 10^4$ \\
      -- & -- & %$ 12000. $
      $1.2\cdot 10^4$ \\
      -- & -- & %$ 140. $
      $1.4\cdot 10^2$ \\
      -- & -- & %$ 0.13 $
      $1.3\cdot 10^{-1}$ \\
      % $ 0.000013 $
      $1.3\cdot 10^{-5}$
                                       & $ 1.1 \cdot 10^{-6} $ & %$ 0.000011 $
      $1.1\cdot 10^{-5}$ \\
       $ 3.1 \cdot 10^{-10} $ & $ 2.5 \cdot 10^{-11} $ & $ 1.0 \cdot 10^{-10} $ \\
       $ 7.7 \cdot 10^{-16} $ & $ 6.2 \cdot 10^{-17} $ & $ 1.0 \cdot 10^{-16} $ \\
       $ 2.3 \cdot 10^{-22} $ & $ 1.9 \cdot 10^{-23} $ & $ 1.2 \cdot 10^{-23} $ \\
       $ 8.5 \cdot 10^{-30} $ & $ 6.8 \cdot 10^{-31} $ & $ 1.8 \cdot 10^{-31} $ \\
       $ 3.8 \cdot 10^{-38} $ & $ 3.1 \cdot 10^{-39} $ & $ 3.2 \cdot 10^{-40} $ \\
       $ 2.3 \cdot 10^{-47} $ & $ 1.9 \cdot 10^{-48} $ & $ 7.2 \cdot 10^{-50} $ \\
       $ 1.6 \cdot 10^{-57} $ & $ 1.3 \cdot 10^{-58} $ & $ 2.1 \cdot 10^{-60} $ \\
       $ 1.5 \cdot 10^{-68} $ & $ 1.2 \cdot 10^{-69} $ & $ 7.5 \cdot 10^{-72} $ \\
       $ 1.9 \cdot 10^{-80} $ & $ 1.5 \cdot 10^{-81} $ & $ 3.6 \cdot 10^{-84} $ \\
       $ 3.1 \cdot 10^{-93} $ & $ 2.5 \cdot 10^{-94} $ & $ 2.2 \cdot 10^{-97} $ \\
    \end{tabular}
  \end{tabular}
  \caption{Comparing \cref{thm:summation-bound-old,thm:summation-bound-new} with
    $\im(\tau) = I_g$.}
  \label{fig:cmp-bounds}
\end{figure}

\begin{cor}
  \label{cor:theta-tail}
  With the above notation, assuming~$R^2\geq\abs{\nu}$, we have
  \begin{align*}
    &\abs[\bigg]{\frac{\partial^{\abs{\nu}} \theta_{0,b}}{\partial z_1^{\nu_1}\cdots\partial z_{g}^{\nu_g}}(z,\tau)
      - (2 \pi i)^{\abs{\nu}} \sum_{n\in \Z^g \cap \eld{v}{R}{\tau}} n_1^{\nu_1}\cdots n_{g}^{\nu_{g}}
      \e\paren[\big]{n^T\tau n + 2n^T(z + \tfrac b2)} } \\
    &\qquad\qquad \leq (2\pi)^{\abs{\nu}} \paren[\Big]{1 + \sqrt{\frac 8\pi}\,} \prod_{j=1}^{g} \paren[\Big]{1 + \frac{\sqrt{2\pi}}{c_j}} \\
    &\qquad\qquad \qquad \cdot  \exp(\pi y^TY^{-1}y) \max\{2,R\}^{g-1} \exp(- R^2)  \paren[\big]{\norm{C^{-1}}_\infty R + \norm{v}_\infty}^{\abs{\nu}}.
  \end{align*}
\end{cor}

\begin{proof}
  Combine \cref{thm:summation-bound-new} and \cref{prop:theta-tail}.
\end{proof}

In the context of the quasi-linear algorithm, we can replace
\cref{thm:summation-bound-new} by the following upper bound, which we also
prove in the next subsection.

\begin{thm}
  \label{thm:shifted-ubound}
  Let~$\tau\in \Half_g$ be a reduced matrix, and
  let~$v\in \R^g$.  Let~$R\in \R_{\geq 0}$ be such that
  $R^2 = \Dist_\tau(v, \Z^g)^2 + \delta^2$ with~$\delta \in \R_{\geq 0}$. Then
  we have
  \begin{displaymath}
    \tailsum{v}{R}{0}{\tau} \leq B(g) \max\{2,\delta\}^{g-1} \exp(-R^2)
  \end{displaymath}
  where~$B(g) = 2^{10(1 + g\log g)}$. In particular, we have
  for every $z\in \C^g$
  \begin{displaymath}
    \abs[\big]{\thetatilde_{a,b}(z,\tau)} \leq B(g) \exp\paren[\big]{-\Dist_\tau(v,\Z^g + \tfrac a2)^2}.
  \end{displaymath}
\end{thm}

In \cref{thm:shifted-ubound}, the constant~$B(g)$ is certainly larger than what
is true in reality, but the right hand side involves~$\delta^{g-1}$ instead of
$R^{g-1}$. This fact will be convenient in the theoretical analysis of the fast
algorithm in~§\ref{sec:ql}, but has little impact in practice. We did not
implement this variant.

\subsection{Proofs}
\label{subsec:tailsum-proof}

We prove \cref{thm:summation-bound-new} by induction on~$g$. When $g=1$
and during the induction process, we rely on the following two lemmas.

\begin{lem}
  \label{lem:dim1-bound-full}
  Let~$v\in \R$, $c \in \R_{>0}$, and let $p, A \in\R_{\geq 0}$.  If
  $A\geq p$, then
  \begin{displaymath}
    \sum_{n\in \Z} \paren[\big]{A+ c^2\abs{n - v}^2}^{p/2} \exp\paren[\big]{-c^2\abs{n - v}^2}
    \leq \paren[\Big]{1 + \frac{\sqrt{2\pi}}{c}} A^{p/2}.
  \end{displaymath}
\end{lem}
\begin{proof}
  For every~$n\in \Z$, by concavity of the logarithm function, we have
  \begin{displaymath}
    \paren[\big]{A + c^2\abs{n-v}^2}^{p/2}
    \leq A^{p/2} \paren[\Big]{1 + \frac{c^2\abs{n-v}^2}{p}}^{p/2}
    \leq A^{p/2} \exp \paren[\Big]{\frac{c^2}{2} \abs{n-v}^2}.
  \end{displaymath}
  The Fourier transform of $x\mapsto \exp(-x^2)$ is another Gaussian function,
  so in particular takes values in~$\R_{\geq 0}$. Applying the Poisson formula, we
  see that
  \begin{displaymath}
    \sum_{n\in \Z} \exp\paren[\Big]{-\frac{c^2}{2}\abs{n-v}^2}
    \leq \sum_{n\in \Z} \exp\paren[\big]{-\frac{c^2}{2} n^2}.
  \end{displaymath}
  We compare the latter sum to an integral on $(-\infty,0]$ and $[0,+\infty)$
  and obtain
  \begin{displaymath}
    \sum_{n\in \Z} \exp\paren[\big]{-\frac{c^2}{2}n^2}
    \leq 1 + \int_{-\infty}^\infty \exp \paren[\big]{-\frac{c^2}{2}t^2}\,dt
    = 1 + \frac{\sqrt{2\pi}}{c}. \qedhere
  \end{displaymath}
\end{proof}

\begin{lem}
  \label{lem:dim1-bound}
  Let~$v\in \R$, $c \in \R_{>0}$, and let $p, A, R \in\R_{\geq 0}$.  If
  $A + R^2\geq p$, then
  \begin{displaymath}
    \sum_{\substack{n\in \Z\\ c\abs{n-v}\geq R}}
    \paren[\big]{A+c^2\abs{n - v}^2}^{p/2} \exp\paren[\big]{-c^2\abs{n - v}^2}
    \leq \paren[\Big]{2 + \frac{\sqrt{2\pi}}{c}} (A+R^2)^{p/2} \exp(-R^2).
  \end{displaymath}
\end{lem}

\begin{proof}
  For every $n\in \Z$ such that $c\abs{n-v}\geq R$, we have
  \begin{displaymath}
    \begin{aligned}
      \paren[\big]{A + c^2\abs{n-v}^2}^{p/2}
      &= (A+R^2)^{p/2}
      \paren[\Big]{1 + \frac{c^2\abs{n-v}^2 - R^2}{p}}^{p/2} \\
      &\leq (A+R^2)^{p/2} \exp\paren[\Big]{\tfrac12 \paren[\big]{c^2 \abs{n-v}^2 - R^2}}.
    \end{aligned}
  \end{displaymath}
  We can compare the resulting sum to an integral on both intervals
  $(-\infty, v - R/c]$ and $[v + R/c,+\infty)$ and obtain:
  \begin{align*}
    \sum_{\substack{n\in \Z\\ c\abs{n-v}\geq R}} \exp\paren[\Big]{-\tfrac12
    \paren[\big]{c^2\abs{n-v}^2 - R^2}}
    &\leq 2 + 2\int_{R/c}^{+\infty} \exp \paren[\big]{-\tfrac12
      (c^2 t^2 - R^2)}\,dt\\[-12pt]
    &\leq 2 + 2\int_{0}^{+\infty} \exp \paren[\big]{-\frac{c^2}{2} t^2}\,dt
      = 2 + \frac{\sqrt{2\pi}}{c}. \hspace{7mm}\qedhere
  \end{align*}
\end{proof}

Next, we settle the case~$R = 0$ in \cref{thm:summation-bound-new}.

\begin{lem}
  \label{lem:bound-full}
  Keep the notation of \cref{thm:summation-bound-new}, but assume that $A\geq p$. Then
  \begin{displaymath}
    \sum_{n\in \Z^g} \paren[\big]{A + \norm{n-v}_\tau^2}^{p/2} \exp\paren[\big]{-\norm{n-v}_\tau^2}
    \leq  A^{p/2} \prod_{j=1}^{g} \paren[\Big]{1 + \frac{\sqrt{2\pi}}{c_j}}.
  \end{displaymath}
\end{lem}

\begin{proof}
  We proceed by induction on $g$. The case $g=1$ is covered by
  \cref{lem:dim1-bound-full}, so we suppose that~$g\geq
  2$. Let~$C',\Xi,\tau',v'$ be as in~\eqref{eq:C-split}
  and~\eqref{eq:pyth}. For $n_g\in \Z$, we define
  \begin{displaymath}
    v''(n_g) = v' + C'^{-1}\Xi(v_g - n_g).
  \end{displaymath}

  We now regroup terms in the main sum according to the value of~$n_{g}$, and
  use~\eqref{eq:pyth} to obtain
  \begin{displaymath}
    \begin{aligned}
      \sum_{n\in \Z^g} &(A + \norm{n-v}_\tau^2)^{p/2} \exp\paren[\big]{-\norm{n-v}_\tau^2}\\
      &= \sum_{n_{g}\in \Z} \exp \paren[\big]{-c_{g}^2\abs{n_{g}-v_{g}}^2}\\
      &\qquad\quad \cdot
        \sum_{n'\in \Z^{g-1}} \paren[\big]{A + c_{g}^2 \abs{n_{g} - v_{g}}^2 + \norm{n' - v''(n_g)}_{\tau'}^2}^{p/2}
        \exp\paren[\big]{-\norm{n' - v''(n_g)}_{\tau'}^2}\\
      &\leq \sum_{n_{g}\in \Z} \exp \paren[\big]{-c_{g}^2 \abs{n_{g}-v_{g}}^2} \cdot
        \paren[\big]{A + c_{g}^2 \abs{n_{g} - v_{g}}^2}^{p/2}
        \prod_{j=1}^{g-1} \paren[\Big]{1 + \frac{\sqrt{2\pi}}{c_j}}\\
      &\leq A^{p/2} \prod_{j=1}^{g} \paren[\Big]{1 + \frac{\sqrt{2\pi}}{c_j}}.
    \end{aligned}
  \end{displaymath}
  The fourth line inequality uses the induction hypothesis for~$g-1$, while the
  fifth line uses \cref{lem:dim1-bound-full}.
\end{proof}

\begin{proof}[Proof of \cref{thm:summation-bound-new}] We proceed by induction
  on~$g$. The case $g=1$ is covered by \cref{lem:dim1-bound}, so we suppose
  that $g\geq 2$. As in the proof of \cref{lem:bound-full}, we split the main
  sum according to the value of~$n_g$. For each~$m\in \Z$, we introduce the
  notation
  \begin{displaymath}
    \Psi(m) = \sum_{\substack{n\in \Z^g \setminus \eld{v}{R}{\tau}\\ n_g = m}}
    \paren[\big]{A + \norm{n - v}_\tau^2}^{p/2} \exp\paren[\big]{-\norm{n - v}_\tau^2}.
  \end{displaymath}
  If $c_g\abs{m - v_g} > R$, then by \cref{lem:bound-full}, we have
  \begin{displaymath}
    \Psi(m) \leq \paren[\big]{A + c_g^2 \norm{m - v_g}^2}^{p/2} \exp\paren[\big]{- c_g^2\abs{m - v_g}^2}
    \prod_{j=1}^{g-1} \paren[\Big]{1 + \frac{\sqrt{2\pi}}{c_j}},
  \end{displaymath}
  so by \cref{lem:dim1-bound},
  \begin{displaymath}
    \sum_{\substack{m\in \Z\\ c_g\abs{m - v_g} > R}} \Psi(m) \leq 2 (A + R^2)^{p/2} \exp(-R^2)
    \prod_{j=1}^{g} \paren[\Big]{1 + \frac{\sqrt{2\pi}}{c_j}}.
  \end{displaymath}
  If $c_g \abs{m - v_g} \leq R$, then by the induction hypothesis,
  \begin{displaymath}
    \Psi(m) \leq K \max\braces[\big]{4, R^2 - c_g^2 \abs{m-v_g}^2}^{(g-2)/2}
    (A + R^2)^{p/2} \exp(-R^2) \prod_{j=1}^{g-1} \paren[\Big]{1 + \frac{\sqrt{2\pi}}{c_j}},
  \end{displaymath}
  where~$K = 2$ if~$g = 2$, and $K = 1 + \sqrt{8/\pi}$
  otherwise.  Since there are at most~$1 + 2R/c_g$ such points~$m$,
  and~$R\geq 2$, we have
  \begin{displaymath}
    \sum_{\substack{m\in \Z\\ c_g\abs{m - v_g}\leq R}} \max\braces[\big]{4, R^2 - c_g^2 \abs{m-v_g}^2}^{(g-2)/2}
      \leq \max\{2, R\}^{g-1} \paren[\Big]{\frac 12 + \frac{2}{c_g}}.
  \end{displaymath}
  Therefore,
  \begin{displaymath}
    \sum_{\substack{m\in \Z\\ c_g\abs{m - v_g} \leq R}} \Psi(m)
    \leq K \sqrt{\frac{2}{\pi}}\, \max\{2,R\}^{g-1} (A+R^2)^c \exp(-R^2)
    \prod_{j=1}^{g} \paren[\Big]{1 + \frac{\sqrt{2\pi}}{c_j}}.
  \end{displaymath}
  The induction succeeds since
  $K \sqrt{2/\pi} + 2^{-g+2}\leq 1 + \sqrt{8/\pi}$.
\end{proof}

\begin{proof}[Proof of \cref{thm:shifted-ubound}]
  We again proceed by induction on~$g$. The case $g=1$ is covered by
  \cref{lem:dim1-bound} with~$p=0$: as~$c^2\geq \pi\sqrt{3}/2$ because~$\tau$ is
  reduced, we have
  \begin{displaymath}
    \tailsum{v}{R}{0}{\tau} \leq 4 \exp(-R^2).
  \end{displaymath}
  Now assume $g\geq 2$, and assume that \cref{thm:shifted-ubound} holds in
  dimension~$g-1$. As in the proof of \cref{thm:summation-bound-new}, we write
  for $m\in \Z$
  \begin{displaymath}
    \Psi(m) = \sum_{\substack{n\in \Z^g\setminus\eld{v}{R}{\tau}\\n_g = m}} \exp\paren[\big]{- \norm{n-v}_\tau^2}.
  \end{displaymath}
  By \cref{lem:submatrix-red}, the upper left~$(g-1)\times(g-1)$ submatrix
  of~$\tau$ is reduced in~$\Half_{g-1}$.  By the induction
  hypothesis, for each~$m$ such that $c_g\abs{m - v_g} > R$, we have
  \begin{displaymath}
    \Psi(m) \leq B(g-1) \exp\paren[\big]{-c_g^2 \abs{m - v_g}^2}.
  \end{displaymath}
  If $\Dist_\tau(v, \Z^g) < c_g\abs{m - v_g} \leq R$, then
  $R^2 - c_g\abs{m-v_g}^2 \leq \delta^2$, so
  \begin{displaymath}
    \Psi(m) \leq B(g-1) \max\{2, \delta\}^{g-2} \exp(-R^2).
  \end{displaymath}
  If~$c_g\abs{m - v_g} \leq \Dist_\tau(v,\Z^g)$, then as we rewrite~$\Psi(m)$
  as a sum on $\Z^{g-1}\setminus \mathcal{E}'$ where~$\mathcal{E'}$ is an
  ellipsoid in dimension~$g-1$, the distance~$D$ between the center
  of~$\mathcal{E}'$ and~$\Z^{g-1}$ satisfies
  $D^2 \geq \Dist_\tau(v,\Z^g)^2 - c_g\abs{m - v_g}^2$. Therefore, we also have
  \begin{displaymath}
    \Psi(m) \leq B(g-1) \max\{2,\delta\}^{g-2} \exp(-R^2).
  \end{displaymath}
  Since~$R\leq \Dist_\tau(v,\Z^g) + \delta$, at most $2 + 2\delta/c_g$
  values of~$m\in \Z$ fall in the second case. Further, at most
  $1 + 2\Dist_\tau(v,\Z^g)/c_g$ values fall in the third case. Overall,
  \begin{align*}
    \tailsum{v}{R}{0}{\tau} &\leq B(g-1) \paren[\Big]{3 + \frac{2 \Dist_\tau(v,\Z^g)}{c_g}
                              + \frac{2\delta}{c_g}}\max\{2, \delta\}^{g-2}\exp(-R^2) \\
    &\quad\, + B(g-1) \paren[\Big]{2 + \frac{\sqrt{2\pi}}{c_g}} \exp(-R^2).
  \end{align*}
  Here we have used \cref{lem:dim1-bound} to bound the total contribution from
  values of~$m$ such that $c_g\abs{m-v_g} > R$. By \cref{prop:hkz-bound,cor:hkz-dist}, we have
  \begin{displaymath}
    c_g \geq \frac{1.649}{g} \quad \text{and}\quad \frac{\Dist(v,\Z^g)}{c_g}\leq \frac12 g^{1+(\log g)/2}.
  \end{displaymath}
  Therefore, we obtain the claimed result provided that
  \begin{displaymath}
    B(g) \geq B(g-1) \paren[\Big]{\frac32 + \frac12 g^{1+(\log g)/2} + \frac{2}{1.649}g
      + 2 + \frac{\sqrt{2\pi}}{1.649}g}.
  \end{displaymath}
  We let the reader check that this inequality holds
  when~$B(g) = 2^{10(1+g\log g)}$.
\end{proof}

\subsection{Computing a suitable ellipsoid radius}
\label{subsec:compute-R}

Given the shape of \cref{thm:summation-bound-new,thm:shifted-ubound}, we
introduce the following algorithm involving the family of functions
\begin{equation}
  \label{eq:fun-f}
  \begin{matrix}
    f_p : & \R_{>0} &\to &\R \\
    & x & \mapsto & x - \frac p2 \log x.
  \end{matrix}
\end{equation}

\begin{algorithm}
  \label{algo:R}
  \algoinput{$p\in \Z_{\geq 0}$ and $\eps\in \R_{>0}$.}
  \algooutput{$R\in \R_{\geq 0}$ such that $R\geq 2$ and
    $R^p \exp(-R^2)\leq \eps$.}

  \begin{enumerate}
  \item \label{step:R-init} Let~$x_0 = p/2$ and
    \begin{displaymath}
      x_1 = \max\braces{4, p, 2 \paren[\big]{\log(1/\eps)-f_p(x_0)} + p\log 2}.
    \end{displaymath}
  \item \label{step:R-newton} \textbf{Newton iteration.} Set~$x = x_1$, then do
    a few (e.g.~10) times:
    \begin{enumerate}
    \item Replace $x$ by $ x - \frac{1}{f_p'(x)} \paren[\big]{f_p(x) - \log(1/\eps)}$;
    \item Replace $x$ by an exact upper bound on $\max\{x, x_0\}$.
    \end{enumerate}
  \item \label{step:R-final} Output $R$, an exact upper bound on
    $\max\{2,\sqrt{x}\}$.
  \end{enumerate}
\end{algorithm}

\begin{prop}
  \label{prop:algo-R-correct}
  \Cref{algo:R} is correct, and the output radius~$R$ satisfies
  \begin{displaymath}
    R^2 = O\paren[\big]{\log(1/\eps) + p}.
  \end{displaymath}
\end{prop}

\begin{proof}
  The function $f_p$ is convex, decreasing then increasing, with a minimum at
  $x_0$. We claim that $f_p(x)\geq \log(1/\eps)$, where $x$ is the value
  obtained at the end of step~(2). Indeed, we can directly check that
  $f_p(x_1)\geq \log(1/\eps)$, and the inequality remains true throughout the
  Newton iteration because $f_p$ is convex. Then, by taking logarithms in
  \cref{thm:summation-bound-new}, we see that~$R$ satisfies the required
  property. (If $\sqrt{x} < 2$, then the inequality $x \geq x_0$ ensures that
  $f_p(4) \geq \log(1/\eps)$, so $R=2$ is a valid output.) Finally, $R$ is
  bounded above as claimed because~$x_1$ is.
\end{proof}

In order to compute a suitable ellipsoid radius for the evaluation of
normalized theta values~$\thetatilde_{0,b}(z,\tau)$ up to some fixed absolute
precision, we can use the following algorithm, which comes in two
variants~\textbf{A} and~\textbf{B}. Its validity directly stems from
\cref{thm:summation-bound-new,thm:shifted-ubound}. Note that variant~\textbf{A}
is actually independent of~$z$.

\begin{algorithm}
  \label{algo:compute-R}
  \algoinput{$(z,\tau)\in \C^g\times \Half_g$; a precision~$N\in \Z_{\geq
      2}$. In variant~\textbf{B}, the value $\Dist_\tau(v,\Z^g)^2$.}
  \algooutput{$R\in \R_{\geq 0}$ such that for all $b\in \{0,1\}^g$,
    \begin{displaymath}
      \abs[\Big]{\thetatilde_{0,b}(z,\tau) - \exp(-\pi y^TY^{-1}y) \sum_{n\in \Z^g\cap \eld{v}{R}{\tau}} (-1)^{n^T b} \e(n^T\tau n + 2n^T z)} \leq 2^{-N}.
    \end{displaymath}}
  \begin{enumerate}
  \item In variant~\textbf{A}, compute~$R$ using \cref{algo:R} with~$p = g-1$
    and
    \begin{displaymath}
      \eps = 2^{-N} \paren[\Big]{1 + \sqrt{\frac{8}{\pi}}}^{-1} \prod_{j=1}^g \paren[\Big]{1+\frac{\sqrt{2\pi}}{c_j}}^{-1}.
    \end{displaymath}
    In variant~\textbf{B}, compute~$R$ using \cref{algo:R} with $p=g-1$ and
    \begin{displaymath}
      \eps = 2^{-N} B(g)^{-1} \exp\paren[\big]{\Dist_\tau(v,\Z^g)^2},
    \end{displaymath}
    then replace~$R$ by an upper bound on $\sqrt{R^2 + \Dist_\tau(v, \Z^g)^2}$.
  \item Output~$R$.
  \end{enumerate}
\end{algorithm}

In the case of partial derivatives, we rely on \cref{cor:theta-tail}
instead. We first check whether
choosing~$R \leq \norm{v}_\infty/\norm{C^{-1}}_\infty$ could work using
\cref{algo:R} with~$p=0$, and
force~$R\geq \norm{v}_\infty/\norm{C^{-1}}_\infty$ otherwise, leading to the
following algorithm.

\begin{algorithm}
  \label{algo:compute-R-der}
  \algoinput{$(z,\tau)\in \C^g\times \Half_g$; a precision~$N\in \Z_{\geq 2}$,
    an integer $B\geq 0$}.  \algooutput{$R\in \R_{\geq 0}$ such that for all
    $b\in \{0,1\}^g$ and all $\nu\in \Z^g$ with nonnegative entries such
    that~$\abs{\nu}\leq B$, we have
    \begin{displaymath}
      \abs[\Big]{\partial^\nu\theta_{0,b}(z,\tau) - (2\pi i)^{\abs{\nu}} \sum_{n\in \Z^g \cap\eld{v}{R}{\tau}} (-1)^{n^T b} n_1^{\nu_1}\cdots n_g^{\nu_{g}} \e(n^T\tau n + 2n^T z)} \leq 2^{-N}.
    \end{displaymath}}
  \begin{enumerate}
  \item Compute~$\norm{C^{-1}}_\infty$ and~$\norm{v}_\infty$.
  \item Compute~$R$ using \cref{algo:R} with~$p = 0$ and $\eps = \max\{1, 2\norm{v}_\infty\}^{-B} \eps'$, where
    \begin{displaymath}
      \eps' = 2^{-N} \exp(-\pi y^T Y^{-1} y) \paren[\Big]{1 + \sqrt{\frac{8}{\pi}}}^{-1} \prod_{j=1}^g \paren[\Big]{1+\frac{\sqrt{2\pi}}{c_j}}^{-1}.
    \end{displaymath}
    If~$\norm{C^{-1}}_\infty R\leq \norm{v}_\infty$, then stop and output~$R$, otherwise continue.
  \item Compute~$R$ using \cref{algo:R} with~$p=B$ and
    $\eps = \max\{1, 2\norm{C^{-1}}_\infty\}^{-B} \eps'$.
  \item Replace $R$ by $\max\{R, \norm{v}_\infty/\norm{C^{-1}}_\infty\}$.
  \item Output~$R$.
  \end{enumerate}
\end{algorithm}

\subsection{Computing a suitable ellipsoid}
\label{subsec:compute-eld}

Once~$R$ is known, in order to compute the ellipsoid
$\eld{v}{R}{\tau}\cap\Z^g$, we use the recursive procedure of~\cite[Remark
p.\,9]{deconinckComputingRiemannTheta2004} (see also~§\ref{subsec:properties})
in the context of interval arithmetic. (Given that~$v$ is not exact in general,
we might actually compute a slight superset of the required ellipsoid.) We
start by determining the possible values of~$n_g$, then for each~$n_g$, the
possible values for~$n_{g-1}$, etc. This organizes the points
in~$\eld{v}{R}{\tau}\cap \Z^g$ in a tree-like structure. Keeping track of this
structure has two advantages. First, not all the points have to be listed,
since we can simply record upper and lower bounds for the possible values
of~$n_1$ in each 1-dimensional layer. Second, we will be able to leverage this
structure to compute partial sums more efficiently in~§\ref{subsec:sum} below.

\begin{defn}
  \label{def:eld}
  Let~$g,d\in \Z_{\geq 1}$ such that~$d\leq g$. We recursively define an
  \emph{ellipsoid data structure} $\E$ of dimension $d$ in ambient dimension $g$
  to consist of:
  \begin{itemize}
  \item If $d=1$, a pair of integers $(M_1^-,M_1^+)$ together with a
    $(g-1)$-tuple of integers $(n_2,\ldots, n_g)$. The set of points
    \emph{contained} in $\E$ is then by definition
    \begin{displaymath}
      \braces[\big]{n = (n_1,\ldots,n_g)\in \Z^g: M_1^-\leq n_1\leq M_1^+}.
    \end{displaymath}
    We refer to $(n_2,\ldots, n_g)$ as the \emph{common coordinates} of~$\E$,
    and to $M_1^-, M_1^+$ as the lower and upper \emph{bounds} of~$\E$.
  \item If $d> 1$, a pair of bounds $(M_d^-, M_d^+)$, a $(g-d)$-tuple of common
    coordinates $(n_{d+1}, \ldots, n_g)$, and for each $M_d^-\leq n_d\leq M_d$,
    an ellipsoid data structure (the \emph{child} of $\E$ attached to $n_d$) of
    dimension $d-1$ in ambient dimension $g$ whose common coordinates are
    $(n_d,\ldots,n_g)$. The set of points \emph{contained} in~$\E$ is then the
    disjoint union of the sets of points contained in children of~$\E$.
  \end{itemize}
\end{defn}

Ellipsoid data structures are purely combinatorial objects: in our
implementation, the type of such a structure is
\texttt{acb\_theta\_eld\_t}. The recursive procedure of~\cite[Remark
p.\,9]{deconinckComputingRiemannTheta2004} can now be formalized as the
following algorithm.  As a byproduct, we additionally obtain for free:
\begin{itemize}
\item the number of points in~$\E$, denoted by~$\abs{\E}$, and
\item for each $1\leq j\leq g$, an upper bound $B_j$ on the values $\abs{n_j}$
  across all $n\in \E$.
\end{itemize}
These quantities will be useful in the summation phase in the next subsection.

\begin{algorithm}
  \label{algo:ellipsoid}

  \algoinput{$\tau\in \Half_g$; $v\in \R^g$; $R^2$, where~$R\in \R_{\geq 0}$.}
  \algooutput{An ellipsoid data structure~$\E$ of dimension~$g$ in ambient
    dimension~$g$ which contains every point in~$\eld{v}{R}{\tau}\cap\Z^g$; the
    quantities~$\abs{\E}$ and $B_j$ for $1\leq j\leq g$.}

  \begin{enumerate}
  \item \label{step:eld-v} Compute~$C, C', \Xi, v'$ and~$\tau'$ as in~\eqref{eq:C-split} and~\eqref{eq:pyth}.
  \item \label{step:eld-bounds} Let~$M_g^-$ be the maximal integer such that
    $M_g^- - 1$ is provably smaller than $v_g-R/c_g$. Similarly, let $M_g^+$ be
    the minimal integer such that $M_g^++1$ is provably greater than
    $v_g + R/c_g$.
  \item If~$g = 1$, then output the $1$-dimensonal ellipsoid with
    bounds~$(M_g^-, M_g^+)$, $\abs{\E} = M_g^+ - M_g^- + 1$ and
    $B_g = \max\{\abs{M_g^-}, \abs{M_g^+}\}$, otherwise continue.
  \item \label{step:eld-rec} For each~$M_g^- \leq n_g \leq M_g^+$,
    \begin{enumerate}
    \item Let~$R'^2$ be an upper bound on $R^2 - c_g^2 \abs{n_g - v_g}^2$.
    \item Let~$w = v' - C'^{-1}\Xi(v_g - n_g)$.
    \item Let~$\E'_{n_g}$ be the ellipsoid data structure in dimension~$g-1$
      produced by \cref{algo:ellipsoid} recursively on~$\tau'$, $w$, and
      $R'^2$. The recursive call also outputs~$\abs{\E'_{n_g}}$ and bounds
      $B'_{j,n_g}$ for $1\leq j\leq g-1$.
    \item Make~$\E'_{n_g}$ into an ellipsoid data structure of dimension~$g-1$
      in ambient dimension~$g$ obtained by appending~$n_g$ to its common
      coordinates, as well as the common coordinates of all its children
      recursively.
    \end{enumerate}
  \item Output:
    \begin{itemize}
    \item $\E$, the ellipsoid data structure in dimension~$g$ with
      bounds~$(M_g^-, M_g^+)$ such that for each $M_g^-\leq n\leq M_g^+$, its
      child attached to~$n_g$ is~$\E'_{n_g}$;
    \item $\abs{\E} = \sum_{n_g} \abs{\E'_{n_g}}$;
    \item $B_j = \max \{B'_{j, n_g}: M_g^-\leq n_g\leq M_g^+\}$ for
      $1\leq j\leq g-1$; and
    \item $B_g = \max\{\abs{M_g^-}, \abs{M_g^+}\}$.
    \end{itemize}
  \end{enumerate}
\end{algorithm}

\begin{prop}
  \label{prop:ellipsoid}
  \Cref{algo:ellipsoid} is correct. If the input corresponds to an exact
  point~$(z,\tau)\in \red_g$, and if~$N\in \Z_{\geq 2}$ denotes the chosen
  (absolute) working precision throughout, then \cref{algo:ellipsoid} outputs
  an ellipsoid data structure~$\E$ included in
  $\eld{v}{\sqrt{R^2 + 2^{-N}g}}{\tau}\cap\Z^g$, and runs in
  time~$2^{O(g\log g)}R^{g-1}\Mul(N + \log R)$. The size of~$\E$ in memory is
  $2^{O(g\log g)} R^{g-1}$. Moreover,~$\abs{\E} = 2^{O(g\log g)} R^g$ and
  $B_j \leq 2 + 2Rg^2$ for each $1\leq j\leq g$.
\end{prop}

\begin{proof}
  The correctness of the algorithm follows from
  equation~\eqref{eq:pyth}. Assume now that~$(z,\tau)\in \red_g$ is exact. When
  computing the radii for recursive calls in step~(\ref{step:eld-rec}.a), the
  $R'^2$ we compute might be increased by~$2^{-N}$, so the total error on~$R^2$
  can be as large as $2^{-N}g$, as claimed. By \cref{prop:eld-width-2}, the
  values of~$n_g$ we manipulate all satisfy~$\abs{n_g}\leq 2 + 2Rg^2$, so each
  elementary operation at absolute precision~$2^{-N}$ costs
  $O\paren[\big]{\Mul(N + \log(Rg)}$ binary operations. By
  \cref{cor:eld-nb-pts}, the total number of points in~$\E$ is at most
  \begin{displaymath}
    \paren[\big]{1 + (R+2^{-N}g)g}^g = 2^{O(g\log g)} R^g.
  \end{displaymath}
  In fact, the number of 1-dimensional layers in~$\E$ is at most
  $\paren[\big]{1 + (R+2^{-N}g)g}^{g-1}$, hence the claimed complexity bound.
\end{proof}

\begin{rem}
  \label{rem:wider-ellipsoid}
  We gather further miscellaneous comments on \cref{algo:ellipsoid}.
  \begin{itemize}
  \item In practice, the cost of computing the ellipsoid data structure is
    negligible compared to that of the summation phase, and we can take~$N=32$,
    say.
  \item When~$(z,\tau)$ is inexact, \cref{algo:ellipsoid} will compute a
    thicker ellipsoid data structure, taking upper or lower bounds on inexact
    quantities at each step as appropriate. The number of extraneous points
    in~$\E$ remains small.
  \item \Cref{algo:ellipsoid} can also compute an ellipsoid data structure that
    is suitable for use in the summation algorithms for several vectors~$z$
    simultaneously and the same~$\tau$, provided that we replace~$v$ by a
    parallelopiped containing all the relevant ellipsoid centers, as
    in~\cite[§5]{deconinckComputingRiemannTheta2004}.  This happens for
    instance when evaluating~$\theta_{a,b}(z,\tau)$ for all
    characteristics~$(a,b)$, and in the context of the quasi-linear algorithm
    (cf.~step~\eqref{step:ql-init} of \cref{algo:ql}). Reducing the vectors~$z$
    as in~§\ref{subsec:z-reduction} ensures that~$v$ is included in~$[-2,2]^g$,
    so the upper bounds of \cref{prop:ellipsoid} remain valid.
  \end{itemize}
\end{rem}

\subsection{An optimized summation process}
\label{subsec:sum}

Suppose that an ellipsoid data structure~$\E$ has been computed by
\cref{algo:ellipsoid}. Our task is now, for each $b\in \{0,1\}^g$, to evaluate
the sum
\begin{equation}
  \label{eq:Sb}
  s_b := \exp(-\pi y^T Y^{-1}y) \sum_{n\in \E} (-1)^{n^T b} \e(n^T\tau n + 2n^T z).
\end{equation}

In order to avoid computing an exponential for each~$n$, we express each term
as a product of the quantities
\begin{equation}
  \label{eq:def-q-w}
  q_{j,k} := \e\paren[\big]{(2 - \delta_{j,k})\tau_{j,k}} \quad \text{and}\quad w_j := \e(2 z_j)
\end{equation}
for all $1\leq j\leq k\leq g$ (remembering that~$\tau$ is symmetric), and their
inverses. There are additional benefits in using these quantities as input to
the summation phase:
\begin{itemize}
\item If the quantities in~\eqref{eq:def-q-w} (and their inverses) are known at
  a given point~$(z,\tau)$, then we can easily derive them at the
  point~$(2z,2\tau)$. This is convenient when applying duplication formulas.
\item If the above quantities are known at a given point~$(z,\tau/2)$, then we
  can easily derive them at points $(z + \tau \tfrac a2,\tau)$ for all
  characteristics~$a$.
\end{itemize}
From v3.3.0 onwards, our FLINT implementation introduces dedicated data
structures to make such manipulations easier.

For the summation itself, one can use the following straightforward algorithm.

\begin{algorithm}
  \label{algo:sum-naive}
  \algoinput{$(z,\tau)\in \C^g\times \Half_g$; a precision~$N\geq 2$; an
    ellipsoid data structure~$\E$ together with the quantities~$\abs{\E}$ and
    $B_j$ for~$1\leq j\leq g$ as in \cref{algo:ellipsoid}.} \algooutput{$s_b$
    to absolute precision~$N$ for every~$b\in \{0,1\}^g$.}
  \begin{enumerate}
  \item Let~$N_0 = N + 2 \log_2(\abs{\E}) + 1$. Let~$s_b = 0$ for each~$b$.
  \item Compute $q_{j,k}$ and~$w_j$ as defined in~\eqref{eq:def-q-w} to
    absolute precision~$N_0$ for all $1\leq j\leq k\leq g$.
  \item \label{step:sum-naive-exp} For each point~$n\in \E$,
    \begin{enumerate}
    \item Compute $u = \e(n^T\tau n + 2 n^T z)$ to absolute precision~$N_0$, as
      \begin{displaymath}
        u = \paren[\bigg]{\prod_{j = 1}^g \prod_{k=j}^g q_{j,k}^{n_j n_k}}\cdot \prod_{j = 1}^{g} w_j^{n_j}.
      \end{displaymath}
    \item For each~$b\in \{0,1\}^g$, add $(-1)^{n^Tb}u$ into~$s_b$.
    \end{enumerate}
  \item \label{step:sum-naive-mul} Compute $\exp(-\pi y^T Y^{-1} y)$ to
    absolute precision~$N_0$.
  \item Multiply~$s_b$ by this quantity for each~$b$, and output the results.
  \end{enumerate}
\end{algorithm}

Combined with \cref{algo:compute-R,algo:ellipsoid}, this algorithm yields a
provably correct way of evaluating $\thetatilde_{0,b}(z,\tau)$ to a given
absolute precision~$N$ for each~$b\in \{0,1\}^g$ when~$(z,\tau) \in \red_g$. It
is essentially the best possible when~$\E$ contains very few points. We will
use it as a subroutine in the quasi-linear algorithm in~§\ref{sec:ql} in
combination with \cref{algo:compute-R}.\textbf{B}. In general,
\cref{algo:sum-naive} uses $2^{O(g\log g)} N^{g/2}\Mul(N)\log N$ binary
operations using fast exponentiations in step~(\ref{step:sum-naive-exp}.a),
which is suboptimal. Note that we would obtain the same asymptotic complexity
if we were computing a new exponential for each~$n\in \E$.

In the rest of this subsection, we describe a much more efficient summation
process when~$\E$ contains a large number of points. We make the following
observations, many of which are already exploited
in~\cite{engeShortAdditionSequences2018} when~$g=1$.
\begin{itemize}
\item Rather than writing
  \begin{equation}
    \label{eq:pow-q}
    \e(n^T\tau n + 2n^T z) = \prod_{1\leq j\leq k\leq g} q_{j,k}^{n_j n_k}
    \prod_{1\leq j\leq g} w_j^{n_j},
  \end{equation}
  independently for each~$n\in\E$, we should minimize the number of
  multiplications and take advantage of the fact that many factors
  in~\eqref{eq:pow-q} are shared between vectors~$n,n'\in \Z^g$ that share some
  or most of their coordinates. For instance, if $n,n'\in \Z^g$ belong to the
  same 1-dimensional layer inside~$\E$, then only their first coordinates
  differ, and we have
  \begin{equation}
    \label{eq:two-terms}
    \e(n'^T\tau n + 2n'^T z) = \e(n^T\tau n + 2n^T z) \cdot q_{1,1}^{{n'_1}^2 - {n_1}^2} \cdot
    \paren[\Big]{w_1 \prod_{j=2}^{g} q_{1,j}}^{n'_1 - n_1}.
  \end{equation}
  We will leverage~\eqref{eq:two-terms} and similar equalities to compute
  $\e(n^T\tau n + 2n^T z)$ for all~$n\in \E$ using strictly less than two
  multiplications on average.
\item We should refrain from using inversions for both speed and numerical
  stability: therefore, we should also precompute $q_{j,k}^{-1}$ and~$w_j^{-1}$
  for all $j,k$.
\item We should pay attention to the precision that we actually need when
  computing each term in~\eqref{eq:Sb}. In a multiplicative context such as
  here, the notions of \emph{relative error} and \emph{relative precision} are
  more natural than absolute precisions (and are used in most computational
  systems, including FLINT): by computing $x\in \C$ to relative precision~$N$,
  we mean computing~$x$ up to an absolute error of
  $2^{-N}\abs{x}$. By~\eqref{eq:term-modulus}, the absolute value of
  $%\abs[\big]
  {\e(n^T\tau n + 2n^Tz)}$ decreases drastically as~$n$ approaches the border
  of the ellipsoid. In the summation algorithm, most multiplications should
  therefore be performed at only a fraction of the full relative precision.
\end{itemize}

In the $g=1$ case, one should also use short addition chains as
in~\cite{engeShortAdditionSequences2018}, but this improvement becomes
asymptotically negligible when~$g\geq 2$.

We summarize the above ideas in the following algorithm, where all arithmetic
operations are performed using interval arithmetic.

\begin{algorithm}
  \label{algo:summation}
  \algoinput{$(z,\tau)\in \C^g\times \Half_g$; a precision~$N\geq 2$. We
    suppose $g\geq 2$.}  \algooutput{$\thetatilde_{0,b}(z,\tau)$ to absolute
    precision~$N$ for every $b\in \{0,1\}^g$.}

  \begin{enumerate}
  \item \label{step:summation-R} Compute~$R\in \R_{\geq 0}$ using
    \cref{algo:compute-R}.\textbf{A} with~$N+1$ instead of~$N$.
  \item \label{step:summation-eld} Compute the corresponding ellipsoid~$\E$
    using \cref{algo:ellipsoid} at absolute precision~$32$, together with the
    quantities $\abs{\E}$ and $B_j$ for $1\leq j\leq g$.
  \item \label{step:summation-precomp} \textbf{Precomputation.}
    \begin{enumerate}
    \item Let~$N_0 = N + \floor{10 \log(gR)}$.
    \item Compute $q_{j,k}$ and $q_{j,k}^{-1}$ to relative
      precision~$N_0$ for all~$1\leq j\leq k\leq g$.
    \item Compute $w_j$ and $w_j^{-1}$ to relative precision~$N_0$ for all~$1\leq j\leq g$.
    \item Compute ${q_j}^{k^2}$ to relative precision~$N_0$ for all
      $1\leq j\leq g$ and $0\leq k\leq B_j$. This uses $2 B_j$ multiplications
      if we also compute~${q_j}^{2k+1}$ and write
      \begin{displaymath}
        {q_j}^{(k+1)^2} = {q_j}^{k^2} \cdot {q_j}^{2k+1}
        \quad\text{and}\quad
        {q_j}^{2(k+1)+1} = {q_j}^{2k+1} \cdot {q_j}^2.
      \end{displaymath}
      Alternatively, we could use short addition chains as
      in~\cite{engeShortAdditionSequences2018}.
    \item Let $s_b = 0$ for each $b\in \{0,1\}^g$.
    \item Go to step~\eqref{step:summation-rec} with~$d=g$, $\eldF = \E$,
      $N=N_0$, $\Pi=1$, and $v'= -Y^{-1}y$.
    \end{enumerate}
  \item \label{step:summation-rec} \textbf{Recursion.} In this step, we
    consider an ellipsoid data structure~$\eldF$ of ambient dimension
    $2\leq d\leq g$ contained in~$\E$. Let~$(n_{d+1},\ldots,n_g)$ be the common
    coordinates of~$\eldF$, and let $M_d^-, M_d^+$ be the bounds of~$\eldF$. If
    $M_d^- < M_d^+$, then there is nothing to be done in this step. Otherwise,
    let $m_d$ be the closest integer to~$v'_d$.  Assume that we are
    given:
    \begin{itemize}[topsep=2pt, itemsep=2pt]
    \item The orthogonal projection~$v'$ of the vector~$v = -Y^{-1}y$
      (for~$\norm{\placeholder}_\tau$) on the $d$-dimensional affine space
      $(\R,\ldots,\R,n_{d+1},\ldots,n_g)$, at low precision;
    \item a current relative precision $N\in \Z_{\geq 2}$ such that
      $2^{-N} \leq 2^{-N_0} \exp(\Delta^2)$, where~$\Delta = \norm{v-v'}_\tau$
      is the distance between $v$ and this affine space;
    \item ${q_{j,k}}^{n_k}$ and~${q_{j,k}}^{-n_k}$ for each~$1\leq j\leq d$ and
      $d+1\leq k\leq g$ to relative precision~$N$; and
    \item
      $\Pi := \displaystyle \prod_{d+1\leq j, k\leq g} q_{j,k}^{n_j n_k}
      \prod_{d+1\leq j\leq g} w_j^{n_j}$ to relative precision~$N$.
    \end{itemize}
    Follow these steps to add all terms indexed by
    $n\in \eldF$ into $s_b$ for $b\in \{0,1\}^g$.
    \begin{enumerate}
    \item Compute \vspace{-4mm}
      \begin{displaymath}
        u := w_d \prod_{k=d+1}^g
        {q_{d,k}}^{n_k} \quad \text{and}\quad
        u^{-1} = w_d^{-1} \prod_{k=d+1}^g {q_{d,k}}^{-n_k} \vspace{-2mm}
      \end{displaymath}
      to relative precision~$N$ using $2(g-d)$ multiplications.
    \item Compute $\Pi u^{m_d}$ to relative precision~$N$ using a
      square-and-multiply algorithm, starting from either $u$ or~$u^{-1}$
      depending on the sign of~$m_d$. Similarly compute ${q_{j, d}}^{\pm m_d}$
      for every $1\leq j < d$.
    \item  For every $n_d$ from $m_d$ to $M_d^+$, do:
      \begin{enumerate}
      \item Let
        $N' = N - \floor[\big]{c_d^2\abs{n_d - v'_d}^2/\log(2)}$
        using low precision.
      \item If $n_d > m_d$, compute $\Pi u^{n_d}$ and ${q_{j,d}}^{\pm n_d}$ for
        all $1\leq j < d$ to relative precision~$N'$ using one multiplication,
        as follows:
        \begin{displaymath}
          \qquad\qquad  \Pi u^{n_d} = (\Pi u^{n_d-1})\cdot u \quad\text{and}\quad q_{j,d}^{\pm n_d} = {q_{j,d}}^{\pm (n_d-1)}\cdot q_{j,d}^{\pm 1}.
        \end{displaymath}
        If $n = m_d$, get these quantities from step~(b) instead.
      \item Call step~\eqref{step:summation-rec} (if $d>2$) or
        step~\eqref{step:summation-base} (if $d=2$) on the child of~$\eldF$
        corresponding to~$n_d$, with $N'$ instead of~$N$. The necessary
        quantities have been computed in~(ii), except the analogue of $\Pi$,
        which is~${q_{d,d}}^{n_d^2} \cdot \Pi u^{n_d}$, and~$v'$, which is
        computed as in step~(\ref{step:eld-rec}.b) of
        \cref{algo:ellipsoid}. The square power is precomputed in
        step~\eqref{step:summation-precomp}, so we need one multiplication to
        relative precision~$N'$.
      \end{enumerate}
    \item Similarly to step~(c), loop over $n_d$ from $m_d-1$ to $M_d^-$
      downwards to call step~\eqref{step:summation-rec} or~\eqref{step:summation-base}
      on the other children of $\eldF$. In the analogue of step (c.ii) above
      for $n_d = m_d-1$, we reuse the quantities from step (b).
    \end{enumerate}

  \item \label{step:summation-base} \textbf{Base case.} In this step, we consider
    the same setting as in step~\eqref{step:summation-rec}, but we suppose
    that~$\eldF$ has dimension $d=1$.
    \begin{enumerate}
    \item For each $M_1^-\leq n_1\leq M_1^+$, compute and store $u^{n_1}$ to
      relative precision~$N'$, where $u$ and~$N'$ are defined as in
      steps~(\ref{step:summation-rec}.a) and~(\ref{step:summation-rec}.c.i)
      respectively. This is done as follows. Up to exchanging $u$ and $u^{-1}$,
      we assume that $m_1\leq 0$.
      \begin{enumerate}
      \item Compute $u^{m_1}$ to relative precision~$N$ using a
        square-and-multiply algorithm starting from $u^{-1}$.
      \item For every $n_1$ from $m_1$ to $M_1^+$, compute $u^{n_1}$ either as
        $(u^{n_1/2})^2$ if $n_1$ is even and positive, or as $u^{n_1-1} \cdot u$ otherwise.
      \item For every $n_1$ from $m_1-1$ to $M_1^-$ downwards, compute
        $u^{n_1}$ either as $(u^{n_1/2})^2$ if $n_1$ is even and
        $n_1/2 \leq m_1$, or as $u^{n_1+1}\cdot u^{-1}$ otherwise.
      \end{enumerate}
    \item Compute
      \begin{displaymath}
        \qquad\qquad
        A_0 = \Pi\cdot \sum_{\substack{M_1^-\leq n_1\leq M_1^+\\ n_1 \text{ even}}} q_1^{n_1^2} u^{n_1}
        \quad\text{and}\quad
        A_1 = \Pi\cdot \sum_{\substack{M_1^-\leq n_1\leq M_1^+\\ n_1 \text{ odd}}} q_1^{n_1^2} u^{n_1}
      \end{displaymath}
      where each inner multiplication is performed at the relative precision
      of~$u^{n_1}$, and multiplications by~$\Pi$ are performed at relative
      precision~$N$.
    \item For each $b\in \{0,1\}^g$, add $(-1)^{\sum_{j=2}^g b_j n_j} (A_0 + (-1)^{b_1} A_1)$ into $s_b$.
    \end{enumerate}
  \item \label{step:summation-output} \textbf{Final step.} For
    each~$b\in \{0,1\}^g$, multiply $s_b$ by~$\exp(-\pi y^T Y^{-1}y)$ at
    relative precision~$N_0$, then add~$2^{-N-1}$ to the error radius of~$s_b$.
  \end{enumerate}
\end{algorithm}

\begin{thm}
  \label{thm:unif-summation}
  If $(z,\tau)\in \red_g$ is exact, then
  \cref{algo:summation} evaluates $\thetatilde_{0,b}(z,\tau)$ for all
  $b\in\{0,1\}^g$ to any given absolute precision $N\in \Z_{\geq 2}$ in time
  $2^{O(g\log g)} N^{g/2}\Mul(N)$, and performs strictly less than one
  multiplication at precision~$N$ per ellipsoid point on average.
\end{thm}

\begin{proof}
  We let the reader check that \cref{algo:summation} correctly computes each
  term of~\eqref{eq:Sb}, and that the relative precision of the output of any
  multiplication is never greater than the relative precision of its
  inputs. Moreover, for every $n\in \E$, the relative error
  on~${q_1}^{n_1^2} u^{n_1}$ that we demand in
  step~(\ref{step:summation-base}.b) is at most
  $2^{-N_0}\exp(\norm{n-v}_\tau^2)$.

  Therefore, if we ignore precision losses occurring in arithmetic operations,
  then the total absolute error on each $s_b$ is at most
  $2^{-N_0} \abs{\E} \leq 2^{-N-1}$ by \cref{prop:ellipsoid}. In fact, the
  number of guard bits we add in~$N_0$ is sufficient to cover precision losses
  occurring in the algorithm, because each term of~$s_b$ is obtained after a
  sequence of at most $2(B_1 + \cdots + B_g)\leq 4g + 4Rg^3$ multiplications,
  again by \cref{prop:ellipsoid}.

  Given the choice of~$R$ in step~\eqref{step:summation-R}, the final step
  correctly computes $\thetatilde_{0,b}(z,\tau)$ up to an error of at
  most~$2^{-N}$, so \cref{algo:summation} is correct overall.

  We now count the average number of multiplications that \cref{algo:summation}
  performs for each point of~$\E$. The asymptotically dominating step is
  step~(\ref{step:summation-base}), which is entered once for each
  1-dimensional ellipsoid data structure contained in~$\E$. In
  step~(\ref{step:summation-base}.a), item~(i) is negligible. In items~(ii)
  and~(iii), the complex squarings are cheaper than general multiplications. In
  combination with our use of smaller precisions, we spend strictly less than
  one multiplication at relative precision~$N$ per lattice point in this
  step. By \cref{prop:ellipsoid}, $\E$ contains $2^{O(g \log g)} N^{g/2}$
  points, so the algorithm indeed runs in time
  $ 2^{O(g \log g)} N^{g/2} \Mul(N)$.
\end{proof}

\begin{rem} We gather further miscellaneous comments on \cref{algo:summation}.
  \begin{itemize}
  \item When evaluating $\theta_{0,b}$ for different vectors~$z$ and a
    fixed~$\tau$, we can compute the ellipsoid~$\E$ only once and mutualize
    steps~(\ref{step:summation-precomp}.b),~(\ref{step:summation-precomp}.d) as
    well as part of step~(\ref{step:summation-rec}.c.ii).
  \item \Cref{algo:summation} can easily be adapted to compute all partial
    derivatives of~$\theta_{0,b}$ with respect to~$z$ of total order at most
    some~$B\geq 0$. First, we compute~$R$ using
    \cref{algo:compute-R-der}. After that, we only modify
    steps~(\ref{step:summation-base}.b),~(\ref{step:summation-base}.c)
    and~(\ref{step:summation-output}). In step~(\ref{step:summation-base}.b),
    we compute~$2(B+1)$ sums, multiplying the summands by $n_1^{\nu_1}$ for
    each~$0\leq \nu_1\leq B$. In step~(\ref{step:summation-base}.c), we
    multiply~$A_0$ and~$A_1$ by all the possible powers of~$n_2,\ldots,n_g$
    before adding them into the partial derivatives. In
    step~(\ref{step:summation-output}), we do not multiply the output
    by~$\exp(-\pi y^T Y^{-1}y)$.
  \item In our FLINT implementation, we take $N_0\simeq N + \log_2(\abs{\E})$
    and allow inexact input. For these reasons, the error radius of the output
    may be slightly larger than~$2^{-N}$. This is not a big issue in practice,
    because in most multiplications occurring in \cref{algo:summation}, the
    relative precision we demand on the output is strictly smaller than the
    relative precision of both inputs.
  \item In our implementation, we perform the dot products in
    step~(\ref{step:summation-base}.b) using the highly optimized FLINT
    function \texttt{acb\_dot}. This is crucial for performance, as this step
    dominates the rest of the algorithm. When $g=1$, the crucial step is
    instead~(\ref{step:summation-precomp}.d).
  \item We could also use squarings in step~(\ref{step:summation-rec}.b.ii) for a
    marginal performance gain. Another idea would be to process $n_d$
    and $-n_d$ simultaneously to avoid computing ${q_{j,d}^{\pm n_d}}$ twice.
  \item It should be possible to combine \cref{algo:ellipsoid,algo:summation}
    to obtain an almost memory-free summation algorithm, but precomputing the
    ellipsoid makes the algorithm easier to describe, implement, and test.
  \end{itemize}
\end{rem}

\section{A uniform quasi-linear algorithm}
\label{sec:ql}

\subsection{Overview}
\label{subsec:dupl}

As explained in the introduction, our approach to evaluate theta constants
$\theta_{a,0}(0,\tau)$ at a given reduced matrix $\tau\in \Half_g$ to
precision~$N\in \Z_{\geq 2}$ is to apply $h\simeq \log_2(N)$ times the
duplication formula
\begin{equation}
  \label{eq:dupl-a0}
  \theta_{a,0}(0,\tau')^2 = \sum_{a'\in \{0,1\}^g} \theta_{a',0}(0,2\tau')\theta_{a+a',0}(0,2\tau'),
\end{equation}
at~$\tau' = 2^{h-1}\tau$, etc., down to~$\tau' = \tau$, extracting the correct
square roots at each step. These square roots are determined by low-precision
approximations computed with summation algorithms. To initialize the process,
we compute theta values at~$2^h\tau$ using summation as well. As announced in
the introduction, this simple algorithm will work if the eigenvalues
of~$\im(\tau)$ are~$O(1)$ and the inequalities~\eqref{eq:big-thetaa0} hold, and
allows us to evaluate theta constants only.

In order to generalize this approach to any pair $(z,\tau)\in \red_g$ (but
still assuming that the eigenvalues of~$\im(\tau)$ are~$O(1)$ for now), we use
the following variants of the duplication formula~\eqref{eq:dupl-a0}.

\begin{prop}
  \label{prop:dupl-gen}
  For all~$\tau'\in \Half_g, x, x'\in \C^g$ and~$a,b\in\{0,1\}^g$, we have
  \begin{equation}
    \label{eq:dupl-z-zp}
    \thetatilde_{a,b}(x,\tau') \thetatilde_{a,b}(x',\tau')
    = \sum_{a'\in(\Z/2\Z)^g} (-1)^{a'^T b} \thetatilde_{a',0}(x+x',2\tau') \thetatilde_{a+a',0}(x-x',2\tau').
  \end{equation}
  In particular,
  \begin{equation}
    \label{eq:dupl-z}
    \thetatilde_{a,b}(x,\tau')^2 = \sum_{a'\in (\Z/2\Z)^g} (-1)^{a'^T b} \thetatilde_{a',0}(2x,2\tau') \theta_{a+a',0}(0,2\tau').
  \end{equation}
\end{prop}

\begin{proof}
  The multiplicative factors in the definition of~$\thetatilde_{a,b}$ cancel
  out in the above formulas, so it is enough to establish them for honest theta
  values. When $b=0$, equation~\eqref{eq:dupl-z-zp} is precisely
  \cite[Prop.~6.4 p.\,221]{mumfordTataLecturesTheta1983} in the case
  $n_1=n_2=1$. The formula for a general~$b$ can be deduced from the case $b=0$
  by writing
  \begin{displaymath}
    \theta_{a,b}(x,\tau') = \theta_{a,0}(x + \tfrac b2,\tau') \quad\text{and}\quad
    \theta_{a,0}(x + b,\tau') = (-1)^{a^T b} \theta_{a,0}(x,\tau'). \qedhere
  \end{displaymath}
\end{proof}

Then, to evaluate~$\theta_{a,b}(z,\tau)$, we introduce an auxiliary
vector~$t\in [0,1]^g$, and assume that the analogue of~\eqref{eq:big-thetaa0}
holds at the points~$2^j t$, $2^{j+1}t$, $2^j(z + t)$ and~$2^j(z + 2t)$ instead
of~$0$. More precisely, using the notation~$v$ from~§\ref{subsec:notation}, we
suppose that the following inequalities hold for all $1\leq j\leq h-1$ and all
$a,b\in \{0,1\}^g$:
\begin{equation}
  \label{eq:big-thetaa0-gen}
  \begin{aligned}
    \abs[\big]{\theta_{a,0}(2^j t,2^j{\tau})} &\geq \eta_{g,h} \exp\paren[\big]{-2^j \Dist_{\tau}(0, \Z^g + \tfrac a2)^2},\\
    \abs[\big]{\theta_{a,0}(2^{j+1}t ,2^j{\tau})} &\geq \eta_{g,h} \exp\paren[\big]{-2^j \Dist_{\tau}(0, \Z^g + \tfrac a2)^2},\\
    \abs[\big]{\thetatilde_{a,0}(2^j(z + t),2^j{\tau})} &\geq \eta_{g,h} \exp\paren[\big]{-2^j \Dist_{\tau}(v, \Z^g + \tfrac a2)^2},\\
    \abs[\big]{\thetatilde_{a,0}(2^j(z + 2t),2^j{\tau})} &\geq \eta_{g,h} \exp\paren[\big]{-2^j \Dist_{\tau}(v, \Z^g + \tfrac a2)^2},\\
    \abs[\big]{\thetatilde_{a,b}(z + 2t, {\tau})} &\geq \eta_{g,h} \exp\paren[\big]{- \Dist_{\tau}(v, \Z^g + \tfrac a2)^2}\\
  \end{aligned}
\end{equation}
for a reasonably large value of~$\eta_{g,h}>0$. The distances on the right hand
side are independent of~$t$ because~$t$ has real entries.

Instead of just manipulating theta values at $(0,2^j\tau)$, we manipulate
$\thetatilde_{a,0}(2^j x, 2^j\tau)$ for $1\leq j\leq h$ and
$x\in \{0, t, 2t, z, z+t, z+2t\}$ and replace~\eqref{eq:dupl-a0}
at each step by the following formulas derived from~\cref{prop:dupl-gen}:
\begin{equation}
  \label{eq:dupl-gen}
  \begin{aligned}
    \theta_{a,0}(t', \tau')^2 &= \sum_{a'\in \{0,1\}^g} \theta_{a',0}(2t',2\tau')\theta_{a+a',0}(0,2\tau'),\\
    \theta_{a,0}(2t', \tau')^2 &= \sum_{a'\in \{0,1\}^g} \theta_{a',0}(4t',2\tau')\theta_{a+a',0}(0,2\tau'),\\
    \thetatilde_{a,0}(z'+t', \tau')^2 &= \sum_{a'\in \{0,1\}^g} \thetatilde_{a',0}(2z'+2t',2\tau')\theta_{a+a',0}(0,2\tau'),\\
    \thetatilde_{a,0}(z'+2t', \tau')^2 &= \sum_{a'\in \{0,1\}^g} \thetatilde_{a',0}(2(z'+2t'),2\tau')\theta_{a+a',0}(0,2\tau'),\\
    \theta_{a,0}(0,\tau') \theta_{a,0}(2t', \tau') &= \sum_{a'\in \{0,1\}^g} \theta_{a',0}(2t',2\tau')\theta_{a+a',0}(2t',2\tau'),\\
    \thetatilde_{a,0}(z',\tau') \thetatilde_{a,0}(z'+2t', \tau') &= \sum_{a'\in \{0,1\}^g} \thetatilde_{a',0}(2z'+2t',2\tau')\theta_{a+a',0}(2t',2\tau'),\\
  \end{aligned}
\end{equation}
where~$t' = 2^jt$, $z' = 2^j z$, and~$\tau' = 2^j\tau$, for~$j = h-1$ down
to~$0$. Crucially, we obtain the values at~$0$ and~$z'$ with the last two
formulas using divisions, so we do not need any low-precision approximations
for those values anymore. At the last step~$j=0$ only, we introduce nonzero
values of~$b$. As we will show, this yields a quasi-linear algorithm that works
well on all inputs~$(z,\tau)$ such that the eigenvalues of~$\im(\tau)$
are~$O(1)$. Besides simplifying the analysis of precision losses, the
assumption that the eigenvalues of~$\im(\tau)$ are small allows us to use
Hadamard transformations to implement the duplication
formulas~\eqref{eq:dupl-gen} more efficiently.

When~$\im(\tau)$ has larger eigenvalues, doing~$h\simeq \log_2(N)$ duplication
steps as above is too much: for an extreme example, if the eigenvalues of~$\im(\tau)$
are already in~$\Omega(N)$, then we should be applying the summation algorithm
directly without any duplication steps. In general, assume that one of the
diagonal coefficients~$c_1,\ldots,c_g$ of~$C$, as defined
in~§\ref{subsec:notation}, is at least 10, say. Let $0\leq d\leq g-1$ be the
minimal value such that~$c_j \geq c_g/10$ for all $d+1\leq j\leq
g$. Then~$c_{d+1},\ldots,c_g$ will be quite large as well because~$\tau$ is
assumed to be reduced. In that case, we choose~$h\in \Z_{\geq 0}$ such that
$h = \log_2(N/c_{d+1}^2) + O(1)$, which is potentially much smaller
than~$\log_2(N)$.

Choosing an auxiliary vector~$t$ and using~$h$ duplication steps as above, it
is enough to evaluate theta functions at $(2^{h}x, 2^{h}\tau)$
where~$x\in \{0,t,2t,z+t,z+2t\}$ (or $x=z$, in the case $h=0$). When doing so,
we only manipulate matrices $\tau'\in \Half_g$ such that the eigenvalues
of~$\im(\tau')$ are in~$O(N)$. Now some, but maybe not all, of the eigenvalues
of~$\im(2^{h}\tau)$ are in~$\Omega_g(N)$; we will make the dependencies on~$g$
explicit as part of the complexity analysis. More precisely, if we write
\begin{equation}
  \label{eq:split-matrix}
  \tau = \begin{pmatrix} \tau_0 & \sigma \\ \sigma^T & \tau_1 \end{pmatrix}
  \quad \text{and}\quad
  C =
  \begin{pmatrix}
    C_0 & \Xi \\ 0 & C_1
  \end{pmatrix}
\end{equation}
where~$\tau_0$ and~$C_0$ are $d\times d$ matrices, then the $g-d$ diagonal
coefficients of~$2^{h/2}C_1$ are all in~$\Omega_g(\sqrt{N})$.

We use this observation to express theta values
$\theta_{a,b}(2^{h}x,2^{h}\tau)$ for $x=z$, say, as short sums of
$d$-dimensional theta values for the matrix~$\tau_0$, as follows. If we were to
use the summation algorithm to evaluate~$\theta_{a,b}(2^{h}x, 2^{h}\tau)$, we
would pick a suitable ellipsoid radius~$R = O_g(\sqrt{N})$ and write
\begin{equation}
  \label{eq:approx-theta-split}
  \theta_{a,b}(2^{h}x, 2^{h}\tau) \simeq \sum_{n\in \paren[\normalsize]{\Z^g + \tfrac a 2} \cap \eld{v}{R}{2^h\tau}}
  \e \paren[\big]{2^h n^T \tau n + 2^{h+1} n^T (x + \tfrac b2)}.
\end{equation}
For each point $n = (n_1,\ldots,n_g)$ appearing in this sum, we define
$n' = (n_1,\ldots,n_d)$ and~$n'' = (n_{d+1},\ldots, n_g)$. We similarly define
the vectors $a',a'',b',b'',x',x'',v',v''$. Then, the inequality
\begin{equation}
  \label{eq:split-norm-ineq}
  \norm{2^{h/2}C_1(n'' - v'')}_2^2 \leq \norm{n - v}_{2^h\tau}^2
\end{equation}
shows that there are only~$O_g(1)$ possibilities for~$n''$, and we can
rewrite~\eqref{eq:approx-theta-split} as
\begin{equation}
  \label{eq:split-theta}
  \theta_{a,b}(2^hx, 2^h\tau) \simeq \sum_{n''}
  \e\paren[\big]{2^h n''^T \tau_1 n'' + 2^{h+1} n''^T(x'' + \tfrac{b''}{2})}
  \theta_{a',b'} \paren[\big]{2^h(x' + \sigma n''), 2^h \tau_0}
\end{equation}
with the convention that $0$-dimensional theta values are equal to 1.
Rescaling~\eqref{eq:split-theta} by the appropriate exponential factors, we can
express the quantities $\thetatilde_{a,b}(2^hx, 2^h\tau)$ as short sums of
normalized theta values
$\thetatilde_{a',b'}\paren[\big]{2^h(x' + \sigma n''), 2^h\tau_0}$, for certain
values of~$n''$. Since $\tau_0$ has size $d\times d$ and $d<g$, we obtain the
following recursive algorithm.

\begin{algorithm}
  \label{algo:ql}
  \algoinput{An exact point $(z,\tau)\in \red_g$;
    a precision $N\in \Z_{\geq 2}$.}  \algooutput{The theta values
    $\thetatilde_{a,b}(z,\tau)$ up to an absolute error of $2^{-N}$, for all
    characteristics $a,b\in\{0,1\}^g$.}

  \begin{enumerate}
  \item \label{step:ql-d} If the diagonal Cholesky
    coefficients~$c_1,\ldots,c_g$ as defined in~§\ref{subsec:notation} are
    bounded above by 10, let~$d=0$. Otherwise, let $0\leq d\leq g-1$ be minimal
    value such that~$c_j \geq c_g/10$ for all $d+1\leq j\leq g$.
  \item \label{step:ql-h} Choose~$h\in \Z_{\geq 0}$ such that
    $h = \log_2(N/c_{d+1}^2) + O(1)$. If~$h=0$, go to
    step~\eqref{step:ql-nothing}, otherwise continue.
  \item \label{step:ql-dist} Compute~$\Dist_\tau(0, \Z^g + \tfrac a2)^2$
    and~$\Dist_\tau(v, \Z^g + \tfrac a2)^2$ for all~$a\in \{0,1\}^g$.
  \item \label{step:ql-t} Choose an auxiliary real vector $t\in [0,1]^g$ so
    that the inequalities~\eqref{eq:big-thetaa0-gen} hold for a reasonable
    value of~$\eta_{g,h}>0$. As a byproduct of this computation, for
    each~$0\leq j\leq h-1$, collect low-precision approximations of the
    required theta values at all points~$(2^jx, 2^j\tau)$
    for~$x\in \{t, 2t, z+t, z+2t\}$ that do not contain~$0$ as complex balls.
  \item \label{step:ql-init} Evaluate $\thetatilde_{a,0}(2^hx, 2^h\tau)$ for
    each $x\in \{0,t,2t,z+t,z+2t\}$ and all $a\in \{0,1\}^g$ to high precision,
    as follows.
    \begin{enumerate}
    \item If~$d=0$, use the summation algorithm~\ref{algo:summation}, otherwise
      continue.
    \item Let~$\tau_0, \sigma, \tau_1$ be as in~\eqref{eq:split-matrix}, and
      compute the set of vectors~$n''$ to consider in
      equation~\eqref{eq:split-theta} for~$x\in \{0,t,2t\}$. Do the same for
      $x\in\{z+t,z+2t\}$.
    \item For each~$n''$, compute
      $\thetatilde_{a',b'}\paren[\big]{2^h(x'+\sigma n''), 2^h\tau_0}$ to high
      precision, using \cref{algo:ql} recursively in dimension~$d$ after having
      reduced the first argument as in~§\ref{subsec:z-reduction}.
    \item Use equation~\eqref{eq:split-theta}, rescaling theta values by the
      approximate exponential factors, to get
      $\thetatilde_{a,0}(2^hx, 2^h\tau)$.
    \end{enumerate}
  \item \label{step:ql-rec} For $j = h-1$ down to $j = 1$, starting from the
    theta values $\thetatilde_{a,0}(2^{j+1}x, 2^{j+1}\tau)$ for
    $x\in \{0,t,2t,z+t,z+2t\}$ and $a\in \{0,1\}^g$, do:
    \begin{enumerate}
    \item Compute $\thetatilde_{a,0}(2^jx, 2^{j}\tau)^2$ for
      $x\in \{t, 2t, z+t, z+2t\}$ and $a\in \{0,1\}^g$
      using equation~\eqref{eq:dupl-z} with $b=0$.
    \item Compute $\thetatilde_{a,0}(2^{j}x, 2^{j}\tau)$ for
      $x\in \{t, 2t, z+t, z+2t\}$ and $a\in \{0,1\}^g$ by square
      root extractions, using the low-precision approximations collected in
      step~\eqref{step:ql-t} to determine the correct signs.
    \item Compute the products
      $\thetatilde_{a,0}(0,2^{j}\tau) \thetatilde_{a,0}(2^{j+1}t, 2^{j}\tau)$ for each
      $a\in \{0,1\}^g$ using equation~\eqref{eq:dupl-z-zp} with $b=0$.
    \item Compute $\thetatilde_{a,0}(0,2^{j}\tau)$ for $a\in \{0,1\}^g$ by
      dividing the results of steps~(\ref{step:ql-rec}.c)
      and~(\ref{step:ql-rec}.b). We have gathered all the necessary data for
      step~(\ref{step:ql-rec}) on~$j-1$, or step~(\ref{step:ql-base}) when~$j=1$.
    \end{enumerate}
  \item \label{step:ql-base} Given $\thetatilde_{a,0}(2x,2\tau)$ for all
    $x\in \{0,t,2t,z+t,z+2t\}$ and $a\in \{0,1\}^g$, do:
    \begin{enumerate}
    \item Compute $\thetatilde_{a,b}(z+2t,\tau)$ for all $a,b\in \{0,1\}^g$
      using equation~\eqref{eq:dupl-z} and square root extractions as in
      step~(\ref{step:ql-rec}.b).
    \item Compute
      $\thetatilde_{a,b}(z,\tau)\thetatilde_{a,b}(z+2t,\tau)$ for all
      $a,b\in \{0,1\}^g$ using equation~\eqref{eq:dupl-z-zp}.
    \item Compute $\thetatilde_{a,b}(z,\tau)$ for $a,b\in\{0,1\}^g$ by dividing
      the results of steps~(\ref{step:ql-base}.b) and~(\ref{step:ql-base}.a),
      output these values and stop.
    \end{enumerate}
  \item \label{step:ql-nothing} Here~$h=0$. Run the equivalent of
    step~\eqref{step:ql-init}, but at~$x=z$ only and for all
    characteristics~$(a,b)$, and output the results.
  \end{enumerate}
\end{algorithm}

We complete the description of \cref{algo:ql} and the proof of
\cref{thm:main-intro} in the rest of this section, which is structured as
follows.  In~§\ref{subsec:distances}, we explain how distances are computed in
step~\eqref{step:ql-dist}. In~§\ref{subsec:precision}, we explain how
precisions are managed in step~\eqref{step:ql-rec} and~\eqref{step:ql-base}
assuming that the inequalities~\eqref{eq:big-thetaa0-gen} are satisfied, and
describe how we use Hadamard transforms. In~§\ref{subsec:lowerdim}, we explain
how precisions are managed when lowering the dimension in
step~\eqref{step:ql-init}. In~§\ref{subsec:t-exists}, we show that suitable
auxiliary vectors~$t$ exist. We explain how~$t$ can be chosen algorithmically,
either at random or deterministically, in~§\ref{subsec:choose-t}. We complete
the proof of \Cref{thm:main-intro} in §\ref{subsec:complexity} based on the
deterministic algorithm to pick~$t$. Finally, in~§\ref{subsec:deriv}, we
explain how to also compute derivatives of theta functions on~$\red_g$ in uniform
quasi-linear time using finite differences.

\begin{rem}
  \label{rem:ql-improvements}
  In our implementation, we further improve on \cref{algo:ql} in the
  following ways.
  \begin{enumerate}
  \item \label{item:several-z} Considering several vectors~$z$ at once
    (including $z=0$) in the input and output of \cref{algo:ql} is
    advantageous, because the auxiliary vector~$t$ and the theta values at
    $(0,2^j\tau), (2^j t,2^j\tau)$ and $(2^{j+1}t,2^j\tau)$ can be
    mutualized. This occurs in particular when some eigenvalues of~$\im(\tau)$
    are large and we fall back to the evaluation algorithm in lower dimensions
    in step~(\ref{step:ql-init}.c), and also when computing derivatives from
    finite differences as in~§\ref{subsec:deriv}.
  \item \label{item:easy-steps} It might be the case that the values
    $\thetatilde_{a,0}(2^j z,2^j\tau)$ or $\thetatilde_{a,0}(0,2^j\tau)$ are
    large enough for the first few values of~$j=0,1,2,\ldots$ In such a case,
    switching to the simpler duplication formula~\eqref{eq:dupl-a0} at these
    steps is possible and cheaper. Then, less computations need to be performed
    at the last step involving~$t$, as in step~\eqref{step:ql-base} above. In
    the special case $j=0$ and~$z=0$, we know that~$\theta_{a,b}(0,\tau)=0$
    whenever $a^Tb$ is odd, so we can omit those theta values when
    checking~\eqref{eq:big-thetaa0} (which obviously does not hold in that
    case). See also \cref{rem:zero-odd}.
  \item \label{item:no-th0} Additional work can be saved when~$j=1$ in
    step~\eqref{step:ql-rec}, as only~$\theta_{a,0}(2t,2\tau)$
    and~$\theta_{a,0}(0,2\tau)$ are needed in the last step. These savings can
    be transposed to previous steps if item~\eqref{item:easy-steps} applies.
  \item \label{item:only-00} Two duplication steps be saved if one wishes to
    only evaluate $\theta_{0,0}(z,\tau)$, using the relation \cite[(3.10)
    p.\,38]{cossetApplicationsFonctionsTheta2011}:
    \begin{displaymath}
      \theta_{0,0}(z,\tau) = \sum_{a\in \{0,1\}^g} \theta_{a,0}(2z, 4\tau).
    \end{displaymath}
  \item \label{item:squares} Evaluating squares of theta functions also
    essentially saves one duplication step, as the duplication
    formula~\eqref{eq:dupl-z} shows that we only have to evaluate
    $\thetatilde_{a,0}(2z,2\tau)$ and~$\thetatilde_{a,0}(0,2\tau)$ for
    all $a\in \{0,1\}^g$.
  \item \label{item:better-d} Finally, the definition of~$d$ in
    step~\eqref{step:ql-d} is a bit arbitrary. In our implementation, we
    attempt to make the best possible choices of~$d$ and~$h$ depending on the
    input~$\tau$ to minimize the overall running time based on experimental
    measurements.
  \end{enumerate}

  \Cref{fig:ql} shows a possible workflow for \cref{algo:ql} based on the above
  improvements~\eqref{item:several-z}--\eqref{item:no-th0} on three input
  vectors~$0,z_1,z_2$ with~$d = 0$ and~$h=2$. The dashed
  arrows may be omitted if one only wishes to output theta values at~$z_1$
  and~$z_2$.

\begin{figure}[ht]
  \centering\scriptsize
  \begin{tikzpicture}
    \tikzstyle{mynode}=[minimum width=1.6cm,minimum height=0.7cm,
    rectangle,draw]
    \node[mynode] (0tau) at (-4,0) {$\theta_{a,b}(0,\tau)$};
    \node[mynode] (2ttau) at (0,0) {$\theta_{a,b}(2t,\tau)$};
    \node[mynode] (z1tau) at (4.5,0) {$\thetatilde_{a,b}(z_1,\tau)$};
    \node[mynode] (z2tau) at (8.5,0) {$\thetatilde_{a,b}(z_2,\tau)$};

    \node[mynode] (02ttau) at (-4,2) {$\theta_{a,b}(0,\tau)\theta_{a,b}(2t,\tau)$};
    \node[mynode] (2ttausqr) at (0,2) {$\theta_{a,b}^2(2t,\tau)$};
    \node[mynode] (z1tausqr) at (4.5,2) {$\thetatilde_{a,b}^2(z_1,\tau)$};
    \node[mynode] (z2tausqr) at (8.5,2) {$\thetatilde_{a,b}^2(z_2,\tau)$};

    \node[mynode] (02tau) at (-4,4) {$\theta_{a,0}(0,2\tau)$};
    \node[mynode] (2t2tau) at (-2, 4) {$\theta_{a,0}(2t, 2\tau)$};
    \node[mynode] (4t2tau) at (0, 4) {$\theta_{a,0}(4t, 2\tau)$};
    \node[mynode] (2z12tau) at (3,4) {$\thetatilde_{a,0}(2z_1, 2\tau)$};
    \node[mynode] (2z14t2tau) at (5.8,4) {$\thetatilde_{a,0}(2z_1 + 4t, 2\tau)$};
    \node[mynode] (2z22tau) at (8.5,4) {$\thetatilde_{a,0}(2z_2,2\tau)$};

    \node[mynode] (04t2tau) at (-4,6) {$\theta_{a,0}(0,2\tau)\theta_{a,0}(4t,2\tau)$};
    \node[mynode] (2t2tausqr) at (-2,7) {$\theta_{a,0}^2(2t, 2\tau)$};
    \node[mynode] (4t2tausqr) at (0,7) {$\theta_{a,0}^2(4t, 2\tau)$};
    \node[mynode] (2z12z14t2tau) at (3,6) {$\thetatilde_{a,0}(2z_1,2\tau) \thetatilde_{a,0}(2z_1+4t, 2\tau)$};
    \node[mynode] (2z14t2tausqr) at (5.8,7) {$\thetatilde_{a,0}^2(2z_1+4t,2\tau)$};
    \node[mynode] (2z22tausqr) at (8.5,7) {$\thetatilde_{a,0}^2(2z_2,2\tau)$};

    \node[mynode] (04tau) at (-4,9) {$\theta_{a,0}(0,4\tau)$};
    \node[mynode] (4t4tau) at (-2,9) {$\theta_{a,0}(4t, 4\tau)$};
    \node[mynode] (8t4tau) at (0,9) {$\theta_{a,0}(8t, 4\tau)$};
    \node[mynode] (4z14t4tau) at (3,9) {$\thetatilde_{a,0}(4z_1 + 4t, 4\tau)$};
    \node[mynode] (4z18t4tau) at (5.8,9) {$\thetatilde_{a,0}(4z_1 + 8t, 4\tau)$};
    \node[mynode] (4z24tau) at (8.5,9) {$\thetatilde_{a,0}(4z_2,4\tau)$};

    \node (sum) at (2.5,10) {};
    \node[above] at (sum) {Summation};
    \node at (sum) {$\bullet$};

    \draw[->, >=latex] (sum) to (04tau.north);
    \draw[->, >=latex] (sum) to (4t4tau.north);
    \draw[->, >=latex] (sum) to (8t4tau.north);
    \draw[->, >=latex] (sum) to (4z14t4tau.north);
    \draw[->, >=latex] (sum) to (4z18t4tau.north);
    \draw[->, >=latex] (sum) to (4z24tau.north);

    \node (dupl21) at (-2, 7.8) {};
    \node (dupl22) at (0, 7.7) {};
    \node (dupl23) at (5.8, 7.8) {};
    \node (dupl24) at (8.5, 8) {};
    \node at (dupl21) {$\bullet$};
    \node at (dupl22) {$\bullet$};
    \node at (dupl23) {$\bullet$};
    \node at (dupl24) {$\bullet$};
    \draw[->, >=latex, dashed] (04tau.south) to (dupl21);
    \draw[->, >=latex](04tau.south) to (dupl22);
    \draw[->, >=latex](04tau.south) to (dupl23);
    \draw[->, >=latex](04tau.south) to (dupl24);
    \draw[->, >=latex, dashed](4t4tau.south) to (dupl21);
    \draw[->, >=latex](8t4tau.south) to (dupl22);
    \draw[->, >=latex](4z18t4tau.south) to (dupl23);
    \draw[->, >=latex](4z24tau.south) to (dupl24);
    \node[right] at (dupl21) {\eqref{eq:dupl-z}};
    \node[right] at (dupl22) {\eqref{eq:dupl-z}};
    \node[right] at (dupl23) {\eqref{eq:dupl-z}};
    \node[right] at (dupl24) {\eqref{eq:dupl-z}};
    \draw[->, >=latex, dashed] (dupl21) to (2t2tausqr.north);
    \draw[->, >=latex] (dupl22) to (4t2tausqr.north);
    \draw[->, >=latex] (dupl23) to (2z14t2tausqr.north);
    \draw[->, >=latex] (dupl24) to (2z22tausqr.north);

    \node (dupl25) at (-4,6.8) {};
    \node (dupl26) at (3,7.4) {};
    \node at (dupl25) {$\bullet$};
    \node at (dupl26) {$\bullet$};
    \draw[->, >=latex](04tau.south) to (dupl25);
    \draw[->, >=latex](4t4tau.south) to (dupl25);
    \draw[->, >=latex](4t4tau.south) to (dupl26);
    \draw[->, >=latex](4z14t4tau.south) to (dupl26);
    \node[left] at (dupl25) {\eqref{eq:dupl-z-zp}};
    \node[right] at (dupl26) {\eqref{eq:dupl-z-zp}};
    \draw[->, >=latex] (dupl25) to (04t2tau.north);
    \draw[->, >=latex] (dupl26) to (2z12z14t2tau.north);

    \draw[->, >=latex, dashed] (2t2tausqr.south) to node[right]{sqrt} (2t2tau.north);
    \draw[->, >=latex] (4t2tausqr.south) to node[right]{sqrt} (4t2tau.north);
    \draw[->, >=latex] (2z14t2tausqr.south) to node[right]{sqrt} (2z14t2tau.north);
    \draw[->, >=latex] (2z22tausqr.south) to node[right]{sqrt} (2z22tau.north);

    \node (div21) at (-4, 4.8) {};
    \node (div22) at (3, 4.8) {};
    \node at (div21) {$\bullet$};
    \node at (div22) {$\bullet$};
    \node[left] at (div21) {divide};
    \node[left] at (div22) {divide};
    \draw[->, >=latex](04t2tau.south) to (div21);
    \draw[->, >=latex](4t2tau.north) to[bend right=20] (div21);
    \draw[->, >=latex](2z12z14t2tau.south) to (div22);
    \draw[->, >=latex](2z14t2tau.north) to[bend right=20] (div22);
    \draw[->, >=latex] (div21) to (02tau.north);
    \draw[->, >=latex] (div22) to (2z12tau.north);

    \node (dupl11) at (0,2.8) {};
    \node (dupl12) at (4.5,2.8) {};
    \node (dupl13) at (8.5,3) {};
    \node at (dupl11) {$\bullet$};
    \node at (dupl12) {$\bullet$};
    \node at (dupl13) {$\bullet$};
    \draw[->, >=latex, dashed] (02tau.south) to (dupl11);
    \draw[->, >=latex] (02tau.south) to (dupl12);
    \draw[->, >=latex] (02tau.south) to (dupl13);
    \draw[->, >=latex, dashed] (4t2tau.south) to (dupl11);
    \draw[->, >=latex] (2z12tau.south) to (dupl12);
    \draw[->, >=latex] (2z22tau.south) to (dupl13);
    \node[right] at (dupl11) {\eqref{eq:dupl-z}};
    \node[right] at (dupl12) {\eqref{eq:dupl-z}};
    \node[right] at (dupl13) {\eqref{eq:dupl-z}};
    \draw[->, >=latex, dashed] (dupl11) to (2ttausqr.north);
    \draw[->, >=latex] (dupl12) to (z1tausqr.north);
    \draw[->, >=latex] (dupl13) to (z2tausqr.north);
    \node (dupl14) at (-4,2.8) {};
    \node at (dupl14) {$\bullet$};
    \draw[->, >=latex, dashed](02tau.south) to (dupl14);
    \draw[->, >=latex, dashed](2t2tau.south) to (dupl14);
    \node[left] at (dupl14) {\eqref{eq:dupl-z-zp}};
    \draw[->, >=latex, dashed] (dupl14) to (02ttau.north);

    \draw[->, >=latex] (z1tausqr.south) to node[right]{sqrt} (z1tau.north);
    \draw[->, >=latex] (z2tausqr.south) to node[right]{sqrt} (z2tau.north);
    \draw[->, >=latex, dashed] (2ttausqr.south) to node[right]{sqrt} (2ttau.north);

    \node (div1) at (-4, 0.8) {};
    \node at (div1) {$\bullet$};
    \node[left] at (div1) {divide};
    \draw[->, >=latex, dashed] (02ttau.south) to (div1);
    \draw[->, >=latex, dashed] (2ttau.north) to[bend right=20] (div1);
    \draw[->, >=latex, dashed] (div1) to (0tau.north);

  \end{tikzpicture}
  \caption{Possible workflow for three input vectors~$0,z_1,z_2$ with $h=2$.}
  \label{fig:ql}
\end{figure}

Note that when using the
improvements~\eqref{item:several-z}--\eqref{item:squares} above, the square
roots we have to extract during the second phase differ, and hence the list of
theta values that should be far from zero as in~\eqref{eq:big-thetaa0-gen}
differs accordingly. Compared to our complexity analysis of \cref{algo:ql} in
the rest of this section, implementing any of the
improvements~\eqref{item:easy-steps}--\eqref{item:squares} essentially amounts
to doing less computations. Item~\eqref{item:several-z} is an exception in this
regard as it might complicate the choice of~$t$ and~$\eta_{g,h}$, but this is
not an issue in practice; see also \cref{rem:several-z} below.
\end{rem}

\subsection{Computing distances}
\label{subsec:distances}

Computing distances in step~\eqref{step:ql-dist} of \cref{algo:ql} is crucial
in order to identify a suitable auxiliary vector~$t$ in step~\eqref{step:ql-t}
and to explain how precisions are managed in duplication steps. We achieve this
using the following recursive algorithm, where we keep the same notation as
\cref{algo:ellipsoid}.

\begin{algorithm}
  \label{algo:dist}
  \algoinput{An exact point~$(z,\tau)\in \red_g$; a
    precision~$N\in \Z_{\geq 2}$.}  \algooutput{$\Dist_\tau(v, \Z^g)^2$ to
    absolute precision~$N$, where~$v$ is defined as in~§\ref{subsec:notation}.}
  \begin{enumerate}
  \item Compute~$v$.
  \item \label{step:dist-start} If~$g=1$, output~$c_g^2\abs{v - n}^2$
    where~$n\in \Z$ is the closest integer to~$v$. Otherwise,
    let~$u = \frac14 g^{2+\log g}c_g^2$ and continue.
  \item \label{step:dist-enum} For each integer~$n_g\in \Z$ such that
    $\abs{n_g - v_g}^2 \leq \frac14 g^{2+ \log g}$, compute the orthogonal
    projection~$w = v' - C'^{-1}\Xi(v_g - n_g)\in \R^{g-1}$. Recursively
    compute $\Dist_{\tau'}(w, \Z^g)^2$, and let
    \begin{displaymath}
      u = \min\{u, \Dist_{\tau'}(w,\Z^g)^2 + c_g^2 \abs{v_g - n_g}^2\}.
    \end{displaymath}
  \item Output~$u$.
  \end{enumerate}
\end{algorithm}

In order to compute~$\Dist_\tau(v, \Z^g + \tfrac{a}{2})^2$ for all~$a$, we simply
repeat \cref{algo:dist} on shifted input.

\begin{prop}
  \label{prop:distances} \Cref{algo:dist} is correct and uses
  $2^{O(g\log^2 g)} \Mul(N + \log\abs{\im(\tau)})$ binary operations.
\end{prop}

\begin{proof}
  By \cref{cor:hkz-dist}, if~$n\in \Z^g$
  satisfies~$\Dist_\tau(v,\Z^g) = \norm{n-v}_\tau$, then its $g$-th coordinate
  $n_g$ will be listed in step~\eqref{step:dist-enum} above. Therefore
  \cref{algo:dist} is correct. By \cref{prop:eld-width-2}, all the complex
  numbers~$x$ we manipulate satisfy
  \begin{displaymath}
    \log\max\{1,\abs{x}\} = O(g + \log\abs{\im(\tau)}).
  \end{displaymath}
  The complexity bound directly follows by induction considering the range of
  values of~$n_g$ in step~\eqref{step:dist-enum}.
\end{proof}

In our implementation, we improve on \cref{algo:dist} in the following
ways. First, instead of using~$u = \frac14 g^{2+\log g}c_g^2$ in
step~\eqref{step:dist-start}, we compute a vector~$n\in \Z^g$ close to~$v$ by
rounding the entries of $C^{-1}v$ to the nearest integers, and
let~$u = \norm{n-v}_\tau^2$. Second, in step~\eqref{step:dist-enum}, we only
have to loop over integers~$n_g$ such that $c_g\abs{n_g - v_g}^2 \leq u$. We
start by considering the integers~$n_g$ that are closest to~$v_g$ since the
value of~$u$ may decrease as we proceed through
step~\eqref{step:dist-enum}.

\subsection{Precision management in duplication steps}
\label{subsec:precision}

In order to analyze precision losses in steps~\eqref{step:ql-rec}
and~\eqref{step:ql-base} of \cref{algo:ql}, we make the following key
definition.

\begin{defn}
  \label{def:shifted}
  Let~$\tau\in \Half_g$, let~$z,z'\in \C^g$, let~$a,b\in \{0,1\}^g$, and let
  $N\in \Z_{\geq 2}$. We recall the notation~$v$ from~§\ref{subsec:notation},
  and define~$v'$ analogously for the vector~$z'$.

  By \emph{computing $\thetatilde_{a,b}(z,\tau)$ to shifted absolute
    precision~$N$}, we mean computing this value up to an absolute error of
  $2^{-N} \exp\paren[\big]{-\Dist_\tau(v, \Z^g + \tfrac a2)^2}$. Similarly, by
  \emph{computing $\thetatilde_{a,b}(z,\tau)\thetatilde_{a,b}(z',\tau)$ to
    shifted absolute precision~$N$} we mean computing this product up to an
  absolute error of
  $2^{-N}\exp\paren[\big]{-\Dist_\tau(v, \Z^g + \tfrac a2)^2 - \Dist_\tau(v',
    \Z^g + \tfrac{a'}{2})^2}$.
\end{defn}

In a sense, this notion of shifted absolute precision is a compromise between
absolute and relative precision. If $\thetatilde_{a,b}(z,\tau)$ is of the
expected order of magnitude as in~\eqref{eq:big-thetaa0-gen}, then~$N$ bits of
shifted absolute precision correspond to at least $N-\log(\eta_{g,h})$ bits of
relative precision. Conversely,~$N$ bits of relative precision always
correspond to at least~$N - 10(1+\log^2 g)$ bits of shifted absolute precision
by \cref{thm:shifted-ubound}.

The main advantage of shifted absolute precisions is that they are compatible
with the duplication formula~\eqref{eq:dupl-z-zp}, as the following inequality
shows.

\begin{prop}
  \label{prop:dist-inequality}
  Let~$\tau\in \Half_g$, let~$x,x'\in \C^g$, and let~$Y = \im(\tau)$. Further
  write $v = -Y^{-1}\im(x)$ and~$v' = -Y^{-1}\im(x')$. Then for all
  $a,a'\in \{0,1\}^g$, we have
  \begin{displaymath}
    \Dist_\tau(v, \Z^g + \tfrac{a}{2})^2 + \Dist_\tau(v', \Z^g + \tfrac{a}{2})^2 \leq \Dist_{2\tau}(\tfrac{v+v'}{2},\Z^g + \tfrac{a'}2)
    + \Dist_{2\tau}(\tfrac{v-v'}{2}, \Z^g + \tfrac{a+a'}{2})^2.
  \end{displaymath}
\end{prop}

\begin{proof}
  Let~$w_{a'}\in \Z^g + \tfrac{a'}{2}$ and~$w_{a+a'}\in \Z^g + \tfrac{a+a'}{2}$
  be the closest lattice points to $\tfrac{v+v'}{2}$ and~$\tfrac{v-v'}{2}$
  respectively. By the parallelogram identity, we have
  \begin{displaymath}
    \norm{v - (w_{a'} + w_{a+a'})}_\tau^2
    + \norm{v' - (w_{a'} - w_{a+a'})}_\tau^2 =
    2\norm{\tfrac{v+v'}{2} - w_{a'}}_\tau^2 + 2\norm{\tfrac{v-v'}{2} - w_{a+a'}}_\tau^2.
  \end{displaymath}
  Since $w_{a'} + w_{a+a'}$ and $w_{a'} - w_{a+a'}$ both belong to $\Z^g + \tfrac{a}{2}$, we obtain
  \begin{displaymath}
    \Dist_\tau(v, \Z^g + \tfrac{a}{2})^2 + \Dist_\tau(v', \Z^g + \tfrac{a}{2})^2 \leq
    2\Dist_\tau(\tfrac{v+v'}{2},\Z^g + \tfrac{a'}2)^2 + 2 \Dist_{\tau}(\tfrac{v-v'}{2}, \Z^g + \tfrac{a+a'}{2})^2.
  \end{displaymath}
  This is the claimed result as $\Dist_{2\tau}^2 = 2\Dist_\tau$.
\end{proof}

\begin{cor}
  \label{cor:simple-precisions}
  Let~$\tau\in \Half_g$, let~$x,x'\in \C^g$, and let~$N\in \Z_{\geq 2}$. Assume
  that the theta values~$\thetatilde_{a,0}(x+x',2\tau)$
  and~$\thetatilde_{a,0}(x-x',2\tau)$ are given to shifted absolute
  precision~$N$ for each $a\in \{0,1\}^g$. Then for each theta characteristic
  $(a,b)$, we can compute $\thetatilde_{a,b}(x,\tau)\thetatilde_{a,b}(x',\tau)$
  to shifted absolute precision~$N - 20(2 + g\log^2 g)$, provided that this
  precision is at least~$2$.
\end{cor}

\begin{proof}
  We use the notation of \cref{prop:dist-inequality}.  For each
  $a'\in \{0,1\}^g$, by \cref{thm:shifted-ubound}, we have
  \begin{align*}
    \abs[\big]{\thetatilde_{a',0}(x+x',2\tau)} &\leq B(g)  \exp\paren[\big]{-2\Dist_\tau(\tfrac{v+v'}{2},\Z^g + \tfrac{a'}2)^2},\\
    \abs[\big]{\thetatilde_{a+a',0}(x-x',2\tau)} &\leq B(g) \exp\paren[\big]{-2\Dist_\tau(\tfrac{v-v'}{2},\Z^g + \tfrac{a+a'}2)^2}.
  \end{align*}
  where~$B(g) = 2^{10(1+g\log^2 g)}$.  Therefore, the absolute error one
  obtains on the product
  $\thetatilde_{a',0}(x+x',2\tau) \thetatilde_{a+a',0}(x-x',2\tau)$ in the
  duplication formula~\eqref{eq:dupl-z-zp} is at most
  \begin{displaymath}
    \eps_{a'} = 2^{-N}(1 + B(g)^2) \exp\paren[\big]{-2\Dist_{\tau}(\tfrac{v+v'}{2},\Z^g + \tfrac{a'}2)^2
      - 2 \Dist_{\tau}(\tfrac{v-v'}{2}, \Z^g + \tfrac{a+a'}{2})^2}.
  \end{displaymath}
  This estimate includes precision losses coming from interval arithmetic. By
  \cref{prop:dist-inequality},
  \begin{displaymath}
    \eps_{a'} \leq 2^{-N} (1 + B(g)^2) \exp\paren[\big]{-\Dist_\tau(v, \Z^g +
      \tfrac{a}{2})^2 - \Dist_\tau(v', \Z^g + \tfrac{a}{2})^2}.
  \end{displaymath}
  The right hand side in the latter inequality is independent of~$a'$. The same
  upper bound holds on~$2^{-g}\sum_{a'} \eps_{a'}$, hence the result.
\end{proof}

Given \cref{cor:simple-precisions}, throughout steps~\eqref{step:ql-rec}
and~\eqref{step:ql-base} of \cref{algo:ql}, it is possible to manipulate theta
values at shifted relative precision~$N$ plus some extra guard bits to account
for precision losses at each step.

In practice, we do not use the simple algorithm from
\cref{cor:simple-precisions} to apply the duplication formula. Instead of
performing $2^g$ products of complex numbers for each individual characteristic
$(a,b)$ separately, it is more efficient to use Hadamard transformations to
compute $\thetatilde_{a,b}(x,\tau)\thetatilde_{a,b}(x',\tau)$ for all
$a\in \{0,1\}^g$ simultaneously (and a fixed~$b$), as noted in
\cite[§4.3]{labrandeComputingThetaFunctions2016}. 

\begin{prop}
  \label{prop:hadamard}
  Let~$g\in \Z_{\geq 1}$, and let~$V_g = \C^{(\Z/2\Z)^g}$, which is a complex
  vector space of dimension~$2^g$. Define the linear operator $H_g: V_g\to V_g$ as
  follows: for each $t = (t_a)_{a\in (\Z/2\Z)^g} \in V_g$ and
  each~$a\in (\Z/2\Z)^g$, set
  \begin{displaymath}
    H_g(t)_a = \sum_{a'\in (\Z/2\Z)^g} (-1)^{a'^T a} t_{a'}.
  \end{displaymath}
  Then for all $t,u\in V_g$, we have
  \begin{displaymath}
    H_g \paren[\big]{H_g (t) \star H_g (u)} = 2^g \paren[\Bigg]{\sum_{a'\in (\Z/2\Z)^g}
      (-1)^{a'^T a} t_{a'} u_{a+a'}}_{a\in (\Z/2\Z)^g}
  \end{displaymath}
  where~$\star$ denotes termwise multiplication.
\end{prop}

The linear operator $H_g$ is the Hadamard transformation. Evaluating~$H_g(t)$
for a given vector~$t$ can be done with $O(2^g)$ additions using a recursive
strategy: if we separate $t$ into two vectors $t_0,t_1\in V_{g-1}$ defined by
\begin{equation}
  \label{eq:rec-hadamard}
  t_0 = (t_{(0,a)})_{a\in (\Z/2\Z)^{g-1}} \quad \text{and} \quad t_1 = (t_{(1,a)})_{a\in (\Z/2\Z)^{g-1}},
\end{equation}
then $H_g(t)$ is the concatenation of $H_{g-1}(t_0) + H_{g-1} (t_1)$ and
$H_{g-1} (t_0) - H_{g-1} (t_1)$.

For us, it is crucial to phrase Hadamard transformations in the setting of
shifted relative precisions, as in steps~\eqref{step:ql-rec}
and~\eqref{step:ql-base} of \cref{algo:ql}, the absolute precision at which
individual entries~$t_a$ are given may heavily depend on~$a$. This issue is
taken care of by the following algorithm.

\begin{algorithm}
  \label{algo:hadamard}
  \algoinput{Vectors
    $t = \paren[\big]{\thetatilde_{a,0}(x+x',2\tau)}_{a\in \{0,1\}^g}$ and
    $u = \paren[\big]{\thetatilde_{a,0}(x-x',2\tau)}_{a\in \{0,1\}^g}$, where
    $(x,\tau)$ and~$(x',\tau)$ belong to~$\red_g$ and
    each entry of~$t$ and~$u$ is given at shifted absolute
    precision~$N \in \Z_{\geq 2}$; the squared
    distances~$\Dist_\tau(w, \Z^g + \tfrac a2)^2$ for all $a\in \{0,1\}^g$ and all
    $w\in \{v, v', \tfrac{v+v'}{2}, \tfrac{v-v'}{2}\}$ where~$v,v'$ are as in
    \cref{prop:dist-inequality}, to some absolute precision~$p\in \Z_{\geq 2}$.}
  \algooutput{The
    vector~$2^{-g}H_g\paren[\big]{H_g(u)\star H_g(v)} =
    \paren[\big]{\thetatilde_{a,0}(x,\tau)\thetatilde_{a,0}(x',\tau)}_{a\in
      \{0,1\}^g}$.}
  \begin{enumerate}
  \item Let~$t',u'$ be the respective centers of~$t,u$. These are vectors with
    exact dyadic entries. Let~$v,v'$ be as in \cref{prop:dist-inequality}.
  \item \label{step:hadamard-bounds} Compute (upper bounds on)
    $m_0,\eps_0\in \R_{\geq 0}$ such that the following inequalities hold,
    keeping notation from \cref{prop:dist-inequality}: for all
    $a\in \{0,1\}^g$,
    \begin{align*}
      \abs{t'_a} &\leq m_0 \exp\paren[\big]{-2\Dist_\tau(\tfrac{v+v'}{2}, \Z^g + \tfrac a2)^2}, \\
      \abs{t_a - t'_a} &\leq \eps_0 \exp\paren[\big]{-2\Dist_\tau(\tfrac{v+v'}{2}, \Z^g + \tfrac a2)^2}.
    \end{align*}
    Similarly compute~$m_1,\eps_1$ satisfying these inequalities
    with~$(t,t',v+v')$ replaced by~$(u,u',v-v')$.
  \item \label{step:hadamard-highprec} Compute (a lower bound on) the maximum
    absolute precision~$N'$ required for an output entry, namely
    \begin{displaymath}
      N' = N + \frac{1}{\log 2}\max_{a\in \{0,1\}^g} \paren[\big]{\Dist_\tau(v, \Z^g + \tfrac a2)^2 + \Dist_\tau(v', \Z^g + \tfrac a2)^2}.
    \end{displaymath}
  \item \label{step:hadamard} Compute
    $w' = 2^{-g} H_g\paren[\big]{H_g(t')\star H_g(u')}$ to absolute precision~$N'$.
  \item \label{step:hadamard-err} For each~$a\in\{0,1\}^g$, add (an upper bound on)
    \begin{displaymath}
      2^g(m_0\eps_1 + m_1\eps_0 + \eps_0\eps_1)\exp\paren[\big]{- \Dist_\tau(v, \Z^g + \tfrac a2)^2 - \Dist_\tau(v', \Z^g + \tfrac a2)^2}
    \end{displaymath}
    to the error radius of~$w'_a$.
  \item Output~$w'$.
  \end{enumerate}
\end{algorithm}

As a bonus, we note that in the special case~$x'=0$, the termwise
multiplications in~$H_g(t')\star H_g(u')$ in step~\eqref{step:hadamard} of
\cref{algo:hadamard} are squarings, which are more efficient than general
multiplications.

\begin{prop}
  \label{prop:algo-hadamard}
  \Cref{algo:hadamard} is correct. If~$(x,\tau)$ and~$(x',\tau)$ are
  chosen as in step~\eqref{step:ql-rec} or~\eqref{step:ql-base} of
  \cref{algo:ql}, then the precision~$N'$ chosen in
  step~\eqref{step:hadamard-highprec} satisfies $N' = O(gN)$,
  \cref{algo:hadamard} runs in time $O(2^g \Mul(gN))$, and its output is computed to
  shifted absolute precision~$N-\Delta$ where~$\Delta= 10(2+g\log^2 g + 2^{-p})$.
\end{prop}

\begin{proof}
  By \cref{prop:hadamard}, we have for every~$a$:
  \begin{displaymath}
    w'_a = \sum_{a'\in (\Z/2\Z)^g} (-1)^{a'^T a} t'_{a'} u'_{a+a'}.
  \end{displaymath}
  If~$w$ denotes the actual vector of products of theta values we wish to
  compute, then
  \begin{displaymath}
    \abs{w_a - w'_a} \leq \sum_{a'\in (\Z/2\Z)^g} \abs{t_a u_{a+a'} - t'_a u'_{a+a'}}.
  \end{displaymath}
  Moreover,
  \begin{align*}
    \abs{t_a u_{a+a'} - t'_a u'_{a+a'}}
    &\leq \abs{(t_a - t'_a)u'_{a+a'}} + \abs{t'_a(u_{a+a'} - u'_{a+a'})} + \abs{(t_a - t'_a)(u_{a+a'} - u'_{a+a'})}\\
    &\leq (m_1 \eps_0 + m_0 \eps_1 + \eps_0\eps_1)\\
      &\qquad\qquad\cdot\exp\paren[\big]{- \Dist_\tau(v, \Z^g + \tfrac a2)^2 - \Dist_\tau(v', \Z^g + \tfrac a2)^2}
  \end{align*}
  by the triangle inequality and \cref{prop:dist-inequality}. Therefore the
  error bound we add in step~\eqref{step:hadamard-err} of \cref{algo:hadamard}
  is indeed correct.

  For the second part of the statement, by construction, the eigenvalues of the
  matrix~$\im(\tau)$ are~$O(N)$. Therefore, for any $v\in \R^g$ and
  any~$a\in \{0,1\}^g$, we have $\Dist_{\tau}(v,\Z^g + \tfrac a2)^2 = O(gN)$,
  and $N' = O(gN)$ in step~\eqref{step:hadamard-highprec} of
  \cref{algo:hadamard}.  By \cref{thm:shifted-ubound},
  step~\eqref{step:hadamard} runs in time $O(2^g \Mul(gN))$. The other steps
  have a negligible cost compared to step~\eqref{step:hadamard}.

  Finally, by \cref{thm:shifted-ubound} and by assumption on the precision of
  the input, we have
  \begin{displaymath}
    m_0 \leq 2^{10(1+\log^2 g)} \exp(2^{-p}) \quad \text{and}\quad \eps_0\leq 2^{-N} \exp(2^{-p}).
  \end{displaymath}
  Thus
  \begin{displaymath}
    m_1\eps_0 + m_0\eps_1 + \eps_0\eps_1 \leq 2^{10(2 + \log^2 g)} 2^{-N} \exp(2^{-p}).
  \end{displaymath}
  The claim on precision losses follows.
\end{proof}

We can now explain how precisions are managed in \cref{algo:ql} when relying on
Hadamard transformations (\cref{algo:hadamard}) at each duplication step.

\begin{prop}
  \label{prop:ql-precisions}
  In the setting of \cref{algo:ql},
  let
  \begin{displaymath}
    \Delta = 10(3 + g\log^2 g) + \abs{\log \eta_{g,h}}.
  \end{displaymath}
  Assume that squared distances are computed up to an absolute error
  of~$2^{-32-h}$ in step~\eqref{step:ql-dist}, and that the theta
  values~$\thetatilde_{a,0}(2^hx, 2^h\tau)$ are computed to shifted absolute
  precision~$N + h\Delta$ in step~\eqref{step:ql-init}. Then each time through
  step~\eqref{step:ql-rec}, the theta values~$\thetatilde_{a,0}(2^jx,2^j\tau)$
  are obtained to shifted absolute precision $N + j\Delta$, and in
  step~\eqref{step:ql-base}, we obtain~$\thetatilde_{a,b}(0,\tau)$
  and~$\thetatilde_{a,b}(z,\tau)$ to shifted absolute precision~$N$. The total
  cost of steps~\eqref{step:ql-init} and~\eqref{step:ql-rec} is
  $O\paren[\big]{2^g h \Mul(gN + g h \Delta) + 4^g \Mul(gN + g\Delta)}$ binary operations.
\end{prop}

\begin{proof}
  When extracting square roots of the theta values appearing in the
  inequalities~\eqref{eq:big-thetaa0-gen}, we lose at most $2$ bits of relative
  precision, hence at most $2 + \abs{\log\eta_{g,h}}$ bits of shifted absolute
  precision. The same observation applies to divisions as performed in
  steps~(\ref{step:ql-rec}.d) and~(\ref{step:ql-base}.c) of \cref{algo:ql}. To
  apply the duplication formulas, we use \cref{algo:hadamard} with~$p = 32$,
  noting that the required distances have all been computed in
  step~\eqref{step:ql-dist} of \cref{algo:ql} up to multiplications by powers
  of~$2$. By \cref{prop:algo-hadamard}, the total number of bits of shifted
  absolute precision that we lose in each duplication step in \cref{algo:ql} is
  at most $\Delta$, and the cost of each step from~$j = h-1$ to $j=1$ is
  bounded above by $O\paren[\big]{2^g \Mul(gN + gh\Delta)}$. When~$j=0$, the
  duplication formula is applied~$O(2^g)$ times, so the cost of that step is
  bounded above by~$O\paren[\big]{4^g \Mul(gN + g\Delta)}$.
\end{proof}

In our implementation, we replace the rather pessimistic number~$h\Delta$ of
necessary guard bits by a more accurate one: rather than using the uniform
value~$\eta_{g,h}$, we use the low-precision approximations of theta values
computed in step~\eqref{step:ql-t} to closely estimate the precision losses at
each duplication step.

\subsection{Lowering the dimension}
\label{subsec:lowerdim}

In this subsection, we focus on step~\eqref{step:ql-init} of \cref{algo:ql}: we
count the number of recursive calls, explain how precisions are managed, and
analyze its complexity (not including the recursive calls
themselves). Analogous results will hold for step~\eqref{step:ql-nothing} in
the case $h=0$. We keep notation from~§\ref{subsec:dupl},
especially~\eqref{eq:split-matrix}.

\begin{prop}
  \label{prop:lowerdim-cost}
  When using step~\eqref{step:ql-init} of \cref{algo:ql} to compute
  $\thetatilde_{a,0}(2^hx,2^h\tau)$ for some fixed~$a\in\{0,1\}^g$
  and~$x\in \{0,t,2t,z+t,z+2t\}$ to shifted absolute
  precision~$N' = N + h\Delta$ for some~$\Delta\geq 0$, the following facts
  hold.
  \begin{enumerate}
  \item In step~(\ref{step:ql-init}.b), let~$S''$ be the set of
    vectors~$n''\in \Z^{g-d}+\tfrac{a''}{2}$ to consider. Then
    \begin{displaymath}
      \#S''= 2^{O((g-d)\log^2 g)} \paren[\Big]{1 + \frac{h\Delta}N}^{g-d},
    \end{displaymath}
    and the points~$n''\in S''$ can be listed in $O(\#S'')$ binary operations.
  \item \label{it:split-prec} In step~(\ref{step:ql-init}.c), it is sufficient to
    compute~$\theta_{a',b'}\paren[\big]{2^h(x'+\sigma n''),2^h\tau_0}$ to
    shifted absolute precision~$N + h\Delta + 2\log(\#S'')$ for each~$n''\in S''$.
  \item \label{it:split-cost} Assume that~$h\Delta \leq N$. Then
    step~\eqref{step:ql-init} of \cref{algo:ql} uses a total of
    $2^{O((g-d)\log^2 g)}\Mul(N)\log(N)$ binary operations, plus
    $2^{O((g-d)\log^2 g)}$ recursive calls to \cref{algo:ql} on the
    matrix~$2^h \tau_0$ in dimension~$d$ at shifted absolute precision
    $N + h\Delta + O((g-d)\log^2 g)$.
  \end{enumerate}
\end{prop}

\begin{proof}
  Without loss of generality, we suppose that ~$x\in\{z+t,z+2t\}$, so that the
  ellipsoids we consider are centered at~$v$. By \cref{thm:shifted-ubound},
  the ellipsoid radius~$R$ that we need to consider to achieve a small enough
  error in~\eqref{eq:approx-theta-split} satisfies
  \begin{displaymath}
    R^2 = \Dist_{2^h\tau}(v, \Z^g + \tfrac a2)^2 + \frac{N'}{\log(2)} + O(g\log N').
  \end{displaymath}
  By construction of~$h$ in \cref{algo:ql}, the diagonal coefficients
  of~$2^{h/2}C_1$ are all bounded below by $2^{-1-\log^2 g}\sqrt{N}$. The
  claimed upper bound on~$\#S'$ thus follows from \cref{cor:eld-nb-pts}. The
  points of~$S''$ can be listed within the claimed complexity bound by
  \cref{prop:ellipsoid}.

  For item~\eqref{it:split-prec}, we fix~$n'' \in S''$ and write
  \begin{displaymath}
    w' = -\im(\tau_0)^{-1} \im(z' + \sigma n'').
  \end{displaymath}
  If we rewrite~\eqref{eq:split-theta} in terms of modified theta
  values~$\thetatilde_{a,b}$, the leading exponential factor in front
  of~$\thetatilde_{a',b'}\paren[\big]{2^h(z'+\sigma n''),2^h\tau_0}$ has absolute value
  $\exp(2^h \pi u)$, where
  \begin{align*}
    u &= - y^T Y^{-1} y + (y' + \im(\sigma)n'')^T \im(\tau_0)^{-1} (y' + \im(\sigma) n'')- n''^T \im(\tau_1) n'' - 2 n''^T y''.
  \end{align*}
  We claim that
  \begin{displaymath}
    \exp(2^h \pi u) \exp \paren[\big]{-\Dist_{2^h\tau_0}(w, \Z^{d} + \tfrac{a'}{2})^2}
    \leq \exp\paren[\big]{-\Dist_{2^h\tau}(v,\Z^g + \tfrac a2)^2}.
  \end{displaymath}

  If this inequality holds and
  $\thetatilde_{a',b'}\paren[\big]{2^h(x'+\sigma n''), 2^h\tau_0}$ is computed
  to shifted absolute precision~$N + h\Delta + 2\log(\#S)$, then the
  corresponding term in the sum defining $\thetatilde_{a,b}(2^hx,2^h\tau)$ can
  be computed to shifted absolute precision $N + h\Delta + 2\log(\#S)$ as well,
  ignoring precision losses. This proves item~\eqref{it:split-prec}, as the
  $2\log(\#S)$ guard bits will absorb precision losses from
  multiplications and the final sum.

  To prove the claim, we get rid of the factors $2^h$, and
  let~$n' \in \Z^{d} + \tfrac{a'}{2}$ be such that
  \begin{displaymath}
    \Dist_{\tau_0}(w, \Z^d + \tfrac {a'}{2}) = \norm{n'-w}_{\tau_0}.
  \end{displaymath}
  The point~$n = (n', n'')$ lies in $\Z^g + \tfrac a2$, so it is sufficient to
  prove that
  \begin{displaymath}
    \pi u -\norm{n'-w}_{\tau_0}^2 = -\norm{n - v}_\tau^2.
  \end{displaymath}
  This equality can be checked directly by expanding the left and right hand sides.

  Finally, for item~\eqref{it:split-cost}, the only step whose complexity
  remains to be analyzed is step~(\ref{step:ql-init}.d). In that step, we first
  compute~$\#S''$ exponential factors to relative precision
  $N + h\Delta + 2\log(\#S'')$, make~$\#S''$ products to the same relative
  precision with the output of step~(\ref{step:ql-init}.c), then add the
  results together working at shifted absolute
  precision $N + h\Delta + 2\log(\#S'')$. This can be done within the claimed
  number of binary operations.
\end{proof}

\begin{rem}
  In the case~$h=0$ in step~\eqref{step:ql-nothing} of \cref{algo:ql}, we could
  be in a setting where~$c_g^2 \gg N$. (This can only happen in the initial
  call, not in recursive ones.) In that case, we are not able to
  output~$\smash{\thetatilde}_{a,b}(z,\tau)$ to shifted absolute precision~$N$,
  because the exponential factors in equation~\eqref{eq:split-theta} could have
  extremely small magnitudes. Fortunately, this is not an issue since we only
  want to compute~$\thetatilde_{a,b}(z,\tau)$ to absolute precision~$N$ in
  \cref{thm:main-intro}.
\end{rem}

\subsection{The existence of a suitable auxiliary vector}
\label{subsec:t-exists}

In this subsection, we show that there exist vectors~$t \in [0,1]^g$ such that
the inequalities~\eqref{eq:big-thetaa0-gen} hold for an explicit value
of~$\eta_{g,h}$. We will take~$t$ to be an exact dyadic vector. The proof
proceeds by approximating the required theta values by multivariate Laurent
polynomials~$P$ in~$\e(t_1),\ldots,\e(t_g)$, then showing that $\abs{P(t)}$
usually attains the expected magnitude~$\norm{P}_\infty$ when~$t\in [0,2]^g$ is
chosen among $2^D$-th roots of unity for a large enough value of~$D$.

\begin{defn}
  \label{def:rel-degrees}
  Let~$P\in \C[X_1,\ldots, X_g, X_1^{-1},\ldots, X_g^{-1}]$ be a Laurent
  polynomial in~$g$ variables. We define the \emph{relative degree of~$P$
    in~$X_g$}, denoted by $\rdeg_{X_g}(P)$, as $-\infty$ if~$P=0$, and
  otherwise as $M-m$ where $m$ (resp~$M$) is the minimal (resp.~maximal)
  exponent of~$X_g$ in~$P$. More generally, fix~$1\leq i\leq g$ and write
  \begin{equation}
    \label{eq:successive-decomposition}
    P = \sum_{n= (n_{i+1},\ldots,n_g)\in \Z^{g-i}} a_n Q_n(X_1,\ldots,X_i) X_{i+1}^{n_{i+1}}\cdots X_g^{n_g}
  \end{equation}
  where the $Q_n$ are Laurent polynomials in~$i$ variables.  Then we define the
  \emph{relative degree of~$P$ in~$X_i$} as
  $\rdeg_{X_i}(P) = \max_n \rdeg_{X_i}(Q_n)$. If $d_1,\ldots, d_g$ are
  integers, we say that~$P$ is of \emph{relative degrees at most}
  $(d_1,\ldots,d_n)$ if $\rdeg_{X_i}(P)\leq d_i$ for each~$1\leq i\leq g$.
\end{defn}

\begin{lem}
  \label{lem:t-dim1}
  Let~$P\in \C[X, X^{-1}]$ be a univariate Laurent polynomial of relative
  degree at most~$d\in \Z_{\geq 0}$.  Let~$D\in \Z_{\geq 2}$. Then the number
  of $D$th roots of unity~$\zeta$ such that
  \begin{displaymath}
    \abs{P(\zeta)}< \frac{1}{d+1} \paren[\Big]{\frac{2}{D}}^{d} \norm{P}_\infty
  \end{displaymath}
  is at most~$d$.
\end{lem}

\begin{proof}
  We may assume that~$D\geq d+1$, and after scaling~$P$ by a suitable of~$X$,
  that~$P$ is an honest polynomial of degree at most~$d$. We argue by
  contradiction, and fix $d+1$ $D$th roots of unity, denoted by
  $\zeta_1,\ldots,\zeta_{d+1}$ where the above inequality holds. By the
  Lagrange interpolation formula,
  \begin{displaymath}
    P_i = \sum_{j=1}^{d+1} P(\zeta_j) \frac{\prod_{k\neq j} (X - \zeta_k)}{\prod_{k\neq j}(\zeta_j - \zeta_k)}.
  \end{displaymath}
  We have $\norm[\big]{\prod_{k\neq j} (X - \zeta_k)}_\infty \leq 2^{d}$, and
  \begin{displaymath}
    \abs[\bigg]{\prod_{k\neq j}{(\zeta_j - \zeta_k)}} \geq 2^{d} \sin^{d}\paren[\Big]{\frac{\pi}{D}} \geq \paren[\Big]{\frac{4}{D}}^{d}.
  \end{displaymath}
  We obtain the contradiction
  % \begin{displaymath}
  $\displaystyle\norm{P}_\infty \leq \paren[\Big]{\frac{D}{2}}^{d} \sum_{j=1}^{d+1} \abs[\big]{P(\zeta_j)}
  < \norm{P}_\infty.$ \qedhere
  % \end{displaymath}
\end{proof}

\begin{prop}
  \label{prop:t-dimg}
  Let~$P\in \C[X_1,\ldots,X_g,X_1^{-1},\ldots,X_g^{-1}]$ be a multivariate
  Laurent polynomial of relative degrees at most~$(d_1,\ldots,d_g)$, where
  $d_i\in \Z_{\geq 1}$ for each $1\leq i\leq g$. Let~$D\in \Z_{\geq 2}$. Then
  among the $D^g$ vectors~$(\zeta_1,\ldots,\zeta_g)\in \C^g$ whoose coordinates are $D$th roots of
  unity, at least $(D-d_1)\cdots (D-d_g)$ have the property that
  \begin{displaymath}
    \abs{P(\zeta_1,\ldots,\zeta_g)}\geq \paren[\Big]{\frac{2}{D}}^{d_1+\cdots+d_g} \norm{P}_\infty \prod_{i=1}^g \frac{1}{d_i+1}.
  \end{displaymath}
\end{prop}

\begin{proof}
  We proceed by induction on~$g$. The case $g=1$ is covered by
  \cref{lem:t-dim1}, so suppose $g>1$. Let~$m\in \Z^g$ be such that
  $\abs{a_m} = \norm{P}_\infty$, and consider the univariate Laurent polynomial
  \begin{displaymath}
    Q := \sum_{n\in \Z} a_{m_1,\ldots,m_{g-1}, n} X^n,
  \end{displaymath}
  which satisfies~$\norm{Q}_\infty = \norm{P}_\infty$.  By \cref{lem:t-dim1},
  at least~$D - d_g$ of the $D$th roots of unity~$\zeta_g$ satisfy
  \begin{displaymath}
    \abs{Q(\zeta_g)}\geq \frac{1}{d+1} \paren[\Big]{\frac{2}{D}}^{d} \norm{P}_\infty.
  \end{displaymath}

  For each such~$\zeta_g$, we then consider the
  polynomial~$R = P \paren[\big]{X_1,\ldots,X_{g-1}, \zeta_g}$ in $g-1$
  variables. The coefficient of~$X_1^{m_1}\cdots X_{g-1}^{m_{g-1}}$ in~$R$
  is~$Q(\zeta_g)$, so $\norm{R}_\infty\geq \abs{Q(\zeta_g)}$. Further, $R$ has
  relative degrees at most~$(d_1,\ldots,d_{g-1})$. By the induction hypothesis,
  at least $(D-d_1)\cdots(D-d_{g-1})$ vectors of $D$th roots of unity
  $(\zeta_1,\ldots,\zeta_{g-1})$ will satisfy
  \begin{displaymath}
    \abs[\big]{R\paren[\big]{\zeta_1,\ldots,\zeta_{g-1}}} \geq \norm{R}_\infty
    \paren[\Big]{\frac{2}{D}}^{d_1+\ldots+d_{g-1}} \prod_{i=1}^{g-1} \frac{1}{d_i+1}.
  \end{displaymath}
  Thus $(\zeta_1,\ldots,\zeta_g)$ satisfies the inequality claimed in the
  proposition. This provides the required number of vectors as~$\zeta_g$ varies.
\end{proof}

When applying \cref{prop:t-dimg} in the context of theta functions, a technical
issue is that the relative degrees of the polynomials involved will depend on
how close the theta values are being approximated, as the dependency
on~$\delta$ in \cref{prop:eld-width} shows.  Despite this, we prove the
following result.

\begin{thm}
  \label{thm:choose-t}
  Fix a point~$(z,\tau)\in \red_g$ and $h\in \Z_{\geq
    1}$. Let~$\eta_{g,h}\in \R_{>0}$ be such that
  \begin{displaymath}
    \log(1/\eta_{g,h}) = 2^{25(1+\log^2 g)}h^2 \log(1+h),
  \end{displaymath}
  and let~$D$ be any power of~$2$ such that
  \begin{displaymath}
    D\geq 2^{20(1+\log^2 g)}\paren[\big]{2^{2g} + 2^{g+h}}.
  \end{displaymath}
  Then at least half of the values $t\in \set[\big]{0, \tfrac 1D,\ldots,\tfrac{D-1}{D}}^g$ will satisfy
  the inequalities~\eqref{eq:big-thetaa0-gen} for the above value
  of~$\eta_{h,g}$.
\end{thm}

\begin{proof}
  Fix values~$\eta\in \R_{>0}$ and~$D\in \Z_{\geq 1}$ that we will adjust
  later. First, we consider the last inequality
  in~\eqref{eq:big-thetaa0-gen}. Fix $a,b\in \{0,1\}^g$. In order to ensure
  that, say
  \begin{displaymath}
    \abs[\big]{\thetatilde_{a,b}(z+2t, \tau)}
    \geq \eta \exp\paren[\big]{-\Dist_\tau(v, \Z^g + \tfrac a2)^2},
  \end{displaymath}
  we will pick a multivariate Laurent polynomial
  $P_{a,b,0}\in \C[X_1,\ldots,X_g, X_1^{-1},\ldots,X_g^{-1}]$ such that
  \begin{displaymath}
    \abs[\big]{\thetatilde_{a,b}(z+2t,\tau) - P_{a,b,0}\paren[\big]{\e(2t_1),\ldots, \e(2t_g)}}
    \leq \eta \exp\paren[\big]{-\Dist_\tau(0, \Z^g + \tfrac a2)^2}
  \end{displaymath}
  and then make sure that for most~$t \in \set[\big]{0,\tfrac 1D,\ldots,\tfrac{D-1}{D}}^g$, we have
  \begin{displaymath}
    \abs[\big]{P_{a,b,0}\paren[\big]{\e(2t_1),\ldots, \e(2t_g)}}\geq 2\eta
    \exp\paren[\big]{-\Dist_\tau(0, \Z^g + \tfrac a2)^2}.
  \end{displaymath}
  Similary, for each~$1\leq j\leq h-1$ and~$a\in \{0,1\}^g$, we will pick
  multivariate Laurent polynomials $P_{a,0,j}, P_{a,1,j}, P_{a,2,j}, P_{a,3,j}$
  such that
  \begin{align*}
    \abs[\big]{\theta_{a,0}(2^jt,2^j\tau) - P_{a,0,j}\paren[\big]{\e(2^jt_1),\ldots, \e(2^jt_g)}}
    &\leq \eta \exp\paren[\big]{-2^j\Dist_\tau(0, \Z^g + \tfrac a2)^2}, \\
    \abs[\big]{\theta_{a,0}(2^{j+1}t,2^j\tau) - P_{a,1,j}\paren[\big]{\e(2^jt_1),\ldots, \e(2^jt_g)}}
    &\leq \eta \exp\paren[\big]{-2^j\Dist_\tau(0, \Z^g + \tfrac a2)^2}, \\
    \abs[\big]{\thetatilde_{a,0}(2^j(z+t),2^j\tau) - P_{a,2,j}\paren[\big]{\e(2^jt_1),\ldots, \e(2^jt_g)}}
    &\leq \eta \exp\paren[\big]{-2^j\Dist_\tau(v, \Z^g + \tfrac a2)^2}, \\
    \abs[\big]{\thetatilde_{a,0}(2^j(z+2t),2^j\tau) - P_{a,3,j}\paren[\big]{\e(2^jt_1),\ldots, \e(2^jt_g)}}
    &\leq \eta \exp\paren[\big]{-2^j\Dist_\tau(v, \Z^g + \tfrac a2)^2},
  \end{align*}
  and then make sure that their values are bounded below most of the time.  In
  total, there are $2^{2g} + (h-1)2^{g+2}$ multivariate polynomials to
  consider.

  We first seek upper bounds on the relative degrees of all these polynomials
  in terms of~$g$ and~$\eta$. By \cref{thm:shifted-ubound}, we can take
  \begin{displaymath}
    P_{a,b,0} = \exp(-\pi y^T Y^{-1} y) \sum_{n\in (\Z^g + \frac a2)\cap \eld{v}{R}{\tau}}
    \e \paren[\big]{n^T\tau n + 2 n^T(z + \tfrac b2)} X_1^{2n_1}\cdots X_g^{2n_g}
  \end{displaymath}
  where~$R\in \R_{\geq 2}$ is chosen such that
  \begin{displaymath}
    R^2 = \Dist_\tau(0, \Z^g + \tfrac{a}{2})^2 + \delta_0^2
  \end{displaymath}
  and~$\delta_0$ satisfies the inequalities
  \begin{displaymath}
    \delta_0\geq 2 \quad\text{and}\quad \delta_0^{g-1} \exp(-\delta_0^2) \leq \frac{\eta}{B(g)}.
  \end{displaymath}
  In particular, we can take
  \begin{equation}
    \label{eq:delta0}
    \delta_0 = \sqrt{\log(1/\eta)} + 20 \paren[\big]{1 + g\log^2 g + g\log\log(1/\eta)}.
  \end{equation}
  By \cref{prop:hkz-bound,prop:eld-width}, the relative degrees
  $d_1,\ldots,d_g$ of~$P_{a,b,0}$ are all bounded above by
  $g^{1+\log(g)/2} + 2g \delta_0 \leq 2^{1+\log^2 g} + 2g\delta_0$.

  Similarly, for each~$1\leq j\leq h$, we can take
  \begin{displaymath}
    P_{a,0,j} = \exp(-2^j \pi y^T Y^{-1} y) \sum_{n\in (\Z^g + \frac a2)\cap \eld{v}{R}{2^j\tau}}
    \e \paren[\big]{2^j n^T\tau n} X_1^{2n_1}\cdots X_g^{2n_g}
  \end{displaymath}
  where~$R\in \R_{\geq 2}$ is chosen such that
  \begin{displaymath}
    R^2 = \Dist_{2^j\tau}(0, \Z^g + \tfrac{a}{2})^2 + \delta_0^2.
  \end{displaymath}
  We then have~$\eld{v}{R}{2^j\tau} = \eld{v}{R'}{\tau}$ where
  \begin{displaymath}
    R'^2 = \Dist_{\tau}(0, \Z^g + \tfrac{a}{2})^2 + 2^{-j}\delta_0^2.
  \end{displaymath}
  By \cref{prop:eld-width} again, the relative degrees $d_1,\ldots,d_g$
  of~$P_{a,0,j}$ are all bounded above by $2^{1+\log^2 g}  2^{1 -
    j/2}g\delta_0$. Because~$P_{a,1,j}$ approximates theta values at~$2^{j+1}t$
  instead of~$2^j t$, its relative degrees are bounded above by
  $2^{2+\log^2 g} + 2^{2 - j/2}g \delta_0$. This upper bound on the relative degrees also
  holds for~$P_{a,2,j}$ and~$P_{a,3,j}$.

  Next, we choose~$D$ such that, when applying \cref{prop:t-dimg} on all the
  polynomials $P_{a,b,j}$, the total proportion of discarded vectors~$t$ (those
  for which one of the inequalities doesn't hold) is at most $1/2$. In other
  words, we want to ensure that
  \begin{displaymath}
    2^{2g} \paren[\big]{D^g - (D - 2^{1+\log^2 g} - 2g\delta_0)^g}
    + \sum_{j=1}^{h-1} 2^{g+2} 2^{j-1} \paren[\big]{D^g - (D - 2^{2+\log^2 g} - 2^{2 - j/2}g\delta_0)^g}
    \leq \frac{D^g}{2}.
  \end{displaymath}
  The factor~$2^{j-1}$ in the sum comes from the fact that the
  map~$t\mapsto \paren[\big]{\e(2^jt_1),\ldots\e(2^j t_g)}$ on the considered
  domain is~$2^{j-1}$-to-1. Since~$(1-\alpha)^g\geq 1-g\alpha$ for
  nonnegative~$\alpha$, it is sufficient to enforce that
  \begin{displaymath}
    2^{2g} \cdot g(2^{1+\log^2 g} + 2g\delta_0) + 2^{g+2}\cdot g \sum_{j=1}^{h-1} 2^j(2^{1 + \log^2 g} + 2^{1-j/2}g\delta_0) \leq \frac D2.
  \end{displaymath}
  This last inequality certainly holds whenever
  \begin{equation}
    \label{eq:D}
    D \geq 4\paren[\big]{2^{2g+2\log^2 g} + 2^{2g+2\log g}\delta_0 + 2^{g+2\log^2 g + h} + 2^{g + 2\log g + h/2}\delta_0 }.
  \end{equation}

  Finally, we choose a suitable value of~$\eta$ so that \cref{prop:t-dimg} will
  provide the claimed result. By construction, for all~$a,b,j$, we have
  \begin{align*}
    \norm{P_{a,b,0}}_\infty &= \exp\paren[\big]{-\Dist_\tau(v,\Z^g + \tfrac a2)^2},\\
    \norm{P_{a,0,j}}_\infty = \norm{P_{a,1,j}}_\infty &= \exp\paren[\big]{-2^j\Dist_\tau(0,\Z^g + \tfrac a2)^2},\\
    \norm{P_{a,2,j}}_\infty = \norm{P_{a,3,j}}_\infty &= \exp\paren[\big]{-2^j\Dist_\tau(v,\Z^g + \tfrac a2)^2}.
  \end{align*}
  Therefore, we obtain the claimed result as soon as
  \begin{equation}
    \label{eq:eta}
    \paren[\Big]{\frac 2D}^{g(2^{2+\log^2 g} + 4g\delta_0)} \cdot \paren[\big]{2^{2+\log^2 g} + 4g\delta_0}^{-g} \geq 2\eta.
  \end{equation}

  Our task is now to produce explicit values of~$D,\eta$ and~$\delta$ in terms
  of~$g$ and~$h$ such that~\eqref{eq:delta0},~\eqref{eq:D} and~\eqref{eq:eta}
  hold simultaneously. We will assume~$\eta\leq 2^{-10}$ and perform crude
  simplifications to keep the expressions manageable. Taking $D$ to be the
  smallest power of~$2$ such that~\eqref{eq:D} holds, we have
  \begin{displaymath}
    \log(D/2) \leq \log(\delta_0) + 2g + 2\log^2 g + h + 4 \leq 8g^2 h \log(\delta_0).
  \end{displaymath}
  Moreover, by~\eqref{eq:delta0},
  \begin{displaymath}
    \delta_0\leq 40 g^2 \sqrt{\log(1/\eta)}
  \end{displaymath}
  so
  \begin{displaymath}
    \log(\delta_0)\leq \frac 12 \log\log(1/\eta) + 2\log g + 10
  \end{displaymath}
  and
  \begin{displaymath}
    \delta_0\log(\delta_0)\leq 2^{10}g^2 \sqrt{\log(1/\eta)} \log\log(1/\eta).
  \end{displaymath}
  Taking logarithms in~\eqref{eq:eta}, a sufficient condition on~$\eta$ is that
  \begin{displaymath}
    g(2^{2+\log^2 g} + 4g\delta_0)  \cdot 8g^2h\log(\delta_0)
    +g \log\paren[\big]{2^{2+\log^2 g} + 4g\delta_0} + \log 2 \leq \log(1/\eta).
  \end{displaymath}
  This holds whenever
  \begin{displaymath}
    2^{6 + 4\log^2 g} h\log(\delta_0) + 2^6g^3h\delta_0 \log(\delta_0) \leq \log(1/\eta).
  \end{displaymath}
  In turn, the previous inequality holds if
  \begin{displaymath}
    2^{10 + 5\log^2 g}h\log\log(1/\eta) + 2^{16}g^5 h\sqrt{\log(1/\eta)} \log\log(1/\eta) \leq \log (1/\eta).
  \end{displaymath}
  Therefore, taking~$\eta = \eta_{g,h}$ as in the statement of the theorem
  indeed works with our chosen values of~$\delta_0$
  and~$D$. From~\eqref{eq:delta0} and~\eqref{eq:D}, we then have
  \begin{displaymath}
    D\leq 2^{20(1+\log^2 g)}\paren[\big]{2^{2g} + 2^{g+h}}.
  \end{displaymath}

  Finally, we show that the conclusion of~\cref{thm:choose-t} remains valid
  when a larger value~$D'\geq D$ (still assumed to be a power of~$2$) is chosen
  instead of our particular~$D$. Consider the partition of~$(\Z/D'\Z)^g$ into
  $k = (D'/D)^g$ cosets for its subgroup~$(\Z/D\Z)^g$. Mapping this partition
  through the obvious bijections
  \begin{displaymath}
    \set[\big]{0,\tfrac{1}{D'},\ldots,\tfrac{D'-1}{D'}}^g \simeq (\Z/D'\Z)^g \quad\text{and}\quad
     \set[\big]{0,\tfrac{1}{D},\ldots,\tfrac{D-1}{D}}^g \simeq (\Z/D\Z)^g,
  \end{displaymath}
  we obtain a partition of
  $\set[\big]{0,\tfrac{1}{D'},\ldots,\tfrac{D'-1}{D'}}$ into $k$ subsets of
  cardinality~$D^g$. Call these subsets~$S_1,\ldots,S_k$.

  Sampling~$t$ from any fixed~$S_i$ is equivalent to sampling~$t$
  from~$\set[\big]{0,\tfrac{1}{D},\ldots,\tfrac{D-1}{D}}^g$, except that the
  coefficients of the polynomials~$P_{a,b,j}$ in the above proof are scaled by
  certain $2^{D'}$-th roots of unity. Since this does not modify the relative
  degrees or the norms of these polynomials, the above proof carries through
  and implies that at least half of the values~$t\in S_i$ satisfy the
  inequalities~\eqref{eq:big-thetaa0-gen} for the given value
  of~$\eta_{h,g}$. Considering all subsets~$S_i$ simultaneously, we conclude
  that at least half of the
  values~$t\in \set[\big]{0,\tfrac{1}{D'},\ldots,\tfrac{D'-1}{D'}}^g$ will
  indeed satisfy~\eqref{eq:big-thetaa0-gen}.
\end{proof}

\begin{rem}
  \label{rem:several-z} If we run the quasi-linear algorithm simultaneously
  on~$n$ vectors~$z$ for the same matrix~$\tau$ as suggested in
  \cref{rem:ql-improvements}, we should mutualize the auxiliary vector~$t$. The
  only modification in the proof of \cref{thm:choose-t} is that the lower bound
  on~$D$ in~\eqref{eq:D} should be multiplied by~$n$, so the statement could be
  adapted to handle several vectors at once.
\end{rem}

\subsection{Choosing an auxiliary vector}
\label{subsec:choose-t}

In this subsection, we provide and analyze two algorithms to choose~$t$ and
compute low-precision theta values as in step~\eqref{step:ql-t} of
\cref{algo:ql}: first a very efficient probabilistic method (choosing~$t$ at
random) that we implemented, then a slower, deterministic one. In any case, experiments suggest
that \cref{thm:choose-t} is far from being optimal, so we first choose more
optimistic values for~$\eta_{g,h}$ and~$D$, and modify them if necessary. It is
also a good idea to check whether~$t=0$ works, as \cref{algo:ql} is simpler in
that case.

\begin{algorithm}
  \label{algo:proba-t}

  \algoinput{An exact point~$(z,\tau)\in \red_g$; a
    precision $N\in \Z_{\geq 2}$; $h\in \Z_{\geq 1}$ chosen as in
    step~\eqref{step:ql-h} of \cref{algo:ql}.}  \algooutput{A dyadic
    vector~$t\in [0,1]^g$ and a dyadic $\eta_{g,h}\in \R_{>0}$ such that the
    inequalities~\eqref{eq:big-thetaa0-gen} hold; low-precision approximations
    of~$\thetatilde_{a,b}(z+2t,\tau)$ for all~$a,b\in \{0,1\}^g$ that do not
    contain zero as complex balls; and for each~$1\leq j\leq h-1$,
    low-precision approximations of~$\thetatilde_{a,0}(2^jx,2^j\tau)$ for all
    $x\in \{t, 2t, z+t, z+2t\}$ and~$a\in \{0,1\}^g$ that do not contain zero
    as complex balls.}

  \begin{enumerate}
  \item \label{step:proba-t-init} Let~$\eta\in \R_{>0}$ be such
    that~$\log_2(1/\eta) = \ceil[\big]{10 h^2\log_2(1+h)}$, and let
     $D=2^{4 + g + h}$. Let~$t=0$ and~$n=0$.
   \item \label{step:proba-t-roots-0} For each~$a\in \{0,1\}^g$, use
     \cref{algo:sum-naive} together with \cref{algo:compute-R}.\textbf{B} to
     evaluate $\thetatilde_{a,b}(z+2t,\tau)$ up to an absolute error of
     $\tfrac 12 \eta \exp(-\Dist_\tau^2(v,\Z^g + \tfrac a2))$ simultaneously
     for all $b\in \{0,1\}^g$. If any of the resulting values contain zero,
     then go to step~\eqref{step:proba-t-fail}.
   \item \label{step:proba-t-roots} For each~$1\leq j\leq h-1$, each
     $a\in \{0,1\}^g$, and each $x\in \{t,2t\}$, use
     \cref{algo:sum-naive,algo:compute-R}.\textbf{B} to evaluate
     $\theta_{a,0}(2^jx, 2^j\tau)$ up to an absolute error of
     $\tfrac 12 \eta\exp(-2^j\Dist_\tau^2(0,\Z^g + \tfrac a2))$. Also evaluate
     $\thetatilde_{a,0}(2^jx, 2^j\tau)$ for~$x\in\{z+t,z+2t\}$ up to an
     absolute error of
     $\tfrac 12\eta\exp(-2^j\Dist_\tau^2(v, \Z^g + \tfrac a2))$. If any of the
     resulting values contain zero as a complex ball, then go to
     step~\eqref{step:proba-t-fail}.
  \item \label{step:proba-t-output} \textbf{Success.} Output~$t$, $\eta$, and the
    theta values computed in steps~\eqref{step:proba-t-roots-0}
    and~\eqref{step:proba-t-roots} for this value of~$t$.
  \item \label{step:proba-t-fail} \textbf{Failure.} If~$n < 4$, then
    let~$n = n+1$, let~$t\in [0,1]^g$ be a uniformly random value with
    denominator~$D$, and go back to
    step~\eqref{step:proba-t-roots-0}. Otherwise, replace~$\eta$ by~$\eta'$
    such that~$\log_2(1/\eta') = \ceil{2^{1+\log^2 g}}\log_2(1/\eta)$, replace~$D$
    by~$2^g D$, let~$t=0$, let~$n=0$, and go back to
    step~\eqref{step:proba-t-roots-0}.
  \end{enumerate}
\end{algorithm}

\begin{rem}
  \label{rem:zero-odd}
  If~$t=0$ and~$z=0$, then step~\eqref{step:proba-t-roots-0} will always fail
  because odd theta constants are identically zero. In that case, we may omit
  the odd pairs~$(a,b)$ from step~\eqref{step:proba-t-roots-0}. Then, in
  step~\eqref{step:ql-base} of \cref{algo:ql}, we obtain~$\theta_{a,b}(0,\tau)$
  using a square root when~$a^T b$ is even, and output~$\theta_{a,b}(0,\tau)=0$
  if~$a^Tb$ is odd.
\end{rem}

\begin{prop}
  \label{prop:setup-proba-cost}
  With high probability, the probabilistic algorithm~\ref{algo:proba-t}
  succeeds and outputs~$t,\eta$ such that
  $\log(1/\eta) = 2^{O(\log^2 g)}h^2\log(1+h)$. Its expected running time is at
  most $2^{O(g\log^2 g)} \log(N)^{g+4}$.
\end{prop}

\begin{proof}
  Given \cref{thm:choose-t} and the initial choices of~$\eta$ and~$D$ in
  step~\eqref{step:proba-t-init}, the expected number of times through
  step~\eqref{step:proba-t-fail} is~$O(1)$. We
  analyze the cost of each step as follows:
  \begin{itemize}
  \item Step~\eqref{step:proba-t-init} costs~$O(\log(1/\eta))$.
  \item In step~\eqref{step:proba-t-roots-0}, by \cref{prop:algo-R-correct},
    \cref{algo:compute-R}.\textbf{B} outputs a radius~$R$ such
    that~$R^2 = \Dist_\tau^2(v,\Z^g + \tfrac a2) + O(\log(1/\eta))$. By
    \cref{prop:ellipsoid} and \cref{cor:eld-nb-pts}, the total number of points
    $n\in \E$ we consider in the summation algorithm~\ref{algo:sum-naive} is at
    most~$2^{O(g\log^2 g)} \log(1/\eta)^{g/2}$. Because each term in the theta
    series has to be computed to relative precision $O(\log(1/\eta))$, the
    total cost of this step is~$2^{O(g\log^2 g)} \log(1/\eta)^{2+ g/2}$ binary
    operations.
  \item Similarly, step~\eqref{step:proba-t-roots} runs in time
    $2^{O(g\log^2 g)} h \log(1/\eta)^{2 + g/2}$; recall $h = O(\log N)$.
  \item Step~\eqref{step:proba-t-output} is free, and
    step~\eqref{step:proba-t-fail} costs~$O(\log(1/\eta))$. \qedhere
  \end{itemize}
\end{proof}

From a theoretical perspective, it is desirable to choose the auxiliary
vector~$t$ in a deterministic way instead. However, simply replacing
step~\eqref{step:proba-t-fail} of \cref{algo:proba-t} by some arbitrary
deterministic way of choosing~$t$ would not yield an algorithm with a
reasonable time complexity. Instead, we go back to the proof of
\cref{prop:t-dimg} and choose a suitable value of~$t_g$ first, then~$t_{g-1}$,
etc., up to~$t_1$. In addition, when choosing~$t_g$ for example, the
inequalities~\eqref{eq:big-thetaa0-gen} for a given~$0 \leq j \leq h-1$ only
depend on~$t_g$ modulo $1/2^{\max\{0,j-1\}}$. To achieve a reasonable
complexity bound, we must first choose~$t_g$ modulo~$1/2^{\max\{0,h-2\}}$, then
(assuming~$h\geq 3$) lift this value modulo $1/2^{h-3}$, etc. To simplify the
description of the algorithm, we make the following definition.

\begin{defn}
  \label{def:tree}
  Let~$h\geq 1$ and~$0\leq t < 1/2^h$. We define the tree $T(t,h)$ as the
  following labeled binary tree of depth~$h$: at each depth~$0\leq j\leq h$,
  the vertices of~$T(t,h)$ consist of the values $t + u/2^h$ where~$u$ is an
  integer satisfying $0\leq u< 2^{j}$; if~$j\geq 1$, those vertices are
  connected to the unique vertex at level~$j-1$ that has a compatible label
  modulo~$1/2^{h-j+1}$. An example is depicted on \cref{fig:tree}.
\end{defn}

\begin{figure}[ht]
  \centering
  \begin{tikzpicture}
    \tikzstyle{mynode}=[minimum width=1.4cm,minimum height=0.8cm, rectangle,draw]
    \node[mynode] (s) at (0,-0.2) {$t$};
    \node[mynode] (0) at (-3.6,1.4) {$t$};
    \node[mynode] (1) at (3.6,1.4) {$t + \tfrac 18$};
    \node[mynode] (00) at (-5.4,2.8) {$t$};
    \node[mynode] (01) at (-1.8,2.8) {$t + \tfrac 14$};
    \node[mynode] (10) at (1.8,2.8) {$t + \tfrac 18$};
    \node[mynode] (11) at (5.4,2.8) {$t + \tfrac 38$};
    \node[mynode] (000) at (-6.3,4.2) {$t$};
    \node[mynode] (001) at (-4.5,4.2) {$t + \tfrac 12$};
    \node[mynode] (010) at (-2.7,4.2) {$t + \tfrac 14$};
    \node[mynode] (011) at (-0.9,4.2) {$t + \tfrac 34$};
    \node[mynode] (100) at (0.9,4.2) {$t + \tfrac 18$};
    \node[mynode] (101) at (2.7,4.2) {$t + \tfrac 58$};
    \node[mynode] (110) at (4.5,4.2) {$t + \tfrac 38$};
    \node[mynode] (111) at (6.3,4.2) {$t + \tfrac 78$};
    \draw (s.north) -- (0.south);
    \draw (s.north) -- (1.south);
    \draw (0.north) -- (00.south);
    \draw (0.north) -- (01.south);
    \draw (1.north) -- (10.south);
    \draw (1.north) -- (11.south);
    \draw (00.north) -- (000.south);
    \draw (00.north) -- (001.south);
    \draw (01.north) -- (010.south);
    \draw (01.north) -- (011.south);
    \draw (10.north) -- (100.south);
    \draw (10.north) -- (101.south);
    \draw (11.north) -- (110.south);
    \draw (11.north) -- (111.south);
  \end{tikzpicture}
  \caption{The tree $T(t,3)$.}
  \label{fig:tree}
\end{figure}

\begin{algorithm}
  \label{algo:deterministic-t}~\\
  \textbf{Input/Output:} Same as \cref{algo:proba-t}.
  \begin{enumerate}
  \item \label{step:deterministic-t-eta} Let~$\eta\in \R_{>0}$ be such
    that~$\log_2(1/\eta) = \ceil[\big]{10 h^2 \log_2(1+h)}$, and
    let~$D = 2^{4 + g + h}$.
  \item \label{step:deterministic-t-closest} For each~$a\in \{0,1\}^g$, choose
    points~$w_a,w'_a\in \Z^g + \tfrac a2$ such that
    \begin{displaymath}
      \norm{w_a}_\tau = \Dist_\tau(0, \Z^g + \tfrac a2) \quad\text{and}\quad
      \norm{v - w'_a}_\tau = \Dist_\tau(v, \Z^g + \tfrac a2).
    \end{displaymath}
  \item \label{step:deterministic-d-delta} Compute~$\delta_0$ as in
    equation~\eqref{eq:delta0}.  Let~$d = 2^{2+\log^2 g} + 4g\delta_0$.
  \item \label{step:deterministic-polys} For each triple~$(a,b,j)$
    where~$a\in \{0,1\}^g$, $0\leq j\leq h-1$, and $b\in \{0,1\}^g$ if~$j=0$
    and $b\in \{0,1,2,3\}$ otherwise, write down the polynomial~$P_{a,b,j}$
    from the proof of \cref{thm:choose-t}. Let~$P'_{a,b,j} = P_{a,b,j}$.
  \item \label{step:deterministic-tk} For~$k=g$ down to~$1$, determine a
    suitable value of~$t_k\in [0,1]$ with denominator~$D$ using the following
    steps, assuming that~$t_{k+1},\ldots,t_g$ have already been determined and
    that we have computed the polynomials
    \begin{displaymath}
      P'_{a,b,j} = P_{a,b,j}\paren[\big]{X_1,\ldots,X_k, \e(2^{\max(1,j)}t_{k+1}),\ldots, \e(2^{\max(1,j)}t_g)}
      = \sum_{m\in \Z^k} a_m X_1^{m_1}\cdots X_k^{m_k}.
    \end{displaymath}
    \begin{enumerate}
    \item For each triple~$(a,b,j)$, let~$m\in \Z^k$ be the vector consisting
      of the first~$k$ coordinates of~$w_a$ (if~$j\geq 1$ and~$b\in\{0,1\}$)
      or~$w'_a$ (otherwise). Let
      \begin{displaymath}
        Q_{a,b,j} = \sum_{n\in \Z} a_{m_1,\ldots,m_{k-1},n} X^n \in \C[X,X^{-1}].
      \end{displaymath}
      Let~$u_k = 0$.
    \item Find~$t_k$ with denominator~$D$ such
      that~$u_k\leq t_k < 1/2^{\max\{0,h-1\}}$ and the inequality
      \begin{displaymath}
        \abs[\big]{Q_{a,b,j}(\e(2^{\max(1,j)}t_k))} \geq \frac{1}{d+1} \paren[\Big]{\frac 2D}^{d} \norm{Q_{a,b,j}}_\infty
      \end{displaymath}
      holds for all pairs $(a,b)$ and~$j=h-1$, by enumerating the possible
      values of~$t_k$ one by one, and evaluating these polynomials. If none of
      the possible values of~$t_k$ works, go to
      step~\eqref{step:deterministic-fail}; otherwise let~$u_k = t_k + 2^{-D}$
      and continue.
    \item Search through the tree $T(t_k, \max\{0,h-2\})$, depth-first, to
      lift~$t_k$ modulo~$1$ in such a way that the inequalities from
      step~(\ref{step:deterministic-tk}.b) hold for all values of
      $0\leq j\leq h-1$. Upon encountering a vertex of this tree of a given
      depth~$l$, we test whether the inequalities involving $2^{h-2-l}t_k$ all
      hold; if not, then the entire branch of the tree spanning from this
      vertex is discarded. If the search fails, go back to
      step~(\ref{step:deterministic-tk}.b); otherwise continue.
    \item For each triple~$(a,b,j)$, replace~$P'_{a,b,j}$ by the result of its
      partial evaluation at~$X_k = \e(2^{\max(1,j)}t_k)$.
    \end{enumerate}
  \item \label{step:deterministic-t-err} \textbf{Success.} For each
    triple~$(a,b,j)$, add~$\eta \norm{P_{a,b,j}}_\infty$ to the error bound of
    the complex ball~$P'_{a,b,j}$ (now a complex number). Output~$t$, $\eta$,
    and the collection of values $P'_{a,b,j}$  as the required low-precision
    approximations of theta values.
  \item \label{step:deterministic-fail} \textbf{Failure.} Replace~$\eta$
    by~$\eta'$ such
    that~$\log_2(1/\eta') = \ceil{2^{1+\log^2 g}}\log_2(1/\eta)$, replace~$D$
    by~$2^g D$, and go back to step~\eqref{step:deterministic-d-delta}.
  \end{enumerate}
\end{algorithm}

\begin{prop}
  \label{prop:deterministic-t}
  Keep notation from \cref{algo:deterministic-t}, and assume for simplicity
  that all computations with complex numbers are exact.  For
  each~$k\in \{g,\ldots,1\}$ and all triples~$(a,b,j)$ as above, we have in
  step~(\ref{step:deterministic-tk}.a) of \cref{algo:deterministic-t}:
  \begin{displaymath}
    \norm{Q_{a,b,j}}_\infty \geq \paren[\Big]{\frac{1}{d + 1}}^{g - k} \paren[\Big]{\frac 2D}^{(g-k)d} \norm{P_{a,b,j}}_\infty.
  \end{displaymath}
  After applying step~\eqref{step:deterministic-fail} $O(1)$ times, there will
  exist a value of~$t_k\in [0,1]$ with denominator~$D$ satisfying the required
  inequalities in steps~(\ref{step:deterministic-tk}.b)
  and~(\ref{step:deterministic-tk}.c). In particular, at the end of a
  successful step~\eqref{step:deterministic-tk}, we have
  \begin{displaymath}
    \abs{P'_{a,b,j}} \geq 2\eta \norm{P_{a,b,j}}_\infty.
  \end{displaymath}
  Moreover, at the end of step~\eqref{step:deterministic-tk}, we have
  \begin{displaymath}
    \abs[\big]{\thetatilde_{a,b}(2t,\tau) - P'_{a,b,0}}
    \leq \eta \norm{P_{a,b,0}}_\infty
  \end{displaymath}
  and for each~$1\leq j\leq h$,
  \begin{align*}
    \abs[\big]{\theta_{a,0}(2^jt,2^j\tau) - P'_{a,0,j}}
    &\leq \eta \norm{P_{a,0,j}}_\infty\\
    \abs[\big]{\theta_{a,0}(2^{j+1}t,2^j\tau) - P'_{a,1,j}}
    &\leq \eta \norm{P_{a,1,j}}_\infty\\
    \abs[\big]{\thetatilde_{a,0}(2^j(z+t),2^j\tau) - P'_{a,2,j}}
    &\leq \eta \norm{P_{a,2,j}}_\infty\\
    \abs[\big]{\thetatilde_{a,0}(2^j(z+2t),2^j\tau) - P'_{a,3,j}}
    &\leq \eta \norm{P_{a,3,j}}_\infty.
  \end{align*}
  Therefore \cref{algo:deterministic-t} is correct.
\end{prop}

\begin{proof}
  The first statement is a straightforward induction as in the proof of
  \cref{prop:t-dimg}. The existence of~$t_k$ after~$O(1)$ applications of
  step~\eqref{step:deterministic-fail} is guaranteed by \cref{prop:t-dimg} and
  the construction of~$D$ during the proof of \cref{thm:choose-t}. Further, by
  construction of~$\eta$, we have
  \begin{displaymath}
    \paren[\Big]{\frac{1}{d+1}}^g \paren[\Big]{\frac 2D}^g \geq 2\eta.
  \end{displaymath}
  By construction, the polynomials~$P_{a,b,j}$ approximate the required theta
  values up to an absolute error of~$\eta\norm{P_{a,b,j}}_\infty$, and we have
  at the end of step~\eqref{step:deterministic-tk}
  \begin{displaymath}
    P'_{a,b,j} = P_{a,b,j} \paren[\big]{\e(2^{\max(1,j)}t_1),\ldots, \e(2^{\max(1,j)}t_g)}.
  \end{displaymath}
  Therefore, the inequalities of the proposition are satisfied, and the
  quantities~$P'_{a,b,j}$ output in step~\eqref{step:deterministic-t-err}
  contain the required theta values but do not contain zero.
\end{proof}

Of course, \cref{algo:deterministic-t} must take into account the fact that we
can only work with complex numbers at some finite precision. More precisely:
\begin{itemize}
\item In step~\eqref{step:deterministic-t-closest}, it is sufficient to
  choose~$w_a$,~$w_a'$ such that
  \begin{displaymath}
    \norm{w_a}_\tau \leq \Dist_\tau(0, \Z^g + \tfrac a2) + 2^{-32} \quad\text{and}\quad
    \norm{v - w'_a}_\tau \leq \Dist_\tau(v, \Z^g + \tfrac a2) + 2^{-32}.
  \end{displaymath}
\item In step~\eqref{step:deterministic-d-delta}, we compute $d$
  and~$\delta_0$ up to an error of~$2^{-32}$, rounding above.
\item In step~\eqref{step:deterministic-polys}, we approximate the coefficients
  of the polynomials~$P_{a,b,j}$ in such a way that the total absolute error we
  commit when evaluating them until the end of
  step~\eqref{step:deterministic-tk} is at
  most~$ 2^{-32} \eta\norm{P_{a,b,j}}_\infty$. The evaluation points are roots
  of unity, and the total number of monomials is bounded above by
  \cref{prop:eld-width}. Thus we only need to compute each coefficient (a
  complex exponential) to relative precision
  $\log(1/\eta) + O(g(\log g + \log\delta_0))$. This precision
  is~$O(\log(1/\eta))$.
\item In step~(\ref{step:deterministic-tk}.b), we can only determine the left
  hand side up to a small error, namely at most
  $2^{-32} \eta\norm{P_{a,b,j}}_\infty$. Taking this error into account, we see
  that \cref{algo:deterministic-t} will certainly obtain a value of~$t_k$ such
  that for all triples $(a,b,j)$ as above,
  \begin{displaymath}
    \abs[\big]{Q_{a,b,j}(\e(2^{\max(1,j)}t_k))}
    \geq \paren[\bigg]{\frac{1}{d+1} \paren[\Big]{\frac 2D}^{d} - 2^{-31}\eta} \norm{Q_{a,b,j}}_\infty.
  \end{displaymath}
  The computations in \cref{thm:choose-t} are loose enough for the algorithm to
  still work with this slightly weaker bound.
\end{itemize}

\begin{prop}
  \label{prop:setup-cost} With the above modifications,
  \Cref{algo:deterministic-t} succeeds, outputs~$\eta$ such
  that~$\log(1/\eta) = 2^{O(\log^2 g)} h^2\log(1+h)$, and uses at most
  $2^{O(g\log^2 g)} \log(N)^{5 + 3g/2}$ binary operations.
\end{prop}

\begin{proof}
  We omit the technical proof that when running \cref{algo:deterministic-t} in
  interval arithmetic as above, the number of times through
  step~\eqref{step:deterministic-fail} is still~$O(1)$ and the output theta
  values do not contain zero as complex balls.  It is then sufficient to
  analyze the cost of the last iteration through
  steps~\eqref{step:deterministic-d-delta}--\eqref{step:deterministic-tk},
  during which
  \begin{displaymath}
    \log(1/\eta) = 2^{O(\log^2 g)} h^2 \log(1+h), \quad D = 2^{O(g)} N, \quad \text{and}\quad
    \delta_0 = 2^{O(\log^2 g)} h \log(1+h)^{1/2}.
  \end{displaymath}
  Recall that~$h = O(\log N)$, so $\delta_0 = 2^{O(\log^2 g)}
  \log(N)^{3/2}$. The multivariate polynomials we manipulate have
  relative degrees at most $2^{O(\log^2 g)}\log(N)^{3/2}$ by
  \cref{prop:eld-width}.

  We first estimate how many times steps~(\ref{step:deterministic-tk}.a)
  and~(\ref{step:deterministic-tk}.b) are performed, and in that case, how many
  vertices in the tree $T(t_k,\max\{0,h-2\})$ are processed. We keep notation
  from the proof of \cref{prop:deterministic-t}.
  Let~$0\leq l\leq \max\{0,h-2\}$. By \cref{lem:t-dim1} and the upper bound on
  the relative degrees of~$P_{a,b,j}$ obtained in the proof of
  \cref{thm:choose-t}, the number of vertices of depth~$l$ in trees of the
  form~$T(t_k, \max\{0,h-2\})$ where one of the required inequalities fail is at most
  \begin{displaymath}
    g\cdot 2^{2g} \cdot (2^{2 + \log^2 g} + 2^{2 - (h-2-l)/2} g\delta_0)
  \end{displaymath}
  throughout the whole step~\eqref{step:deterministic-tk}. Summing over~$l$ and
  simplifying slightly, the total number of vertices where an inequality fails
  is at most $2^{4g + 4 + h/2} \delta_0$.  Moreover, given the binary structure
  of the trees we manipulate, the total number of vertices visited during
  step~\eqref{step:deterministic-tk} is at most twice this amount.

  We can now conclude on the cost of \cref{algo:deterministic-t} step by step:
  \begin{itemize}
  \item Step~\eqref{step:deterministic-t-eta} runs in time~$O(\log(1/\eta))$.
  \item Step~\eqref{step:deterministic-t-closest} runs in time
    $2^{O(g \log^2 g)} \Mul(\log(N))$, since it is actually a byproduct of
    computing distances as in \cref{prop:distances}.
  \item Step~\eqref{step:deterministic-d-delta} runs in time~$O(\log(1/\eta))$.
  \item Step~\eqref{step:deterministic-polys} amounts to
    computing~$2^{O(g\log^2 g)} \log(N)^{3g/2}$ exponentials to relative
    precision~$O(\log(1/\eta))$. This is done in time
    $2^{O(g\log^2 g)} \log(N)^{3g/2}\log(1/\eta)^2$.
  \item Step~(\ref{step:deterministic-tk}.a) is free.
  \item In step~(\ref{step:deterministic-tk}.b), processing an individual value
    of~$0\leq t_k < 1/2^{\max\{0,h-2\}}$ amounts to evaluating at most~$2^{2g}$
    univariate Laurent polynomials at $\e(2^{\max\{1,h-1\}}t_k)$ to
    precision~$O(\log(1/\eta))$. This is done in time
    $2^{O(g)}\log(N)^{3/2}\log(1/\eta)^2$.  The number of values of~$t_k$ we
    consider is at most $D/2^{\max\{0,h-2\}} = 2^{O(g)}$, so the total cost of
    step~(\ref{step:deterministic-tk}.b) is
    $2^{O(g)} \log(N)^{3/2} \log(1/\eta)^2$.
  \item In step~(\ref{step:deterministic-tk}.c), the total number of visited
    vertices is at most~$2^{4g + 4 + h/2}\delta_0$ as we argued above. Visiting
    one vertex involves at most~$2^{2g}$ evaluations of univariate Laurent
    polynomials, so this step
    costs~$2^{O(g\log^2 g)} \log(N)^{3} \log(1/\eta)^2$ binary operations in
    total.
  \item Step~(\ref{step:deterministic-tk}.d) is performed only once for
    each~$k \in \{g, \ldots,1\}$, and involves at most $2^{2g} + h2^{g+2}$
    evaluations of multivariate Laurent polynomials. Considering the fact that
    the relative degrees of~$P_{a,b,j}$ decrease as~$j$ increases, the total
    cost of this step is
    $2^{O(g\log^2 g)} \log(N)^{3g/2} \log(1/\eta)^2$ binary operations.
  \item Step~\eqref{step:deterministic-t-err} costs~$O(\log(1/\eta))$.
  \end{itemize}
  Summing the contributions of each step yields the result.
\end{proof}

\subsection{Complexity analysis of \texorpdfstring{\cref{algo:ql}}{the fast algorithm}.}
\label{subsec:complexity}

Combining~\cref{prop:ql-precisions,prop:lowerdim-cost,prop:setup-cost} yields
our main result that \cref{algo:ql} has a uniform quasi-linear time complexity.

\begin{thm}[{$={}$\cref{thm:main-intro}}]
  \label{thm:ql-complexity}
  Let $(z,\tau)\in \red_g$ be an exact point, and let~$N\in \Z_{\geq 2}$. Then
  \cref{algo:ql}, where
  \begin{itemize}
  \item squared distances are computed to absolute precision $32+h$ in
    step~\eqref{step:ql-dist} using \cref{algo:dist},
  \item choosing~$t,\eta_{g,h}$ and collecting low-precision theta values in
    step~\eqref{step:ql-t} is done with the deterministic
    algorithm~\ref{algo:deterministic-t},
  \item the required theta values in steps~\eqref{step:ql-init}
    or~\eqref{step:ql-nothing} are computed to shifted absolute
    precision~$N+h\Delta$ where~$\Delta = 10(3+g\log^2 g) + \log(1/\eta_{g,h})$
    as in \cref{prop:lowerdim-cost}, except in step~\eqref{step:ql-nothing} of
    the initial call where they are computed to absolute precision~$N$,
  \item the duplication steps~\eqref{step:ql-rec}--\eqref{step:ql-base} are
    carried out using Hadamard transformations as in \cref{prop:ql-precisions},
  \end{itemize}
  is correct and runs in time $2^{O(g\log^2 g)} \Mul(N) \log(N)$ uniformly
  in~$(z,\tau)$. The algorithm outputs the theta
  values~$\thetatilde_{a,b}(z,\tau)$ to shifted absolute precision~$N$ for all
  $a,b\in \{0,1\}^g$, unless~$h=0$ in the initial call, in which case theta
  values are output to absolute precision~$N$ only.
\end{thm}

\begin{proof}
  We prove this statement by induction on~$g$. Let~$z,\tau$ and~$N$ be as
  above. By \cref{prop:setup-cost}, we obtain a value of~$\eta_{g,h}$ in
  step~\eqref{step:ql-t} such that
  \begin{displaymath}
    \log(1/\eta_{g,h}) = 2^{O(\log^2 g)} h^2 \log(1+h).
  \end{displaymath}
  Because~$h = O(\log N)$, we
  have~$\Delta = 2^{O(\log^2 g)} \log(N)^2\log\log(N)$ in
  \cref{prop:ql-precisions}. In particular, $h\Delta/N = 2^{O(\log^2 g)}$.

  If~$d>0$, then by the induction hypothesis (applied in dimension~$g-d$ at
  shifted absolute precision~$N+h\Delta$) and \cref{prop:lowerdim-cost} (or its
  analogue in the case~$h=0$), the total cost of recursive calls to
  \cref{algo:ql} in step~\eqref{step:ql-init} is
  \begin{displaymath}
    2^{O((g-d)\log^2 g)} \cdot 2^{O(d\log^2 d)} 2^{O(\log^2 g)} \Mul(N)\log N = 2^{O(g\log^2 g)} \Mul(N)\log N.
  \end{displaymath}
  In that step, the cost of reducing the first argument is
  negligible. If~$d=0$, then the cost of \cref{algo:summation} is also
  $2^{O(g\log^2 g)} \Mul(N)\log N$ by construction of~$h$.

  By \cref{prop:setup-cost,prop:lowerdim-cost,prop:ql-precisions}, the total
  cost of other steps in~\cref{algo:ql} is dominated by
  $2^{O(g\log^2 g)} \Mul(N)\log N$, so the induction succeeds.
\end{proof}

\subsection{Derivatives of theta functions from finite differences}
\label{subsec:deriv}

Evaluating derivatives of theta functions, not just theta functions themselves,
in quasi-linear time is also desirable in some applications. One approach could
be to differentiate the duplication formula~\eqref{eq:dupl-z-zp}, and to
manipulate both theta functions and their derivatives at each step of
\cref{algo:ql}. In our implementation, we follow a second approach, and compute
finite differences of theta values obtained via the quasi-linear algorithm. We
will prove the following theorem; recall the notation for partial derivatives
from~§\ref{subsec:notation}.

\begin{thm}
  \label{thm:deriv}
  Given~$g\in \Z_{\geq 1}$,~$N\in \Z_{\geq 2}$,~$B\geq 0$, and an
  exact point $(z,\tau)\in \red_g$, one can compute the
  partial derivatives $\partial^\nu\theta_{a,b}(z,\tau)$
  % \begin{displaymath}
  %   \frac{\partial^{\abs{k}}\theta_{a,b}}{\partial z_1^{k_1}\cdots \partial z_g^{k_g}}(z,\tau)
  % \end{displaymath}
  to absolute precision~$N$ for all $a,b\in \{0,1\}^g$ and for every
  vector~$\nu$ of nonnegative integers such that $\abs{\nu}\leq B$, in total
  time $2^{O(g\log^2 g)} B^g \Mul(N + B\log B + \abs{z})\log (N + B\log B + \abs{z})$.
\end{thm}

Note that the complexity bound of~\cref{thm:deriv} is both quasi-linear in~$N$
and, when~$N$ is large enough, linear in the number of derivatives to be computed. In
any case, it is quasi-linear in the output size for a
fixed~$g$.

Fix a point $(z,\tau)\in \red_g$, and consider the Taylor expansion
\begin{equation}
  \label{eq:taylor}
  \theta_{a,b}(z+h,\tau) = \sum_{\nu\in \Z^g,\ \nu\geq 0} a_\nu h_1^{\nu_1}\cdots h_g^{\nu_g}.
\end{equation}
For every nonnegative vector~$\nu$, we have
\begin{equation}
  \label{eq:partial-taylor}
  \partial^\nu\theta_{a,b}(z,\tau) = \nu_1!\cdots\nu_g!\, a_\nu.
\end{equation}
We will compute the coefficients $a_\nu$ using a discrete Fourier
transform. Let~$\zeta$ be a primitive $B+1$st root of
unity, and fix~$\eps\in \R_{>0}$. For every $n\in\{0,\ldots,B\}^g$, we
evaluate $\theta_{a,b}(z+h,\tau)$ where
\begin{equation}
  \label{eq:hn}
  h = h_n = (\eps \zeta^{n_1},\ldots,\eps \zeta^{n_g}).
\end{equation}

From these evaluations, we can recover the partial derivatives of index~$\nu$
for every $\nu\in \Z^g$ whose entries are between $0$ and~$B$. (The vectors
$\nu$ such that $\abs{\nu}\leq B$ make up a proportion roughly $1/(g-1)!$ of
this set.) Indeed, fixing such a $\nu\in \Z^g$, we have
\begin{equation}
  \label{eq:fourier}
  \begin{aligned}
    \sum_{n\in \{0,\ldots,B\}^g} \zeta^{-\nu^Tn} \theta_{a,b}(z + h_n,\tau)
    &= (B+1)^g \sum_{\substack{p\in \Z^g,\ p\geq 0 \\ p = \nu \ (\mathrm{mod}\ B+1)}} a_p\,\eps^{\abs{p}}
    \\
    &= (B+1)^g \paren[\big]{a_\nu \eps^{\abs{\nu}} + T}
  \end{aligned}
\end{equation}
where the tail~$T$ of the latter sum should be of the order
of~$\eps^{\abs{\nu}+B+1}$. Thus, we can indeed recover~$a_\nu$ provided
that~$\eps$ is small enough. In order to obtain a certified error bound
on~$a_\nu$, we need an explicit upper bound on~$T$.

\begin{lem}
  \label{lem:ball-radius}
  Let~$\gamma\in \R_{\geq 0}$ and~$\rho\in \R_{>0}$ be chosen in such a way that
  $\abs{\theta_{a,b}(x,\tau)}\leq \gamma$ for every~$x\in \C^g$ such
  that~$\norm{x-z}_\infty\leq \rho$. Suppose further that
  $\eps\leq (2g)^{-1/(B+1)}\rho$. Then with the above notation, we have
  \begin{displaymath}
    % \abs[\Bigg]{\sum_{\substack{p\in \Z^g,\ p\geq 0\\ p = \nu\ (\mathrm{mod}\ B+1),\ p\neq \nu}} a_p\,\eps^{\abs{p}}}
    \abs{T}
    \leq 2\gamma g \paren[\Big]{\frac{\eps}{\rho}}^{\abs{\nu}+B+1}.
  \end{displaymath}
\end{lem}

\begin{proof}
  By the Cauchy integration formula, we have
  $\abs{a_p} \leq \gamma \rho^{-\abs{p}}$ for each~$p$. Moreover, for
  each~$j\in \Z_{\geq 1}$, at most~$g^j$ indices~$p\in \Z^g$
  satisfy~$\abs{p} = \abs{\nu} + j(B+1)$.
  %\begin{displaymath}
  %  \binom{g-1+j}{g-1} = \binom{g-1+j}{j} \leq g^j.
  %\end{displaymath}
  Therefore,
  \begin{displaymath}
    \abs{T} \leq \gamma\,\paren[\Big]{\frac{\eps}{\rho}}^{\abs{k}}
    \sum_{j\geq 1} g^j \paren[\Big]{\frac{\eps}{\rho}}^{j(B+1)} \leq 2\gamma g \paren[\Big]{\frac{\eps}{\rho}}^{\abs{\nu} + B+1},
  \end{displaymath}
  as we are summing a geometric series of ratio at most $\tfrac 12$.
\end{proof}

Combining \cref{lem:ball-radius} with equalities~\eqref{eq:fourier}
and~\eqref{eq:partial-taylor}, the final error bound on the partial derivative
$\partial^\nu\theta_{a,b}(z,\tau)$ (not taking precision losses into account)
will be
\begin{equation}
  \label{eq:der-error}
  \eta_\nu = \nu_1!\cdots \nu_g!\, (B+1)^{-g} \cdot 2\gamma g \frac{\eps^{B+1}}{\rho^{\abs{\nu} +
      B + 1}}.
\end{equation}
Next, we explain how suitable values of~$\gamma$ and~$\rho$ can be computed. We
continue using notation from~§\ref{subsec:notation}.

\begin{lem}
  \label{lem:uniform-ubound} Define
  \begin{align*}
    \gamma_0 &= \paren[\Big]{1 + \sqrt{\frac 8\pi}\,} 2^{g-1} \prod_{j=1}^{g}
      \paren[\Big]{1 + \frac{\sqrt{2\pi}}{c_j}},\\
    \gamma_1 &= \sqrt{\pi y^T Y^{-1} y}, \quad\text{and}\\
    \gamma_2 &= \sup_{x\in \R^g,\ \norm{x}_\infty \leq 1} \sqrt{\pi x^T Y^{-1} x}.
  \end{align*}
  Then for every~$\rho\in \R_{>0}$, the pair $(\rho,\gamma)$
  where~$\gamma = \gamma_0\exp\paren[\big]{(\gamma_1 + \gamma_2\rho)^2)}$ satisfies the conditions
  of \cref{lem:ball-radius}.
\end{lem}

\begin{proof}
  Let~$\rho > 0$, and let~$x\in \C^g$ be such
  that~$\norm{x - z}_\infty\leq \rho$. By \cref{thm:summation-bound-new}
  with $A=R=p=0$, we have for every characteristic~$(a,b)$
  \begin{displaymath}
    \abs{\theta_{a,b}(x,\tau)} \leq \gamma_0 \exp \paren[\big]{\pi \im(x) Y^{-1}\im(x)}.
  \end{displaymath}
  Write $\im(x) = y + \rho h$ where~$h\in \R^g$
  satisfies~$\norm{h}_\infty\leq 1$. Applying the triangle inequality to the
  Euclidean norm with Gram matrix~$\pi Y^{-1}$, we get
  \begin{displaymath}
    \pi \im(x) Y^{-1}\im(x) \leq \paren{\gamma_1 + \gamma_2\rho}^2. \qedhere
  \end{displaymath}
\end{proof}

After these preparations, we can describe the algorithm \cref{thm:deriv} refers
to.

\begin{algorithm}
  \label{algo:deriv}
  \algoinput{An exact point $(z,\tau)\in \red_g$, a
    precision~$N\in \Z_{\geq 2}$, and $B\in \Z_{\geq 0}$.}  \algooutput{The
    partial derivatives $\partial^\nu\theta_{a,b}(z,\tau)$ to absolute
    precision~$N$ for all $a,b\in \{0,1\}^g$ and every $\nu\in \Z_{\geq 0}^g$
    such that $\abs{\nu}\leq B$.}
  \begin{enumerate}
  \item Compute upper bounds on~$\gamma_0,\gamma_1,\gamma_2$ as defined in
    \cref{lem:uniform-ubound}. An upper bound on~$\gamma_2$ can be obtained
    from the Cholesky decomposition of~$\pi Y^{-1}$.
  \item Look for values of~$(\gamma,\rho)$ that minimize the final error bound
    in~\eqref{eq:der-error}, with the constraints~$\gamma\geq 1$
    and~$\rho\leq 1$. We choose~$\rho$ such that
    $\exp\paren[\big]{(\gamma_1 + \gamma_2\rho)^2}/\rho^{2B+1}$ is minimal,
    i.e.~$\rho$ is the minimum of~$1$ and the positive root of the quadratic
    equation $2\gamma_2\rho(\gamma_1 + \gamma_2\rho) = 2B+1$.
  \item \label{step:eps} Choose~$\eps\in \R_{>0}$ so that
    \begin{displaymath}
      \eps\leq (2g)^{-1/(B+1})\rho\quad\text{and}\quad 2\gamma g B!\, (B+1)^{-g} \eps^{B+1}\leq 2^{-N-1} \rho^{2B+1}.
    \end{displaymath}
  \item \label{step:fourier-eval} Evaluate~$\thetatilde_{a,b}(z + h_n,\tau)$
    for all $a,b\in \{0,1\}^g$ for every $n\in \{0,\ldots,B\}^g$, where~$h_n$
    is defined as in~\eqref{eq:hn} to some absolute precision $N'\geq 2N$ (to
    be specified below), using \cref{thm:main-intro}. Deduce the theta
    values~$\theta_{a,b}(z + h_n,\tau)$.
  \item \label{step:fourier-sum} For each vector~$\nu$ as above, evaluate the
    left hand side $S_\nu$ of equation~\eqref{eq:fourier}.
    Let~$d_\nu = \nu_1!\cdots \nu_g!\, (B+1)^{-g} \eps^{-\abs{\nu}} S_\nu$, then add~$\eta_\nu$
    as defined in~\eqref{eq:der-error} to the error radius of~$d_\nu$.
  \item Output the collection $(d_\nu)_\nu$.
  \end{enumerate}
\end{algorithm}

We claim that \cref{algo:deriv} achieves the specifications of
\cref{thm:deriv}. Along the way, we compute the absolute precision~$N'$ to be
used in step~\eqref{step:fourier-eval} in order to output values at absolute
precision~$N$: we can always take
\begin{equation}
\label{eq:Nprime}
  N' = 2N + O\paren[\big]{B\log B + \abs{z} + 2^{O(\log^2 g)}}.
\end{equation}

\begin{proof}
  Since~$(z,\tau)\in \red_g$, by \cref{prop:hkz-bound}, we can take~$\gamma_0$
  to be in~$2^{O(g\log^2 g)}$ and~$\gamma_2$ to be in $2^{O(\log^2
    g)}$. Moreover, $\norm{Y^{-1}y}_\infty\leq \tfrac 12$, so
  $\gamma_1 = O(\sqrt{\abs{z}})$. As a consequence, $\rho$ and~$\gamma$ satisfy
  \begin{displaymath}
    \log(\rho^{-1}) = O\paren[\big]{\log^2 g + \log\max\{1,\abs{z}\}} \quad \text{and}\quad  \log(\gamma) = O\paren[\big]{\abs{z} +  2^{O(\log^2 g)}}.
  \end{displaymath}
  In step~\eqref{step:eps}, we can thus take~$\eps$ such that
  \begin{displaymath}
    (B+1)\log(1/\eps) = N + O\paren[\big]{B\log B + \abs{z} + 2^{O(\log^2 g)}}.
  \end{displaymath}
  Given step~\eqref{step:fourier-sum}, the absolute precision~$N'$ to consider
  in step~\eqref{step:fourier-eval} indeed satisfies~\eqref{eq:Nprime}. This
  estimate takes precision losses in arithmetic operations into account, as
  well as multiplications by exponential factors to
  compute~$\theta_{a,b}(z+h_n,\tau)$ from~$\thetatilde_{a,b}(z+h_n,\tau)$.

  By \cref{thm:main-intro}, step~\eqref{step:fourier-eval} uses at most
  $(B+1)^g \cdot 2^{O(g\log^2 g)} \Mul(N')\log(N')$ binary operations, and
  dominates the rest of the algorithm.
\end{proof}

In our implementation, the sums on the left hand side of~\eqref{eq:fourier} are
computed efficiently thanks to existing functionality to perform discrete
Fourier transforms in FLINT. The required precision~$N'$ is computed on the fly
based on the computed values of~$\rho$ and~$\gamma$.

\begin{rem}
  \Cref{algo:deriv} uses in an essential way that~$\theta_{a,b}$ is
  holomorphic, and therefore admits a Taylor expansion as
  in~\eqref{eq:taylor}. We see no obvious way of evaluating the partial
  derivatives of a given real-analytic function in $g$ variables, say, up to
  some order~$B$ in quasi-linear time in~$B^g$ using finite differences.
\end{rem}

\section{Performance comparisons}
\label{sec:cmp}

In this section, we compare our FLINT implementation of the summation
algorithm~\ref{algo:summation} first with existing software, then with our
implementation of the fast algorithm~\ref{algo:ql}. The code used to run these
experiments is available at
\url{https://github.com/j-kieffer/theta-implementations}, as well as
instructions on how to reproduce them.

\subsection{Comparison with existing implementations of summation algorithms}
\label{subsec:sum-compare}

\begin{figure}[p]
  \centering
  {\footnotesize
    \begin{tabular}{|c|c|cccc|}
      \hline
      $g$ & Prec. & \texttt{Theta.jl} & Magma's \texttt{Theta} & \texttt{RiemannTheta} & \texttt{acb\_theta\_sum} \\
      \hline & &&&&\\[-1em]
          & $2^6$ & $2.20 \cdot 10^{-3}$ & $7.18\cdot 10^{-4}$ & $4.85\cdot 10^{-3}$ & $3.61 \cdot 10^{-5}$ \\
          & $2^7$ & & $1.19\cdot 10^{-3}$ & $6.46 \cdot 10^{-3}$ & $6.64\cdot 10^{-5}$ \\
          & $2^8$ & & $2.19\cdot 10^{-3}$ & $1.05\cdot 10^{-2}$ & $1.28\cdot 10^{-4}$ \\
          & $2^9$ & & $4.88\cdot 10^{-3}$ & $1.94 \cdot 10^{-2}$ & $4.15\cdot 10^{-4}$ \\
          & $2^{10}$ & & $1.50\cdot 10^{-2}$ & $4.65\cdot 10^{-2}$ & $1.86\cdot 10^{-3}$ \\
      2 & $2^{11}$ & & $5.93\cdot 10^{-2}$ & $1.62\cdot 10^{-1}$ & $8.17\cdot 10^{-3}$\\
          & $2^{12}$ & & $3.10\cdot 10^{-1}$ & $8.77\cdot 10^{-1}$ & $3.86\cdot 10^{-2}$ \\
          & $2^{13}$ & & $2.00$ & $5.80$ & $2.06\cdot 10^{-1}$ \\
          & $2^{14}$ & & $15.44$ & $45.5$ & $1.07$ \\
          & $2^{15}$ & & - & - & $5.53$ \\
          & $2^{16}$ & & - & - & $26.1$ \\ \hline & &&&&\\[-1em]

          & $2^{6}$ & $5.42 \cdot 10^{-3}$ & $4.69\cdot 10^{-3}$ & $6.72\cdot 10^{-3}$ & $1.99\cdot 10^{-4}$ \\
          & $2^{7}$ & & $9.88\cdot 10^{-3}$ & $1.19\cdot 10^{-2}$ & $6.10\cdot 10^{-4}$ \\
          & $2^{8}$ & & $3.66\cdot 10^{-2}$ & $2.87\cdot 10^{-2}$ & $1.68\cdot 10^{-3}$ \\
          & $2^{9}$ & & $1.11\cdot 10^{-1}$ & $9.82\cdot 10^{-2}$ & $6.59\cdot 10^{-3}$ \\
      3 & $2^{10}$ & & $5.00\cdot 10^{-1}$ & $4.65\cdot 10^{-1}$ & $3.75\cdot 10^{-2}$ \\
          & $2^{11}$ & & $3.03$ & $3.11$ & $2.39\cdot 10^{-1}$ \\
          & $2^{12}$ & & $16.43$ & $29.7$ & $1.41$ \\
          & $2^{13}$ & & - & - & $9.77$ \\
          & $2^{14}$ & & - & - & $70.6$ \\ \hline & &&&&\\[-1em]

          & $2^{6}$ & $6.16 \cdot 10^{-2}$ & $2.28\cdot 10^{-2}$ & $1.35\cdot 10^{-2}$ & $1.38\cdot 10^{-3}$ \\
          & $2^{7}$ & & $6.56\cdot 10^{-2}$ & $4.67\cdot 10^{-2}$ & $4.57\cdot 10^{-3}$ \\
          & $2^{8}$ & & $3.19 \cdot 10^{-1}$ & $1.86\cdot 10^{-1}$ & $1.61\cdot 10^{-2}$ \\
      4 & $2^{9}$ & & $1.46$ & $1.04$ & $8.90\cdot 10^{-2}$ \\
          & $2^{10}$ & & $6.98$ & $7.61$ & $7.48\cdot 10^{-1}$ \\
          & $2^{11}$ & & $62.15$ & $74.0$ & $5.71$ \\
          & $2^{12}$ & & - & - & $46.1$ \\ \hline & &&&&\\[-1em]

          & $2^{6}$ & $5.23\cdot 10^{-1}$ & $1.79\cdot 10^{-1}$ & $4.76\cdot 10^{-2}$ & $7.22\cdot 10^{-3}$ \\
          & $2^{7}$ & & $6.87\cdot 10^{-1}$ & $2.56\cdot 10^{-1}$ & $3.08 \cdot 10^{-2}$ \\
      5 & $2^{8}$ & & $3.24$ & $1.60$& $1.63\cdot 10^{-1}$ \\
          & $2^{9}$ & & $12.7$ & $12.7$ & $1.12$ \\
          & $2^{10}$ & & - & - & $11.5$ \\ \hline & &&&&\\[-1em]

          & $2^{6}$ &  $7.76$ & $8.13\cdot 10^{-1}$ & $2.31\cdot 10^{-1}$ & $3.82\cdot 10^{-2}$ \\
      6  & $2^{7}$ & & $4.51$ & $1.68$ & $2.27 \cdot 10^{-1}$ \\
          & $2^{8}$ & & $26.7$ & $14.5$ & $1.38$ \\
          & $2^{9}$ & & - & - & $12.9$ \\ \hline & &&&&\\[-1em]

          & $2^{6}$ &  $122$ & $3.14$ & $1.38$ &  $2.26\cdot 10^{-1}$ \\
      7  & $2^{7}$ & & $21.9$ & $11.7$ & $1.40$ \\
          & $2^{8}$ & & - & - & $11.7$ \\ \hline & &&&&\\[-1em]

          & $2^{6}$ & - & $11.2$ & $6.83$ &  $1.08$ \\
      8  & $2^{7}$ & & - & $72.6$ & $8.48$ \\
          & $2^{8}$ & & - & - & $94.7$ \\ \hline
    \end{tabular} }
  \vspace{-4pt}
  \caption{Time in seconds to evaluate $\theta_{0,0}(0,\tau)$ by summation.}
  \label{fig:compare-sum}
\end{figure}

We first compare our FLINT implementation of \cref{algo:summation} (as the
function \texttt{acb\_theta\_sum}) with other recent implementations of
summation algorithms:
\begin{itemize}
\item the Julia package
  \texttt{Theta.jl}~\cite{agostiniComputingThetaFunctions2021}, version 0.1.2;
\item Magma's intrinsic \texttt{Theta}~\cite{bosmaMagmaAlgebraSystem1997}, version 2.27-7;
\item and the SageMath package
  \texttt{RiemannTheta}~\cite{bruinRiemannThetaSageMathPackage2021} version
  1.0.0, using Sage version 10.6.
\end{itemize}

We do not include the SageMath package
\texttt{abelfunctions}~\cite{swierczewskiComputingRiemannTheta2016} which did
not appear to be immediately compatible with our setup. We refer
to~\cite{agostiniComputingThetaFunctions2021} instead for experimental
comparisons between \texttt{abelfunctions} and~\texttt{Theta.jl} suggesting
that the latter is usually faster. Finally, Christian Klein kindly shared with
us the Matlab code associated
with~\cite{frauendienerEfficientComputationMultidimensional2019}; quick
experiments showed this software to be slower than \texttt{Theta.jl} as well.

We consider $2\leq g\leq 8$, as our FLINT implementation relies
on~\cite{engeShortAdditionSequences2018} (more precisely on the function
\texttt{acb\_modular\_theta\_sum}) in the case $g=1$. For each~$g$, we sample a
representative reduced matrix~$\tau\in \Half_g$ as follows:
\begin{itemize}
\item the entries of $\re(\tau)$ are sampled uniformly in $[-1/2,1/2]$, and
\item $\im(\tau)$ is the Gram matrix of an HKZ-reduced lattice whose successive
  minima are equal to or greater than~1, whose construction is inspired by
  \cite[Rem.~3.1]{lagariasKorkinZolotarevBasesSuccessive1990}: we let
  \begin{displaymath}
    B =
    \begin{pmatrix}
      I_{g-1} & v\\ 0 & \sqrt{3}/2
    \end{pmatrix}
  \end{displaymath}
  where~$v$ consists of $g-1$ random values in $[-1/2,1/2]$, and take~$\im(\tau) = B^T B$.
\end{itemize}

\Cref{fig:compare-sum} presents the time taken to evaluate
$\theta_{0,0}(0,\tau)$ using these software packages on a single core of a 2023
laptop, in each dimension~$g$, at binary precisions ranging from $64$ to
$2^{16} = 65536$ bits. (Note that \texttt{Theta.jl} only supports low
precision.) Computations were not run if a previous, easier computation took
more than 10 seconds to complete. We took care to average the faster timings
over several iterations, and in the case of \texttt{Theta.jl}, to ignore the
first run during which compilation occurs.

While the precise timings always vary slightly when the experiment is
reproduced, we observe that our implementation of the summation algorithms is
consistently faster than the other software packages, often by a factor of at
least 10.

\subsection{Comparison between summation and the fast algorithm}
\label{subsec:ql-compare}

\begin{figure}
  \centering
  {\footnotesize
    \begin{tabular}{|c|c|cc||c|cc|}
      \hline
      $g$ & Prec. & \texttt{acb\_theta\_sum}
      & \parbox{3.05cm}{\begin{tabular}{c} \texttt{acb\_theta\_} \\ \texttt{jet\_notransform} \end{tabular}}
      & $g$ & \texttt{acb\_theta\_sum}
      & \parbox{3.05cm}{\begin{tabular}{c} \texttt{acb\_theta\_} \\ \texttt{jet\_notransform} \end{tabular}}
    \\ \hline & & & & & & \\[-1em]
        & $2^6$ & $4.26\cdot 10^{-6}$ & $7.17\cdot 10^{-6}$ & & $3.61\cdot 10^{-5}$ & $7.79\cdot 10^{-5}$ \\
        & $2^7$ & $5.97\cdot 10^{-6}$ & $8.31\cdot 10^{-6}$ & & $6.64\cdot 10^{-5}$ & $2.60\cdot 10^{-4}$ \\
        & $2^8$ & $9.83\cdot 10^{-6}$ & $1.10\cdot 10^{-5}$ & & $1.28\cdot 10^{-4}$ & $3.20\cdot 10^{-4}$ \\
        & $2^9$ & $1.46\cdot 10^{-5}$ & $1.52\cdot 10^{-5}$ & & $4.15\cdot 10^{-4}$ & $4.23\cdot 10^{-4}$ \\
        & $2^{10}$ & $3.20\cdot 10^{-5}$ & $2.77\cdot 10^{-5}$ & & $1.86\cdot 10^{-3}$ & $5.66\cdot 10^{-4}$ \\
        & $2^{11}$ & $8.89\cdot 10^{-5}$ & $6.76\cdot 10^{-5}$ & & $8.17\cdot 10^{-3}$ & $9.08\cdot 10^{-4}$\\
        & $2^{12}$ & $3.26\cdot 10^{-4}$ & $2.34\cdot 10^{-4}$ & & $3.86\cdot 10^{-2}$ & $1.76\cdot 10^{-3}$\\
        & $2^{13}$ & $1.08\cdot 10^{-3}$ & $6.97\cdot 10^{-4}$ & & $2.12\cdot 10^{-1}$ & $3.89\cdot 10^{-3}$\\
        & $2^{14}$ & $3.37\cdot 10^{-3}$ & $2.17\cdot 10^{-3}$ & & $1.07$ & $1.02\cdot 10^{-2}$ \\
       1 & $2^{15}$ & $1.16\cdot 10^{-2}$ & $6.72\cdot 10^{-3}$& 2 & $5.50$ & $2.95\cdot 10^{-2}$ \\
        & $2^{16}$ & $3.96\cdot 10^{-2}$ & $2.15\cdot 10^{-2}$ & & $26.1$ & $1.12\cdot 10^{-1}$\\
        & $2^{17}$ & $1.79\cdot 10^{-1}$ & $1.18\cdot 10^{-1}$ & & - & $2.86\cdot 10^{-1}$\\
        & $2^{18}$ & $3.85\cdot 10^{-1}$ & $2.50\cdot 10^{-1}$ & & - & $6.02\cdot 10^{-1}$\\
        & $2^{19}$ & $9.69\cdot 10^{-1}$ & $5.63\cdot 10^{-1}$ & & - & $1.36$ \\
        & $2^{20}$ & $2.49$ & $1.28$ & & - & $2.93$ \\
        & $2^{21}$ & $6.12$ & $2.88$ & & - & $6.51$ \\
        & $2^{22}$ & $14.1$ & $4.34$ & & - & $13.2$ \\
        & $2^{23}$ & - & $9.74$ & & - & - \\
        & $2^{24}$ & - & $22.0$ & & - & - \\  \hline & & & & & & \\[-1em]

     & $2^6$ & $1.99\cdot 10^{-4}$ & $9.77\cdot 10^{-4}$ & & $1.38\cdot 10^{-3}$ & $2.47\cdot 10^{-3}$\\
        & $2^7$ & $6.10\cdot 10^{-4}$ & $1.31\cdot 10^{-3}$ & & $4.57\cdot 10^{-3}$  & $2.99\cdot 10^{-3}$ \\
        & $2^8$ & $1.68\cdot 10^{-3}$ & $1.45\cdot 10^{-3}$ & & $1.61\cdot 10^{-2}$ & $3.39\cdot 10^{-3}$ \\
        & $2^9$ & $6.59\cdot 10^{-3}$ &  $1.67\cdot 10^{-3}$ & &$8.90\cdot 10^{-2}$ & $3.95\cdot 10^{-3}$ \\
        & $2^{10}$ & $3.75\cdot 10^{-2}$ &  $2.07\cdot 10^{-3}$ & & $7.48\cdot 10^{-1}$ & $4.96\cdot 10^{-3}$\\
        & $2^{11}$ & $2.39\cdot 10^{-1}$ &  $3.20\cdot 10^{-3}$ & & $5.70$ & $7.16\cdot 10^{-3}$\\
        & $2^{12}$ & $1.41$ & $5.11\cdot 10^{-3}$ & & $46.1$ & $1.20\cdot 10^{-2}$ \\
    3   & $2^{13}$ & $9.77$ & $1.05\cdot 10^{-2}$ & 4 & - & $2.53\cdot 10^{-2}$ \\
        & $2^{14}$ & $70.6$ & $2.58\cdot 10^{-2}$ & & - & $6.34\cdot 10^{-2}$\\
     & $2^{15}$ & - & $7.02\cdot 10^{-2}$ & & - & $1.92\cdot 10^{-1}$ \\
        & $2^{16}$ & - & $2.19\cdot 10^{-1}$ & &- & $4.98\cdot 10^{-1}$\\
        & $2^{17}$ & -& $5.40\cdot 10^{-1}$ & & - & $1.12$ \\
        & $2^{18}$ & -& $1.15$& & - & $2.39$ \\
        & $2^{19}$ & -& $2.66$ & & - & $5.34$ \\
        & $2^{20}$ & -& $5.55$ & & - & $11.2$\\
        & $2^{21}$ & -& $12.3$ & & - & -\\ \hline
  \end{tabular}}
  \caption{Time in seconds to evaluate $\theta_{0,0}(0,\tau)$ with FLINT, $1\leq g\leq 4$.}
  \label{fig:compare-ql}
\end{figure}

\begin{figure}
  \centering
  {\footnotesize
    \begin{tabular}{|c|c|cc||c|cc|}
      \hline
    $g$ & Prec. & \texttt{acb\_theta\_sum} & \parbox{3.05cm}{\begin{tabular}{c} \texttt{acb\_theta\_} \\ \texttt{jet\_notransform} \end{tabular}}
    & $g$ & \texttt{acb\_theta\_sum} & \parbox{3.05cm}{\begin{tabular}{c} \texttt{acb\_theta\_} \\ \texttt{jet\_notransform} \end{tabular}}
    \\ \hline & & & & & & \\[-1em]
        & $2^6$ & $7.22\cdot 10^{-3}$ & $8.78\cdot 10^{-3}$ & & $3.82\cdot 10^{-2}$ & $2.27\cdot 10^{-2}$ \\
        & $2^7$ & $3.08\cdot 10^{-2}$ & $1.03\cdot 10^{-2}$ & & $2.27\cdot 10^{-1}$ & $2.59\cdot 10^{-2}$ \\
        & $2^8$ & $1.63\cdot 10^{-1}$ & $1.12\cdot 10^{-2}$ & & $1.38$ & $2.86\cdot 10^{-2}$ \\
        & $2^9$ & $1.12$ & $1.28\cdot 10^{-2}$ & & $12.9$ & $3.24\cdot 10^{-2}$ \\
        & $2^{10}$ & $11.5$ & $1.49\cdot 10^{-2}$ & & -& $3.91\cdot 10^{-2}$ \\
        & $2^{11}$ & -& $2.07\cdot 10^{-2}$ & & -& $5.54\cdot 10^{-2}$\\
        5 & $2^{12}$ & -& $3.21\cdot 10^{-2}$ & 6 & -& $8.79\cdot 10^{-2}$\\
        & $2^{13}$ & -& $6.37\cdot 10^{-2}$ & & -& $1.90\cdot 10^{-1}$\\
        & $2^{14}$ & -& $1.63\cdot 10^{-1}$& & -& $4.28\cdot 10^{-1}$ \\
        & $2^{15}$ & -& $4.17\cdot 10^{-1}$& &- & $1.08$ \\
        & $2^{16}$ & -& $1.08$ & & -& $2.78$\\
        & $2^{17}$ & -& $2.28$ & &- & $5.33$ \\
        & $2^{18}$ & -& $4.82$ & & -& $11.2$ \\
        & $2^{19}$ & -& $10.6$ & & -& -\\ \hline & & & & & & \\[-1em]

        & $2^6$ & $2.26\cdot 10^{-1}$ & $8.82\cdot 10^{-2}$ & & $1.07$ & $3.05\cdot 10^{-1}$ \\
        & $2^7$ & $1.40$ & $1.01\cdot 10^{-1}$ & & $8.48$ & $3.07\cdot 10^{-1}$ \\
        & $2^8$ & $11.7$ & $1.05\cdot 10^{-1}$ & & $94.7$ & $3.31\cdot 10^{-1}$ \\
        & $2^9$ & -& $1.16\cdot 10^{-1}$ & & -& $3.55\cdot 10^{-1}$ \\
        & $2^{10}$ & -& $1.50\cdot 10^{-1}$ & & -& $4.13\cdot 10^{-1}$ \\
        7 & $2^{11}$ & -& $1.81\cdot 10^{-1}$ & 8 & -& $5.32\cdot 10^{-1}$ \\
        & $2^{12}$ & -& $2.60\cdot 10^{-1}$ & & -& $7.91\cdot 10^{-1}$ \\
        & $2^{13}$ & -& $4.72$ & & -& $1.44$ \\
        & $2^{14}$ & -& $1.07$ & & -& $3.25$ \\
        & $2^{15}$ & -& $2.64$ & & -& $7.97$ \\
        & $2^{16}$ & -& $6.62$ & & -& $19.8$ \\
        & $2^{17}$ & -& $12.2$ & & -& - \\ \hline & & & & & & \\[-1em]

        & $2^6$ & $5.54$ & $3.11$ && $29.1$ & $5.64$ \\
        & $2^7$ & $53.4$ & $3.15$ && - & $5.62$ \\
        & $2^8$ & - & $3.15$ & & - & $5.73$ \\
        & $2^9$ & - & $3.24$ && - & $6.00$ \\
       9 & $2^{10}$ & - & $3.45$ & 10 & - & $6.47$\\
        & $2^{11}$ & - & $3.66$ && - & $7.37$ \\
        & $2^{12}$ & - & $4.27$ && - & $9.27$ \\
        & $2^{13}$ & - & $5.79$ && - & $14.3$ \\
        & $2^{14}$ & - & $10.1$ && - & - \\ \hline
  \end{tabular}}
  \caption{Time in seconds to evaluate $\theta_{0,0}(0,\tau)$ with FLINT, $5\leq g\leq 10$.}
  \label{fig:compare-ql-2}
\end{figure}

Next, \cref{fig:compare-ql,fig:compare-ql-2} compare the performance of our
implementation of \cref{algo:ql} (as the function
\texttt{acb\_theta\_jet\_notransform} whose parameters are chosen to only
compute $\theta_{0,0}(0,\tau)$ and no higher-order derivatives) with
\texttt{acb\_theta\_sum}, using the same conventions as above except that we
consider $1\leq g\leq 10$.

The results suggest that \cref{algo:ql} should become the new standard to
evaluate theta functions and their derivatives numerically, even at fairly low
precisions. Even in the very competitive genus~1 case, the fast algorithm
becomes faster than summation early on, around 1000--2000 bits of
precision. This reflects the fact that our implementation of
\texttt{acb\_theta\_jet\_notransform} is carefully written to avoid overhead at
low precisions.

As~$g$ grows, we also observe that the cost of the fast algorithm does not
increase linearly at low precisions, and starts from a high base value
instead. This is due to the cost of providing low-precision approximations of
theta values in step~\eqref{step:ql-t} of \cref{algo:ql}, especially for the
very last duplication step, from~$2\tau$ to~$\tau$ (or rather~$8\tau$
to~$4\tau$ given \cref{rem:ql-improvements}). Clearly, a major avenue for
improvement in higher genus would be to decrease the cost of the summation
algorithm for this particular task, perhaps by summing a few terms of large
magnitude, then using half of the absolute value of the result (say) as an
error bound for the tail of the series. Any improvement on
\cref{thm:summation-bound-new} would be of tremendous help here to improve
performances while maintaining provable correctness. Disregarding the
correctness requirement (for the low-precision approximations only!) would be
another way to gain performance at low precisions in high dimensions. In all
likelihood, the sign choices, hence the result of \cref{algo:ql} would remain
mathematically correct.

We stress that in our experiments, the input matrices~$\tau$ are exact. If
$\tau$ is specified with nonzero error radii, then our implementation will
attempt to propagate this error bound on the output in a tight way, using
low-precision evaluations of derivatives of theta functions at~$\tau$. This
computation becomes expensive (several seconds) when~$g\geq 8$, and this
affects the crossover point with \texttt{acb\_theta\_sum}, which only uses the
built-in error propagation on arithmetic operations and exponentials, for
every~$g$. Using cheaper, but less tight, error propagation could be another
fruitful axis for further development.

Finally, in the specific case of genus~$2$ theta constants, we can report that
our implementation of \cref{algo:ql} is at least 10 times faster than the
quasi-linear algorithm based on Newton's method and the
AGM~\cite{engeCMHComputationGenus2014} at high precisions. A quick explanation
for this fact is in Newton's method, a $3\times 3$ Jacobian matrix has to be
computed, hence $9$ complex AGMs have to be performed. In contrast,
\cref{algo:ql} uses only one AGM-like sequence.

\section{Application to the explicit inverse Galois problem}
\label{sec:galois}

In this final section, we construct Galois extensions of~$\Q$ of degree 65 with
conjectural Galois group $\SL_2(\F_{64})$ as acting on
$\mathbb{P}^1(\F_{64})$. Their defining polynomials arise by numerically
evaluating theta constants on period matrices of certain special principally
polarized abelian varieties (p.p.a.v.'s) of dimension~$g=6$ to moderately high
precision (a few thousand bits).

The idea of using abelian varieties (or, relatedly, modular forms) to solve the
(explicit) inverse Galois problem has a long history: we refer
to~\cite{vanbommel17T7GaloisGroup2024} for a related, but more delicate,
realization of the group \texttt{17T7} as a Galois group, and many
references. Our innovation here is really that numerical computations in
dimension~6 would have been challenging using summation algorithms or AGM-based
methods, but are now routine with the help of the quasi-linear
algorithm~\ref{algo:ql}. In fact, the numerical computations
in~\cite{vanbommel17T7GaloisGroup2024} relied on our algorithm as well
for~$g=4$.

We detail the whole computation in one example where the chosen p.p.a.v.~is the
Jacobian of an Atkin-Lehner quotient of a modular curve
(§\ref{subsec:galois-setup}--§\ref{subsec:galois-check}). We then briefly
report on two other examples where the abelian variety arises from classical
modular eigenforms with specific properties (§\ref{subsec:galois-mf}).

\subsection{The modular curve setup}
\label{subsec:galois-setup}

Consider the quotient $\Crv = X_0(271)/w_{271}$ of the modular curve $X_0(271)$
by its Atkin--Lehner involution (note that $271$ is a prime). The curve~$\Crv$
has genus~6, and admits the following model over~$\Q$:
\begin{equation}
  \scalebox{0.9}{$
  \begin{aligned}
    x^6 &+ (-6y+1)x^5 + (-y^3+17y^2-3y-6)x^4 \\
    &+ (4y^4-29y^3-8y^2+49y-13)x^3 + (-7y^5+29y^4+30y^3-78y^2-6y+12)x^2 \\
    &+ (4y^6-2y^5-65y^4+93y^3-8y^2+5y-5)x \\ &+ (-8y^6+33y^5-29y^4-9y^3+7y^2-3y+1) = 0.
  \end{aligned}$}
\end{equation}

The Jacobian~$J$ of~$\Crv$ is a p.p.a.v. of dimension~6, with real
multiplication (RM) by the maximal order~$\O$ in the totally real degree~6
number field with defining polynomial
$x^6 - x^5 - 5x^4 + 4x^3 + 5x^2 - 2x - 1$. Information on this field is
available on the LMFDB at
\url{https://www.lmfdb.org/NumberField/6.6.592661.1}. In particular,~$\O$ is
monogenic, and generated by a unit. Choose a generator~$e\in \O^\times$ and an
embedding~$\iota:\O\hookrightarrow\End_\Q(J)$. We are interested in
$(2,2,2,2,2,2)$-isogenies from~$J$ that respect its RM structure. In other
words, we are interested in maximal isotropic subgroups~$K\subset J[2](\Qbar)$
for the Weil pairing that are stable under~$\iota(\O)$.

Since~$2$ is inert in~$\O\otimes\Q$ and~$\O$ has odd discriminant, we have
$\O/2\O\simeq \F_{64}$. Moreover, $J[2](\Qbar)$ is a free $\O/2\O$-module of
rank~$2$, and this $\O$-module structure is compatible with the Weil pairing
\cite[Lemma 3.1]{banaszakImage$ell$adicGalois2006}: there exists a basis
of~$J[2](\Qbar)$ as an $\O/2\O$-module in which the Weil pairing is
$\mathrm{Tr}_{\F_{64}/\F_{2}}\, h(\cdot,\cdot)$, where
\begin{equation}
  \label{eq:weil-pairing}
  h((a,b),(c,d)) = ad-bc.
\end{equation}

Consequently, the above subgroups $K\subset J[2](\Qbar)$ are in 1-to-1
correspondence with lines in $(\O/2\O)^2 \simeq \F_{64}^2$. There are exactly
65 such lines.

Next, we consider the Galois action on~$J[2]$, which is a map
\begin{equation}
  \label{eq:galois}
  \rho: \Gal(\Qbar/\Q) \to \mathrm{Aut}(J[2](\Qbar)).
\end{equation}
Since the RM structure of~$J$ is defined over~$\Q$, the map~$\rho$ actually
factors through $\mathrm{Aut}(\O/2\O)\simeq \GL_2(\F_{64})$. In fact, we have:
\begin{prop}
  \label{prop:rho-factors}
  The map $\rho$ factors through $\SL_2(\F_{64})$. If~$\im(\rho)$ contains the
  whole of~$\SL_2(\F_{64})$, then the 65 maximal isotropic
  subgroups~$K\subset J[2](\Qbar)$ that are stable under~$\O$ form a single
  Galois orbit; their field of definition~$F$ has Galois group isomorphic
  to~$\SL_2(\F_{64})$, and is unramified away from $\{2,271\}$.
\end{prop}

\begin{proof}
  The inclusion $\im(\rho)\subset \SL_2(\F_{64})$ stems from compatibility
  relations between the Galois action and the Weil pairing: by~\cite[Lemma
  4.7]{banaszakImage$ell$adicGalois2006} and the nondegeneracy of the trace
  map, composing the Galois action~$\rho:\Gal(\Qbar/\Q)\to \GL_2(\F_{64})$ with
  the determinant yields the cyclotomic character modulo~$2$, which is trivial.

  If~$\im(\rho)$ contains the whole of~$\SL_2(\F_{64})$, then $\Gal(\Qbar/\Q)$
  acts transitively on $(\O/2\O)$-lines in~$(\O/2\O)^2\simeq J[2](\Qbar)$,
  hence the second part of the statement. Finally,~$J$ has good reduction
  away from~$271$, so the reduction map on~$J[2]$ is injective at all
  primes~$\ell\notin\{2,271\}$ and~$F$ is unramified at such primes~$\ell$.
\end{proof}

The assumption that~$\im(\rho)=\SL_2(\F_{64})$ will be experimentally validated
a posteriori, when checking that the polynomial we construct is irreducible.

In order to compute the field~$F$, we will use numerical computations
over~$\C$, and construct the 65 isogenous p.p.a.v.'s of the form $J/K$. Let us
explain how those can be enumerated.

Let~$\tau\in \Half_6$ be any small period matrix for~$J$,
let~$\Lambda(\tau) = \tau\Z^6 + \Z^6$, and fix an
isomorphism~$\eta: \C^6/\Lambda(\tau) \to J(\C)$. Consider the canonical
symplectic basis~$(b_1,\ldots,b_{12})$ of~$\Lambda(\tau)$, defined by
\begin{equation}
  \label{eq:sp-basis}
  b_1 = \tau
  \begin{pmatrix}
    1 \\ 0 \\ \vdots \\ 0
  \end{pmatrix}, \quad \ldots, \quad b_6 = \tau \begin{pmatrix}
    0 \\ \vdots \\ 0 \\ 1
  \end{pmatrix},
  b_7 = \begin{pmatrix}
    1 \\ 0 \\ \vdots \\ 0
  \end{pmatrix}, \quad \ldots, \quad b_{12} = \begin{pmatrix}
    0 \\ \vdots \\ 0 \\ 1
  \end{pmatrix}
\end{equation}
and let~$M_e\in \GL_{12}(\Z)$ be the matrix of~$e$ in this basis. Further
let~$M_\chi\in \GL_6(\Z)$ be the companion matrix to the characteristic
polynomial of~$e$ as an element of~$\O$.

Since~$J$ is principally polarized, we know that~$\Lambda(\tau)$ is isomorphic
to $(\O\oplus\O^\vee, \mathrm{Tr}\, h)$ as a symplectic $\O$-module,
where~$\O^\vee$ denotes the inverse different
of~$\O$~\cite[Prop.~9.2.3]{birkenhakeComplexAbelianVarieties2004}. Therefore,
there exists a matrix~$B\in \Sp_{12}(\Z)$ such that
\begin{equation}
  \label{eq:block-diag}
  BM_e B^{-1} =
  \begin{pmatrix}
    M_\chi & 0 \\ 0 & M_\chi^{T}
  \end{pmatrix}.
\end{equation}

This being set, recall that for any $g\geq 1$, the group~$\Sp_{2g}(\Q)$ fits in
an exact sequence
\begin{equation}
  1\to \Sp_{2g}(\Q) \to \GSp_{2g}(\Q) \overset{\mu}{\to} \Q^\times \to 1,
\end{equation}
where~$\mu$ denotes the ``multiplier'' character. We write
\begin{equation}
  \label{eq:gsp+}
  \GSp_{2g}(\Q)^+ = \{m\in \GSp_{2g}(\Q): \mu(m) >0\}
\end{equation}
and note that the action of~$\Sp_{2g}(\Z)$ on~$\Half_g$, as
in~\eqref{eq:action}, extends to~$\GSp_{2g}(\Q)^+$ using the same formula.

We then define $S(2)\subset \GSp_{2g}(\Q)^+$ to be the following set of 65
matrices with integral coefficients: $S(2)$ contains the matrices $B^{T} m' (B^{-1})^T$
where
\begin{equation}
  \label{eq:S2}
  m' =
  \begin{pmatrix}
    2 I_g & 0 \\ 0 & I_g
  \end{pmatrix} \quad \text{or} \quad m' =
  \begin{pmatrix}
    I_g & Q(M_\chi^T) \\ 0 & 2 I_g
  \end{pmatrix}
\end{equation}
where~$Q$ runs through polynomials with coefficients in ${0,1}$ of degree at
most~$5$.

\begin{lem}
  \label{lem:isog-enum}
  With the above notations, the matrices $m\tau\in \Half_6$ where~$m\in S(2)$
  are small period matrices for the quotients $J/K$,
  where~$K\subset J[2](\Qbar)$ runs through maximal isotropic subgroups stable
  under~$\iota(\O)$.
\end{lem}

\begin{proof}
  The matrix~$\tau'=(B^{-1})^T\tau\in \Half_6$ is another small period matrix for~$J$,
  and the complex torus $\C^6/\Lambda(\tau')$ inherits a compatible
  action~$\iota'$ of~$\O$. We check that:
  \begin{enumerate}
  \item The matrix of~$\iota'(e)$ on the canonical
    basis~$(b_1',\ldots,b_{12}')$ of~$\Lambda(\tau')$, defined as
    in~\eqref{eq:sp-basis}, is $B M_e B^{-1}$.
  \item For each matrix~$m'$ as in~\eqref{eq:S2}, we have
    $B^{T}m'(B^{-1})^T\tau = B^Tm\tau'$, which is $\Sp_{12}(\Z)$-equivalent
    to~$m\tau'$.
  \end{enumerate}

  Replacing~$\tau$ by~$\tau'$, we can therefore assume
  that~$B=I_{12}$. Let~$(\overline{b_1},\overline{b_7})$ denote the vectors
  $(b_1/2,b_7/2)$ seen as elements of $\C^6/\Lambda(\tau) \simeq J(\C)$.
  Then~$(\overline{b_1},\overline{b_7})$ is a symplectic basis of~$J[2](\Qbar)$
  as an $\O/2\O$-module. The 65 lines in~$J[2](\Qbar)$ can be explicitly
  listed: they are the $\langle \overline{b_1} + \alpha \overline{b_7} \rangle$
  for $\alpha\in \O/2\O$, and $\langle \overline{b_7} \rangle$. In the first
  case, there exists exactly one polynomial~$Q$ as above such
  that~$\alpha = Q(e)$ modulo $2\O$. Then the kernel of the isogeny
  \begin{displaymath}
    \C^6/\Lambda(\tau) \to \C^6/ \Lambda((\tau + Q(M_\chi^T))/2)
  \end{displaymath}
  is exactly $\langle \overline{b_1} + \alpha \overline{b_7} \rangle$. One can
  also check that the last line $\langle \overline{b_7} \rangle$ corresponds to
  the action of the matrix $
  \paren[\big]{\begin{smallmatrix}
    2 I_g & 0\\ 0 & I_g
  \end{smallmatrix}}$ after applying the transformation $\tau\mapsto -\tau^{-1}$.
\end{proof}

Given \cref{lem:isog-enum}, the values of Siegel modular forms at the points
$m\tau\in \Half_6$, properly rescaled, will yield algebraic invariants of the
p.p.a.v.'s $A/K$. The following proposition makes this statement more precise.

\begin{prop}
  \label{prop:alg-invariants}
  With the above notation, let~$f$ be a Siegel modular form on~$\Half_6$ of
  some weight~$k\geq 0$ with integral Fourier coefficients. Then the quantities
  \begin{displaymath}
    q(m):=\det(\gamma\tau+\delta)^{-k} f(m\tau)/f(\tau),
  \end{displaymath}
  where~$m$ runs through $S(2)$ and~$\gamma,\delta$ denote the lower
  $g\times g$ blocks of~$m$, are algebraic numbers. Moreover, the map which to
  a maximal isotropic subgroup $K\subset A(\tau_0)[2]$ associates $q(m)$ for
  the corresponding matrix~$m$ is $\Gal(\Qbar/\Q)$-equivariant.
\end{prop}

\begin{proof}
  This is a consequence of the algebraic interpretation of modular forms with
  integral Fourier expansions as algebraic sections of the Hodge bundle on the
  moduli space~$\mathcal{A}_g$: see for
  instance~\cite[Thm.~4.11]{vanbommelComputingIsogenyClasses2024} in the~$g=2$
  case.
\end{proof}

By \cref{prop:alg-invariants}, if the quantities $q(m)$ for $m\in S(2)$ are
distinct as complex numbers, then the polynomial
\begin{equation}
  \label{eq:P}
  P(X) = \prod_{m\in S(2)}(X - q(m))
\end{equation}
is indeed a defining polynomial for the number field~$F$ from
\cref{prop:rho-factors}.

In practice, it is advantageous to consider a Siegel modular form~$f$ whose
weight is as small as possible, so we take~$f$ to be the Eisenstein
series~$E_4$: we have for each~$\tau\in \Half_6$
\begin{equation}
  E_4(\tau) = \sum_{a,b\in \{0,1\}^g} \theta_{a,b}^8(0,\tau).
\end{equation}

The remaining challenge is then to compute~$P$ to a sufficiently high precision
so that we can recognize its coefficients as rational numbers (at least
heuristically), then to certify a posteriori that $P\in \Q[X]$ has the correct
Galois group.

\subsection{Numerical computations}
\label{subsec:galois-compute}

Using Magma~\cite{bosmaMagmaAlgebraSystem1997}, we compute a small period
matrix~$\tau$ of $J$ to 160 decimal digits of precision (which is not reduced
as $\det\im(\tau)\approx 0.01$), then the approximate endomorphism ring of the
complex abelian
variety~$\C^6/\Lambda(\tau)$~\cite{costaRigorousComputationEndomorphism2019},
as follows:
\begin{center}
\begin{verbatim}
    A<x,y> := PolynomialRing(Rationals(), 2);
    f := x^6 + (-6*y+1)*x^5
        + (-y^3+17*y^2-3*y-6)*x^4
        + (4*y^4-29*y^3-8*y^2+49*y-13)*x^3
        + (-7*y^5+29*y^4+30*y^3-78*y^2-6*y+12)*x^2
        + (4*y^6-2*y^5-65*y^4+93*y^3-8*y^2+5*y-5)*x
        + (-8*y^6+33*y^5-29*y^4-9*y^3+7*y^2-3*y+1);
    X := RiemannSurface(f : Precision:=160);
    E := EndomorphismRing(SmallPeriodMatrix(X)); E.1;
\end{verbatim}
\end{center}

As expected, $\C^6/\Lambda(\tau)$ appears to have real multiplication by~$\O$,
and the above code produces (a heuristic candidate for) the matrix~$M_e$,
namely \setcounter{MaxMatrixCols}{12}
\begin{equation}
  \label{eq:Me}
  M_e = {
  \scalebox{0.8}{$
  \begin{pmatrix}
     0& 1& 1& 1& 1& 0& 0& 1& 1& 0& 1 &-1 \\
    -1& 0& 0 &-1 &-1 &-1 &-1& 0& 0 &-1 &-1& 0\\
    0& 1& 2& 2& 0& 1 &-1& 0& 0& 1& 0& 0\\
    0& 0& 1& 0& 0& 0& 0& 1 &-1& 0& 0& 0\\
    1& 1& 0& 1& 0& 2 &-1& 1& 0& 0& 0 &-1\\
    0& 0& 0& 0& 1& 0& 1& 0& 0& 0& 1& 0\\
    0 &-1& 0& 1& 0& 1& 0 &-1& 0& 0& 1& 0\\
    1& 0 &-1& 1& 1& 1& 1& 0& 1& 0& 1& 0\\
    0& 1& 0& 1& 1& 0& 1& 0& 2& 1& 0& 0\\
    -1 &-1 &-1& 0& 1 &-1& 1 &-1& 2& 0& 1& 0\\
    0 &-1 &-1 &-1& 0 &-1& 1 &-1& 0& 0& 0& 1\\
    -1 &-1& 0& 1& 1& 0& 0 &-1& 1& 0& 2& 0
  \end{pmatrix}$}.}
\end{equation}
The corresponding unit $e\in \O^\times$ has characteristic polynomial
$\chi=X^6-2X^5-5X^4+12X^3-3X^2-3X+1$. These initial computations finish in a
few minutes.

To find the matrix $B\in \Sp_{12}(\Z)$, we let the 1st and 7th columns of~$B$
to be the 1st and 7th canonical basis vectors; the rest of the matrix is then
fixed by~\eqref{eq:block-diag}, and we check a posteriori that~$B$ is
symplectic. (In other examples below, we sometimes had to consider the 2nd and
8th canonical basis vectors instead.)

Next, we collect the values $E_4(\tau)$ and $E_4(m\tau)$ as complex numbers at
the current working precision, where~$m\in S(2)$. To this end, we combine the
reduction algorithm~\ref{algo:siegel-red} (adapted to inexact arithmetic as
explained in~§\ref{subsec:inexact-reduction}) with \cref{algo:ql} to
efficiently evaluate theta constants at the reduced matrices, then apply the
transformation transformula as explained
in~§\ref{subsec:transf}--\ref{subsec:decomp}.

As predicted by \cref{prop:alg-invariants}, we observe that the complex
numbers~$q(m)$ are either real or appear by pairs of complex conjugates, which
is a nice sanity check. However, the 160 decimal digits of precision that Magma
provided on~$\tau$ are not sufficient to recognize the coefficients of the
polynomial~$P$ from~\eqref{eq:P} as rational numbers. Allowing Magma more time
(about an hour), we can obtain $\tau$ to 640 binary digits of precision, which
is still not sufficient. We discontinued the analogous computation at 1280
decimal digits after 48 hours of computation time.

A more efficient approach to raise the initial precision on~$\tau$ is to use
Newton iterations. First, the constraint that $\C^6/\Lambda(\tau)$ has RM
forces~$\tau$ to lie on a 6-dimensional subvariety of~$\Half_6$: writing
\begin{equation}
  M_e =
  \begin{pmatrix}
    \alpha & \beta \\ \gamma & \alpha^T
  \end{pmatrix}
\end{equation}
where~$\beta$ and~$\gamma$ are antisymmetric, we have
\begin{equation}
  \tau \beta \tau + \alpha^T\tau - \tau\alpha - \gamma = 0.
\end{equation}

We obtain six more exact conditions on~$\tau$ by looking at the values of
certain Siegel modular forms at~$\tau$. For $1\leq k\leq 7$, write
\begin{equation}
  g_k = \sum_{a,b\in \{0,1\}^g} \theta_{a,b}^{8k}.
\end{equation}
Using the value of~$\tau$ at 160 decimal digits of precision, we numerically
observe that for some~$\lambda\in \C^\times$, the quantities
$\lambda^k g_k(\tau)$ for $1\leq k\leq 7$ are the rational numbers
\begin{center}
  \fontsize{7}{8} \selectfont \vspace{-3mm}
  \begin{align*}
    & 77345394568888,\\
    &241286401644230849105066641, \\
    &-4212712293961190441265125813075057531065 / 8, \\
    &9828793569897300171012751081827064387670943610967837217 / 64, \\
    &-914542704676801663371882511577830901380061397080349036660281455526089/512\\
    &-18024118779569029845130097817939912940714254335855423376619982839818431780177166671/4096
    \\
    &2035549115746672221166143046627907502907564649191589998224193743271881921775786459188405565827623/32768.
  \end{align*}
\end{center}

The additional six conditions we use are then that $g_k(\tau)/E_4(\tau)^k$, for
$1\leq k\leq 7$, is a known rational number. We observe experimentally that
these conditions are independent of the RM condition at the first order. We can
then run the Newton iteration, evaluating the necessary first-order derivatives
of the functions $g_k(\tau)/E_4(\tau)^k$ using finite differences. (Computing
these derivatives using \cref{thm:deriv} would be another possibility.) This
technique allows us to (heuristically) increase the precision on~$\tau$ to any
desired value~$N$ in quasi-linear time in~$N$.

From the above value of~$\lambda g_1(\tau) = \lambda E_4(\tau)$, it is
beneficial to replace the quantities $q(m)$ from \cref{prop:alg-invariants} by
\begin{equation}
  \label{eq:qprime}
  q'(m) = 77345394568888 \det(\gamma\tau + \delta)^{-4} E_4(m\tau)/E_4(\tau),
\end{equation}
with the hope that the algebraic numbers $q'(m)$ will now be algebraic integers
divided by small powers of~$2$:
see~\cite[Thm.~4.11]{vanbommelComputingIsogenyClasses2024} for a rigorous
justification to this technique in the~$g=2$ setting.

Forming now the polynomial $\prod_{m\in S(2)} (X - q'(m))$, slightly more than
800 decimal digits of precision are sufficient to recognize it in~$\Q[X]$. It
turns out that
\begin{equation}
  Q = \prod_{m\in S(2)} (X - 2^{21} q'(m))
\end{equation}
has integral coefficients. At this point, we are confident that the monic
polynomial~$Q\in \Z[X]$ defines a number field with Galois group
$\SL_2(\F_{64})$.

\subsection{Algebraic post-processing}
\label{subsec:galois-check}

The discriminant of~$Q$ has high valuations at $2,3,5$, and $271$ (as could
probably be expected), but no other small primes. It is a square as expected.
Moreover, factoring~$Q$ modulo primes up to 1000 uncovered all the cycle
structures in $\SL_2(\F_{64})$, except the trivial one which should only arise
once in $|\SL_2(\F_{64})| = 262080$ times, and no others.

Assuming that the discriminant of~$F= \Q[X]/(Q) $ has no prime factors larger
than~$10^4$, the \texttt{nfdisc} algorithm implemented in
PARI/GP~\cite{theparigroupPariGPVersion2025} computes that
$\mathrm{disc}(F) = 2^{172} 271^{32}$, so~$F $ appears to be unramified away
from~$2$ and~$271$, as expected, with root discriminant approximately $60.44$.

However, the polynomial~$Q$ has huge coefficients (around 800 decimal
digits). It is therefore desirable to find other defining polynomials for~$F$
with smaller coefficients.

The dedicated PARI/GP algorithm \texttt{polredbest([Q,[2,271]])}, which is
polynomial time, did not appear to terminate within a reasonable time, so we
attempted a hands-on reduction instead. We computed a (conjectural) $\Z$-basis
of the integer ring~$\Z_F$ using PARI/GP's \texttt{nfbasis([Q,[2,271]])}, then
attempted to look for a more reduced basis by hand according to the inner
product
\begin{equation}
  \begin{matrix}
    \Q[X]/(Q) \times \Q[X]/(Q)&\to &\R\\
    (a,b) &\mapsto &\langle a,b\rangle = \sum_{Q(x)=0} \re\paren[\big]{a(x) \overline{b(x)}}.
  \end{matrix}
\end{equation}
By iteratively LLL-reducing the current $\Z$-basis of~$\Z_F$, and replacing
elements of that basis by any elements of smaller norm that we can find, we
found after about a day of computation another basis containing an
element~$a\in \Z_F\backslash\Z$ with $\sqrt{\langle a,a\rangle}\approx
261.86$.

A quick computation suggests that the norm of~$a$ is probably close to optimal:
assuming that~$\Z_F/\Z$ is a ``random'' rank~64 lattice of discriminant
$\mathrm{disc}(F)/65$ (taking the orthogonal projection parallel to
$\Q\cdot 1$), the expected number of nonzero vectors of norm at most~$N$ is the
volume of the norm-$N$ ball times the density of the lattice, namely
\begin{displaymath}
  \frac{\pi^{32} N^{32}}{32!} \cdot \paren[\Big]{\frac{65}{\mathrm{disc}(F)}}^{1/2}.
\end{displaymath}
This number is approximately~$1$ for $N=245.78$ and $7.1$ for $N=261.3$, which
is the adjusted norm of~$f$ in~$\Z_F/\Z$ as $\mathrm{Tr}_{F/\Q}(f) = -6$.

At this point, we used PARI/GP's \texttt{forqfvec} function to exhaustively
look for elements in~$\Z_F/\Z$ of smaller norms. This revealed the shortest
element~$b\in \Z_F\backslash\Z$ up to sign, with
$\sqrt{\langle b,b\rangle}\approx 257.16$. Its minimal polynomial is displayed
on \cref{fig:minpoly-1}. Overall, this algebraic post-processing was much more
time-consuming that the numerical computations themselves.

\begin{figure}
  \footnotesize
  \begin{align*}
    &X^{65} - 15X^{64} + 88X^{63} - 200X^{62} - 308X^{61} + 3328X^{60} - 8496X^{59} + 5668X^{58}\\ & + 33468X^{57} - 161480X^{56} + 329024X^{55} - 193528X^{54}- 1210660X^{53}\\ & + 5273620X^{52}- 7474904X^{51} + 6085740X^{50} + 15228888X^{49} - 141791164X^{48}\\ & + 173687420X^{47} + 5844532X^{46} + 36445836X^{45} + 2262870792X^{44} \\ & - 4257739384X^{43} - 2809126072X^{42} + 1960413324X^{41} - 14213264596X^{40}\\ & + 62112208680X^{39} + 25413459352X^{38} - 109254363512X^{37} + 4207088448X^{36}\\ & - 450477802432X^{35} + 398072641376X^{34} + 1140423619492X^{33} + 337558277636X^{32}\\ & + 938145285752X^{31} - 9690658271544X^{30} - 540230410376X^{29} + 13729474803440X^{28} \\ &+ 5611016774600X^{27} + 10221468157040X^{26} + 969150963792X^{25} - 100436840208656X^{24}\\ & - 47742710156192X^{23} + 247107904006072X^{22} + 151085073423296X^{21} - 425485805828968X^{20}\\ & - 169092107877248X^{19} + 441986131668336X^{18} + 145585844071664X^{17} - 359752332255168X^{16}\\ & - 1164356212896X^{15} + 79988297469504X^{14} + 86882389581680X^{13} - 78853608846528X^{12}\\ & - 3133125851056X^{11} + 4054780310592X^{10} + 11368690383504X^9 - 3164433646912X^8\\ & - 2488615997120X^7 + 326796982336X^6 + 758124184992X^5 - 233430779520X^4\\ & - 20336656288X^3 + 6234453840X^2 + 2182704216X + 173220328.
  \end{align*}
  \caption{The minimal polynomial of~$b\in \Z_F$.}
  \label{fig:minpoly-1}
\end{figure}

Recall that any number field admits a unique \texttt{polredabs} defining
polynomial, in PARI/GP terminology. Our expectations can then be summarized as
follows.

\begin{conj}
  \label{conj:1}
  Let~$R\in \Z[X]$ be the polynomial displayed on \cref{fig:minpoly-1}.  Then
  the number field $F = \Q[X]/(R)$ has Galois group $\SL_2(\F_{64})$ over~$\Q$,
  and $R$ is the \textup{\texttt{polredabs}} defining polynomial for~$F$.
\end{conj}

\subsection{An alternative approach using modular forms}
\label{subsec:galois-mf}

Let~$N\geq 1$, and let~$f$ be a classical Hecke eigenform for the modular
group~$\Gamma_0(N)$ whose $q$-expansion coefficients generate a totally real
degree~6 number field $K_f$ in which~$2$ is an inert prime. The Galois
representation attached to~$f$ arises from the $\ell$-torsion subgroups of an
abelian variety~$A_f$ of dimension~6, defined over~$\Q$, with RM by~$K_f$; this
abelian variety is an isogeny factor of the Jacobian of~$X_0(N)$. If~$K_f$ has
narrow class number one, then we additionally know that~$A_f$ will admit a
principal polarization, but it will not be a Jacobian in general.

Computing a small period matrix for~$A_f$ allows us to repeat the above
computations and find other number fields with Galois group
$\SL_2(\F_{64})$. These number fields will be unramified away from~$2$ and~$N$,
so if~$N$ is sufficiently smooth, we can hope to uncover such number fields
with smaller root discriminants.

A first example is provided by the modular form~$f$ with LMFDB label
\texttt{729.2.a.a} for~$N=729=3^6$: see
\url{https://www.lmfdb.org/ModularForm/GL2/Q/holomorphic/729/2/a/a}. A period
matrix of~$A_f$ to 160 decimal digits can be computed with the following Magma
code (see also the example \texttt{ModSym\_BSD} at
\url{https://magma.maths.usyd.edu.au/magma/handbook/text/1681}):
\begin{center}
\begin{verbatim}
    SetDefaultRealField(RealField(160));
    M729 := ModularSymbols(729,2);
    H1 := CuspidalSubspace(M729);
    H1N := NewSubspace(H1);
    D := NewformDecomposition(H1N);
    P4 := Periods(D[4],40000);
\end{verbatim}
\end{center}

Here 40000 represents the length of the $q$-expansion used to compute the
periods. We determine that this leads to rougly 160 decimal digits of precision
on~$\tau$ by comparing it to the output with slightly different $q$-expansion
lengths. Even though~$K_f$ has narrow class number two, we are lucky in
that~$A_f$ is principally polarized. Eventually, we found the first
polynomial~$R$ displayed on \cref{fig:minpoly-2}.

\begin{figure}
  \footnotesize
  \begin{align*}
    &X^{65} + 26X^{64} + 304X^{63} + 2040X^{62} + 7860X^{61} + 9348X^{60} - 83220X^{59}- 603120X^{58}\\ & - 2022540X^{57}
      - 2786492X^{56} + 8832476X^{55} + 69424348X^{54} + 234301368X^{53} + 447117000X^{52}\\ &  + 104147400X^{51}
      - 2663155140X^{50} - 10911635112X^{49} - 24904042368X^{48} - 30178328804X^{47}\\ &  + 20625545204X^{46}
      + 212027151148X^{45} + 608209898664X^{44} + 1076276880816X^{43} \\ &+ 998018978988X^{42} - 928741848348X^{41}
      - 6321394370616X^{40} - 15581033533848X^{39} \\ &- 25061758639460X^{38} - 24165520906888X^{37} + 4474325817016X^{36}
      + 77604303820632X^{35}\\ & + 195100630401684X^{34} + 320232919053174X^{33} + 370512314728044X^{32}
      + 235833341774808X^{31}\\ & - 170943655879416X^{30} - 849975532111396X^{29} - 1675500866042732X^{28}
      - 2415604343375500X^{27}\\ & - 2813981074184496X^{26} - 2699969415003636X^{25} - 2071537338636012X^{24}\\ &
      - 1104272608855308X^{23} - 77673906252732X^{22} + 742077141516768X^{21} + 1199322337281552X^{20}\\ &
      + 1281393781407432X^{19} + 1085906215867572X^{18} + 755148889980456X^{17}\\ & + 414762427981128X^{16}
      + 140486142261300X^{15} - 42200143569612X^{14} - 138661457629356X^{13}\\ & - 166634845140600X^{12} - 148605734150792X^{11}
      - 108778806259012X^{10} - 67496716054964X^9\\ & - 35988490581144X^8 - 16502695848624X^7 - 6429930262332X^6
      - 2095688217504X^5\\ & - 572105153040X^4 - 133224111336X^3 - 25618209524X^2 - 3425350051X - 218511686.
  \end{align*}
  \begin{align*}
    &X^{65} + 11X^{64} + 74X^{63} + 336X^{62} + 1260X^{61} + 3780X^{60} + 11368X^{59} + 31560X^{58} + 105576X^{57} \\ &+ 306756X^{56} + 964964X^{55} + 2348500X^{54} + 5985896X^{53} + 11435536X^{52} + 25056696X^{51} \\ &+ 41189064X^{50} + 89996088X^{49} + 142432780X^{48} + 306010852X^{47} + 401970604X^{46} \\ & + 689239600X^{45}+ 494052556X^{44} + 341477852X^{43} - 1184544128X^{42} - 2605183644X^{41} \\ &- 8005808132X^{40} - 10815089684X^{39} - 26030121348X^{38} - 28591561440X^{37} - 53984664248X^{36}\\ & - 3696384276X^{35} - 55882169700X^{34} + 123027466366X^{33} + 163225526310X^{32}\\ & + 313209342180X^{31} + 507686797168X^{30} + 1028364168852X^{29} + 1110373975024X^{28}\\ & + 1347215828700X^{27} + 1618481242140X^{26} + 1408587577992X^{25} + 1043458478944X^{24}\\ & + 645568403928X^{23} + 187639420444X^{22} - 281714530048X^{21} + 6537077624X^{20}\\ & + 129995452328X^{19} + 425487780088X^{18} + 605636646552X^{17} + 662893432068X^{16}\\ & + 506213438024X^{15} + 450066000860X^{14} + 292069972516X^{13} + 204757278152X^{12}\\ & + 118965446376X^{11} + 84125543432X^{10} + 39994979992X^9 + 30118190040X^8\\ & + 13459759188X^7 + 8521507848X^6 + 3658616976X^5 + 1735208832X^4\\ & + 713279056X^3 + 227902216X^2 + 73426162X + 14161670.
  \end{align*}
  \caption{Two polynomials with conjectural Galois group $\SL_2(\F_{64})$.}
  \label{fig:minpoly-2}
\end{figure}

Another candidate is the modular form~$f$ with LMFDB label \texttt{343.2.a.d},
which has the smallest possible level. Here $K_f$ again has narrow class number
two, but the corresponding~$A_f$ turns out not to be principally
polarized. Adapting our strategy to this example would therefore require more
work.

Instead, our second example in this subsection is provided by the modular
form~$f$ with LMFDB label \texttt{1372.2.a.d}: see
\url{https://www.lmfdb.org/ModularForm/GL2/Q/holomorphic/1372/2/a/d},
with~$N = 1372 = 2^2 7^3$. It led to the second polynomial displayed on
\cref{fig:minpoly-2}.

\begin{conj}
  \label{conj:2}
  Let~$R_1,R_2\in \Z[X]$ be the two polynomials displayed on
  \cref{fig:minpoly-2}.  Then the number fields $F_1 = \Q[X]/(R_1)$ and
  $F_1 = \Q[X]/(R_2)$ have Galois group $\SL_2(\F_{64})$ over~$\Q$, and
  $-R_i(-X)$ for $i\in \{1,2\}$ are their \textup{\texttt{polredabs}} defining
  polynomials. Their discriminants are $2^{126}3^{154}$ and $2^{126}7^{86}$,
  with associated root discriminants less than~$51.76$ and~$50.32$
  respectively.
\end{conj}

We leave the verification of \cref{conj:1,conj:2}, as well as the search for
explicit polynomials defining $\SL_2(\F_{64})$ number fields with smaller
discriminants, as topics for future work.

\clearpage

\printbibliography

\end{document}